\documentclass{article}
\usepackage{ifpdf}
\usepackage[utf8]{inputenc}
\usepackage{xcolor}
\usepackage{amsthm}
\usepackage{amsmath}
\usepackage{mathtools}
\usepackage{hyperref}
\usepackage{hyperref}
\hypersetup{
    colorlinks=true,
    linkcolor=blue,
    filecolor=magenta,      
    urlcolor=cyan,
}
\usepackage[export]{adjustbox}
\usepackage{braket,amsfonts}
\usepackage{amsmath}

\DeclareMathOperator*{\argmin}{arg\,min}
\newtheorem{lemma}{Lemma}
\newtheorem{remark}{Remark}
\newtheorem{corollary}{Corollary}

\usepackage{array}

\usepackage[caption=false]{subfig}

\usepackage{algorithm,algpseudocode}
\algnewcommand{\Inputs}[1]{%
  \State \textbf{Inputs:}
  \Statex \hspace*{\algorithmicindent}\parbox[t]{.8\linewidth}{\raggedright #1}
}
\algnewcommand{\Outputs}[1]{%
  \State \textbf{Outputs:}
  \Statex \hspace*{\algorithmicindent}\parbox[t]{.8\linewidth}{\raggedright #1}
}
\algnewcommand{\Initialize}[1]{%
  \State \textbf{Initialize:}
  \Statex \hspace*{\algorithmicindent}\parbox[t]{.8\linewidth}{\raggedright #1}
}
\usepackage{graphicx,epstopdf}

\usepackage{amsopn}

\usepackage{xspace}
\usepackage{bold-extra}
\usepackage[most]{tcolorbox}

\colorlet{texcscolor}{blue!50!black}
\colorlet{texemcolor}{red!70!black}
\colorlet{texpreamble}{red!70!black}
\colorlet{codebackground}{black!25!white!25}
\usepackage{bm}
\usepackage{acronym}

\acrodef{FTFC}[FTFC]{\emph{Fundamental Theorem of Fractional Calculus}}
\acrodef{DMD}[DMD]{\emph{Dynamic Mode Decomposition}}
\acrodef{SDMD}[S-DMD]{\emph{Symmetric DMD}}
\acrodef{SVD}[SVD]{\emph{Singular Vector Decomposition}}
\acrodef{TV}[TV]{\emph{Total Variation}}
\acrodef{DFT}[DFT]{\emph{Discrete Fourier Transform}}
\acrodef{CFT}[CFT]{\emph{Continuous Fourier Transform}}
\acrodef{POD}[POD]{\emph{Proper Orthogonal Decomposition}}
\acrodef{PDE}[PDE]{\emph{Partial Differential Equation}}
\acrodef{PDEs}[PDEs]{\emph{Partial Differential Equations}}
\acrodef{TV}[TV]{\emph{Total Variation}}
\acrodef{OrthoNS}[OrthoNS]{\emph{Orthogonal Nonlinear Spectral decomposition}}

\lstdefinestyle{siamlatex}{%
  style=tcblatex,
  texcsstyle=*\color{texcscolor},
  texcsstyle=[2]\color{texemcolor},
  keywordstyle=[2]\color{texemcolor},
  moretexcs={cref,Cref,maketitle,mathcal,text,headers,email,url},
}

\tcbset{%
  colframe=black!75!white!75,
  coltitle=white,
  colback=codebackground, 
  colbacklower=white, 
  fonttitle=\bfseries,
  arc=0pt,outer arc=0pt,
  top=1pt,bottom=1pt,left=1mm,right=1mm,middle=1mm,boxsep=1mm,
  leftrule=0.3mm,rightrule=0.3mm,toprule=0.3mm,bottomrule=0.3mm,
  listing options={style=siamlatex}
}

\newtcblisting[use counter=example]{example}[2][]{%
  title={Example~\thetcbcounter: #2},#1}

\newtcbinputlisting[use counter=example]{\examplefile}[3][]{%
  title={Example~\thetcbcounter: #2},listing file={#3},#1}

\DeclareTotalTCBox{\code}{ v O{} }
{ 
  fontupper=\ttfamily\color{black},
  nobeforeafter,
  tcbox raise base,
  colback=codebackground,colframe=white,
  top=0pt,bottom=0pt,left=0mm,right=0mm,
  leftrule=0pt,rightrule=0pt,toprule=0mm,bottomrule=0mm,
  boxsep=0.5mm,
  #2}{#1}

\patchcmd\newpage{\vfil}{}{}{}
\providecommand{\keywords}[1]
{
  \small	
  \textbf{\textit{Keywords---}} #1
}
\newtheorem{theorem}{Theorem}

\newtheorem{definition}{Definition}

\title{Modes of Homogeneous Gradient Flows\thanks{Ido would like to thank Prof. Andrea Bertozzi for the opportunity of studying in mathematics science department, UCLA in general and for helpful conversations, related to this work, in particular. We thank Prof. Ronen Talmon for stimulating discussions. And we also would like to thank Shachar Praisler for his helpfull advice.
{\bf{Funding:}} This work was supported by the European Union’s Horizon 2020 research and innovation programme under the Marie Sk{\l}odowska-Curie grant agreement No. 777826 (NoMADS). GG acknowledges support by the Israel Science Foundation (Grant No.  534/19) and by the Ollendorff Minerva Center. }}

\author{Ido Cohen\thanks{Electrical Engineering Department at the Technion -- Israel Institute of Technology (\href{mailto:idoc@campus.technion.ac.il}{idoc@campus.technion.ac.il},\href{mailto:pavel@ee.technion.ac.il}{pavel@ee.technion.ac.il},\href{mailto:guy.gilboa@ee.technion.ac.il}{guy.gilboa@ee.technion.ac.il}).}
\and Omri Azencot\thanks{Department of Mathematics, University of California Los Angeles (\href{mailto:azencot@math.ucla.edu}{azencot@math.ucla.edu}).} 
\and Pavel Lifshits \footnotemark[2]
\and Guy Gilboa \footnotemark[2]}

\ifpdf
\hypersetup{pdftitle={Mode decomposition for Homogeneous Symmetric Operators}}
\fi

\usepackage{mathtools}
\DeclarePairedDelimiterX{\inp}[2]{\langle}{\rangle}{#1, #2} 
\usepackage{physics} 

\usepackage{amssymb} 
\usepackage{mathtools}
\usepackage{longtable}

\date{\today}
\begin{document}
\maketitle 

\begin{abstract}
Finding latent structures in data is drawing increasing attention in diverse fields such as image and signal processing, fluid dynamics, and machine learning. 
In this work we examine the problem of finding the main modes of gradient flows. Gradient descent is a fundamental process in optimization where its stochastic version is prominent in training of neural networks. 
Here our aim is to establish a consistent theory for gradient flows $\psi_t = P(\psi)$, where $P$ is a nonlinear homogeneous operator. Our proposed framework stems from analytic solutions of homogeneous flows, previously formalized by Cohen-Gilboa, where the initial condition $\psi_0$ admits the nonlinear eigenvalue problem $P(\psi_0)=\lambda \psi_0 $.
We first present an analytic solution for \ac{DMD} in such cases.
We show an inherent flaw of \ac{DMD}, which is unable to recover the essential dynamics of the flow. It is evident that \ac{DMD} is best suited for homogeneous flows of degree one. We propose an adaptive time sampling scheme and show its dynamics are analogue to homogeneous flows of degree one with a fixed step size. Moreover, we adapt \ac{DMD} to yield a real spectrum, using symmetric matrices.
Our analytic solution of the proposed scheme recovers the dynamics perfectly and yields zero error. We then proceed to show that in the general case the  orthogonal modes $\{ \phi_i \}$ are approximately nonlinear eigenfunctions $P(\phi_i) \approx\lambda_i \phi_i $.  We formulate Orthogonal Nonlinear Spectral decomposition (\emph{OrthoNS}), which recovers the essential latent structures of the gradient descent process. Definitions for spectrum and filtering are given, and a Parseval-type identity is shown. Experimental results on images, show the resemblance to direct computations of nonlinear sepctral decomposition. A significant speedup (by about two orders of magnitude) is achieved for this application using the proposed method.
\end{abstract}


\section{Introduction}
Finding latent structures in data is a fundamental task in diverse fields. Some canonical examples are wavelets and dictionaries in image and signal processing \cite{ricker1953wavelet, elad2006image, gurevich2008finite}, dynamic modes in fluid dynamics analysis \cite{jovanovic2014sparsity, ohmichi2017preconditioned}, and dimensionality reduction and invariant representations in machine learning \cite{papyan2018theoretical, ng2002spectral, kaliroff2019self}. Understanding the latent structures allows to better model and to simplify the problem at hand, facilitating solutions for broad applications such as denoising, prediction, and classification \cite{shaham2018spectralnet}. These structures are formulated differently in different disciplines.  For example, in image processing, the structures can be formed via repetitive patches in different scales \cite{shaham2019singan}, while in signal processing they can be a sum of audio frequencies, or of nonlinear eigenfunctions \cite{burger2016spectral, biton2019optoacoustic}. In fluid dynamics the structures are represented as a sum of modes \cite{schmid2010dynamic}, and in machine learning they might be based on the recurrence of words \cite{kuang2017crime}. Despite this diversity, different techniques from different disciplines typically share similar fundamental principles.

Gradient descent flow is a central process in control \cite{agarwal2019online}
and in machine learning \cite{bottou2018optimization}, where it is common to solve optimization problems. Thus, analysing the gradient flow process draws attention in these areas and plays an important role, particularly when the cost function is non-convex 
\cite{osher2018laplacian, gradu2020non}, or when a model for a dynamical system is investigated \cite{arora2018towards}.

In this work, we propose a method to analyze latent structures of certain common gradient decent flows by using \acf{DMD}. \ac{DMD} is often used today  in fluid dynamics for finding the main modes of a dynamical system. \ac{DMD} is an effective tool for analyzing nonlinear flows \cite{askham2018variable, leroux2016dynamic, gueniat2015dynamic}. It is an approximation of the linear infinite-dimensional Koopman operator \cite{schmid2010dynamic,koopman1931hamiltonian,mezic2005spectral}.
We focus on a gradient flow of a homogeneous functional $R$, 
\begin{equation*}
    \psi_t=P(\psi),\qquad P=-\partial R_\psi,\quad \psi(t=0)=f,
\end{equation*}
where $P$ is a homogeneous operator (typically with order of homogeneity in the range $[0,1]$). When a norm or a semi-norm is minimized, in its standard or quadratic form, we obtain such flows.
As shown in \cite{cohen2020Introducing}, the solution of this equation reaches its steady state in finite time (for order strictly less than one). Moreover, the solution is separable (in time and space) if $f$ is a (nonlinear) eigenfunction of $P$, i.e. $f$ solves the nonlinear eigenvalue problem $P(f)=\lambda \cdot f$.
Precisely, the solution is a multiplication between the initial condition $f$ and a time dependant function $a(t)$: $\psi(t)=a(t)\cdot f$, where $a(t)$ has a closed form solution, which depends on the degree of the homogeneity of $P$ and on the eigenvalue $\lambda$. 

With the purpose of better understanding \ac{DMD} of homogeneous flows, we examine the analytic solutions of such flows. We present a closed form solution of \ac{DMD} in these cases and discover an inherent flaw. Specifically, for general homogeneous flows, \ac{DMD} can not recover the extinction time of the dynamics, and it induces a significant reconstruction error. Our analysis further shows that \ac{DMD} is well-suited for flows with one-homogeneous operators. Consequently and inspired by \cite{cohen2019stable}, we suggest a new scheme which employs an adaptive time sampling instead of a fixed time step size. We show that our temporal re-scaling is equivalent to evolving a one-homogeneous flow. With this adaptation, \ac{DMD} is able to recover homogeneous flows of order $[0,1]$. In the general case, we additionally obtain a much better mode recovering scheme which captures the dynamics of the flow well. Next, we show that the obtained modes approximate nonlinear eigenfunctions, allowing us to link \ac{DMD} to nonlinear spectral theory. In summary, our analysis and results yield a new and simple spectral decomposition framework. Our work generalizes previous studies which directly formulated nonlinear spectral representations based on total-variation \cite{gilboa2013spectral, gilboa2014total} and one-homogeneous functionals
\cite{burger2016spectral,bungert2019nonlinear} by applying (weak) time-derivatives to the solution of a gradient flow. It is also related to previous research in which signals were analyzed by their decay profile, as shown in \cite{Katzir2017Thesis, gilboa2018nonlinear}.

\paragraph{Main contributions and structure of paper}
Our contributions can be summarized as follows: \begin{enumerate}
    \item It is shown that \ac{DMD} is not effective for homogeneous flows with homogeneity different than one.
    This inherent limitation is formulated by what we term \emph{the \ac{DMD} paradox}, where as the step-size decreases, the standard DMD-error approaches zero, whereas the reconstruction error has a strictly positive lower bound.
    
    \item We propose a temporal re-parametrization scheme of the data sampling. We study cases with an analytic solution, and we show that our re-parametrization yields a single mode in \ac{DMD} which can be perfectly reconstructed. Finally, the relation to an analogue one-homogeneous flow is shown.
    
    \item The temporal re-parametrization of the data is generalized to arbitrary step sizes and to any homogeneity. We term this adaptation as the \emph{blind homogeneity normalization}, where the blind is twofold, neither the operator nor the temporal sampling are known.
    
    \item We adapt the \ac{DMD} algorithm to real valued spectrum systems, common in smoothing-type (non-oscillatory) flows. We refer to it as \acf{SDMD}.

    \item We introduce a new discrete analysis and synthesis framework of signals related to homogeneous flows of homogeneity order in $[0, 1]$. Our framework is based on orthogonal modes which approximate nonlinear eigenfunctions. We thus refer to it as Orthogonal Nonlinear Spectral decomposition. We numerically compare our decomposition to the method in \cite{cohen2020Introducing}, and we show that our scheme is simpler, more general, and it is $1-2$ orders of magnitude faster than \cite{cohen2020Introducing}. 

\end{enumerate}

The paper is organized as follows. We briefly recall the necessary mathematical definitions and previous results in Sec. \ref{sec:Preliminaries}. In Sec. \ref{sec:ourCont} a closed form solution for \ac{DMD} in certain cases is given and the paradox for homogeneous flows is stated. Our solution is proposed (in non-blind and blind versions) and analyzed. The \emph{OrthoNS} representation is formalized.
In addition, we introduce the \ac{SDMD} method such that \ac{DMD} is based on a symmetric matrix. In Sec. \ref{sec:results} we demonstrate \ac{SDMD}, then show the main modes of two gradient descent flows with respect to the  $p-$Dirichlet energy ($p=1.01$ and $p=1.5$), when initialized with a square peak. Filtering of signals by \emph{OrthoNS} are presented along with a comparison to \cite{cohen2020Introducing}.  We conclude our work and discuss future directions in Sec. \ref{sec:conclusion}.

\section{Preliminaries}\label{sec:Preliminaries} 
Let $\mathcal{H}$ be a real Hilbert space equipped with a norm $\|\cdot \|$. Typically, in a discrete setting, we have  $\mathcal{H}=\mathbb{R}^M$ and a Euclidean norm.
A common optimization problem, given some data $f\in \mathcal{H}$, is to seek a solution $\psi \in \mathcal{H}$
which minimizes
\begin{equation}\label{eq:optPro}
    G(\psi)=F(\psi,f)+R(\psi),
\end{equation}
where $F:\mathcal{H}\to \mathbb{R}^+$ is a fidelity (or data) term
and $R:\mathcal{H}\to \mathbb{R}^+$ is a regularization term.
In the most simple case, the denoising problem, $F$ can be the square $\ell^2$ norm and $R$ is the Dirichlet energy or, alternatively, the total-variation energy (yielding Tikhonov \cite{tikhonov2013numerical} or ROF \cite{rudin1992nonlinear} models, respectively).
The solution $\psi^*= \argmin_\psi G(\psi,f)$ is a compromise between the noisy data and a regular solution.
In this paper we focus on regularization terms which are absolutely $p$-homogeneous functionals, admitting
\begin{equation}\label{eq:homoFun}
    R(a\cdot \psi)=\abs{a}^{p}\cdot R(\psi),
\end{equation}
for any $a\in\mathbb{R}$. One can obtain a local minimizer by evolving a gradient descent process with respect to the total energy $G$. In the denoising case, when the fidelity is a simple Euclidean norm, an alternative solution is to evolve gradient descent with respect to $R$ only and to stop at a certain desired time-point in the process.
Given $-P = \partial_\psi R(\psi)$ the gradient descent flow is 
\begin{equation}\label{eq:homoFlow}
    \psi_t=-\partial_\psi R(\psi)=P(\psi),\quad \psi(0)=f,
\end{equation}
where $\psi_t$ is the time  derivative of the solution, and the initial condition is $f$.
The operator $P$ is $(p-1)$-homogeneous, 
\begin{equation}\label{eq:homo}
    P(a\cdot \psi)=a\abs{a}^{p-2}\cdot P(\psi), \,\,\,a\in\mathbb{R}.
\end{equation}
We refer to Eq. \eqref{eq:homoFlow} as a \emph{homogeneous flow}.

\subsection{The {\it{\textbf{p}}}-Framework} 

Based on nonlinear spectral representations of one-homogeneous functionals \cite{gilboa2014total,burger2016spectral,bungert2019nonlinear}, an extension to functionals of homogeneity $p\in(1,2)$ was proposed in \cite{cohen2020Introducing}.
As in previous studies, the representation is based on manipulating a gradient flow process. In \cite{cohen2020Introducing} the flow of Eq. \eqref{eq:homoFlow} was analyzed. 
A special case was investigated more deeply, where the solution of Eq. \eqref{eq:homoFlow} admits a separation of variables,
\begin{equation} \label{eq:SpF}
    \psi(t)=a(t) \cdot f.
\end{equation}
In this case the initial condition $f$ remains unchanged (spatially), while its scale changes over time. This form of solution is obtained iff $f$ is a (nonlinear) eigenfunction of $P$, i.e. it solves the following nonlinear eigenvalue problem,
\begin{equation}\label{eq:homoEF}\tag{\bf EF}
    P(f)=\lambda\cdot f,
\end{equation}
where $\lambda \in \mathbb{R}$ is the eigenvalue. 
The function $a(t)$ can be viewed as a \emph{decay profile}. The decay depends on the eigenvalue and the order of homogeneity, and it is given by,
\begin{equation}\label{eq:decayProfile}
    a(t)=\left[(1+(2-p)\lambda\cdot t)^+\right]^\frac{1}{2-p},
\end{equation}
where $(a)^+=\max\{0,a\}$. 
We note that the operator $-P(\cdot)$ is assumed to be maximally monotone. Therefore, its spectrum is non-positive, where $\lambda \le 0$.
The decay profile \eqref{eq:decayProfile} has thus a finite support in time. The solution reaches its steady state in finite time, termed as the \emph{extinction time}. The extinction time is given by
\begin{equation}\label{eq:homoExtinctionTime}
    T=-\frac{1}{\lambda(2-p)}.
\end{equation}
It can also be shown that under mild conditions on  $P$ the flow \eqref{eq:homoFlow} converges in finite time for arbitrary initial conditions.
A particular example for $R$ is the $p$-Dirichlet energy, thoroughly studied in the context of image processing by Kuijper (c.f. \cite{kuijper2007p}),
\begin{equation}\label{eq:pEnergy}
    R(\psi)=J_p(\psi)=\frac{1}{p}\norm{\nabla \psi}^p.
\end{equation}
Kuijper suggested to use the respective gradient descent flow as a nonlinear scale space,
\begin{equation}\label{eq:pFlow}\tag{\bf p-Flow}
    \psi_t = \Delta_p \psi,\quad \psi(0)=f,
\end{equation} 
where $\Delta_p(\cdot)$ is the $p$-Laplacian operator, 
\begin{equation}
    \Delta_p(\psi)=\div{\{\abs{\nabla \psi}^{p-2} \nabla \psi\}}.
\end{equation}
The $p$-Dirichlet is an absolutely $p$-homogeneous functional and the $p$-Laplacian operator is coercive and maximally monotone. Therefore, the discussion above is valid for \eqref{eq:pFlow} for $p\in[1,2)$. 
%
%
If the initial condition, $f$, admits \eqref{eq:homoEF}, for $P=\Delta_p$, then the solution of \eqref{eq:pFlow} is given by Eqs. \eqref{eq:SpF}, \eqref{eq:decayProfile} and  the extinction time is \eqref{eq:homoExtinctionTime}. 
In Fig. \ref{Fig:EFDecay} the process \eqref{eq:pFlow} is depicted for $p=1.5$.

\begin{figure}[phtb!]
\centering
\captionsetup[subfigure]{justification=centering}
\subfloat[Decay profile]
{
\includegraphics[trim=30 20 0 40, clip,width=0.25\textwidth,valign=c]{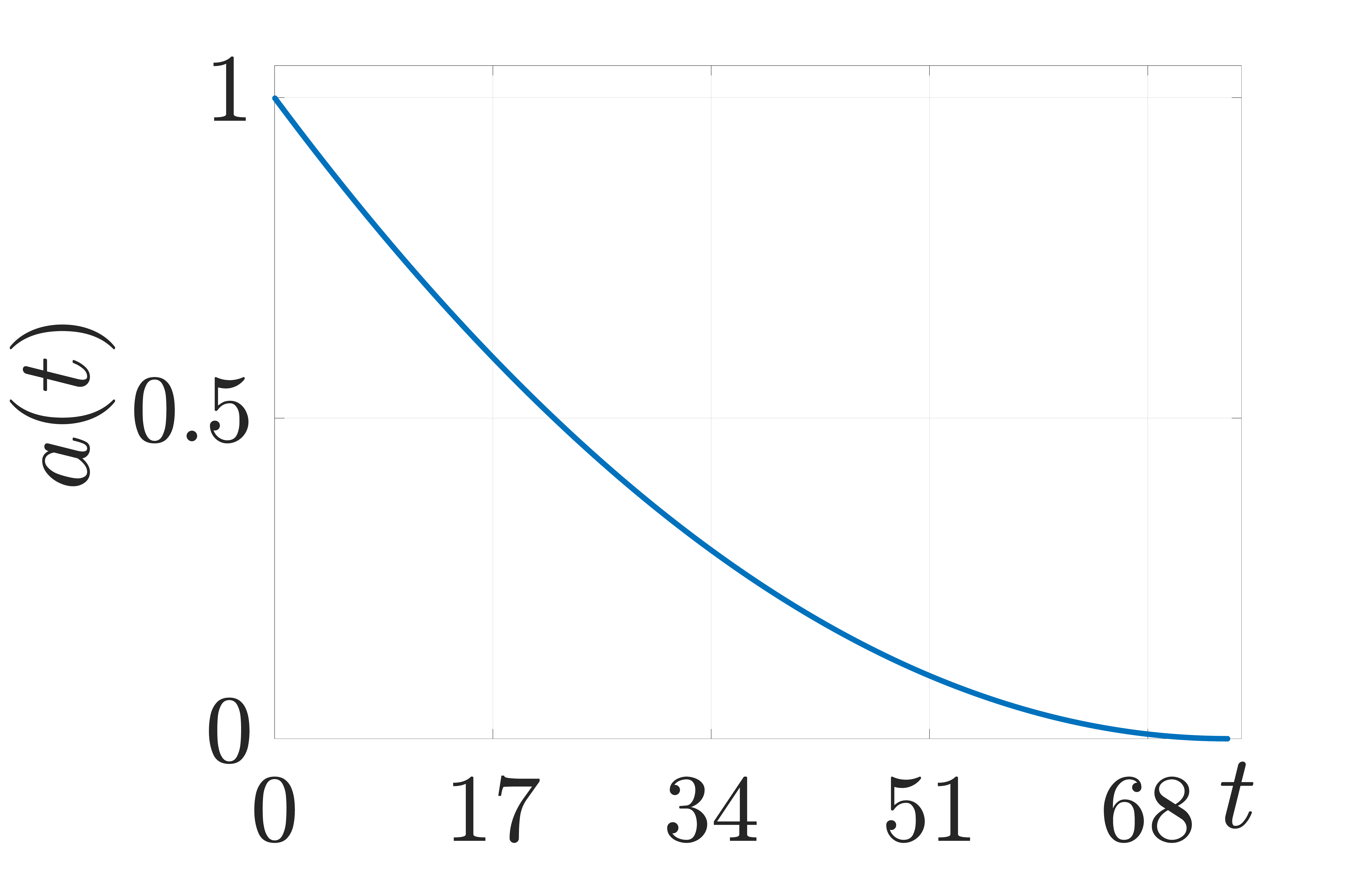}
\label{subfig:decayD2}
}$\quad$
\subfloat[$T=0$]
  {
\includegraphics[trim=0 -30 0 0, clip,width=0.125\textwidth,valign=c]{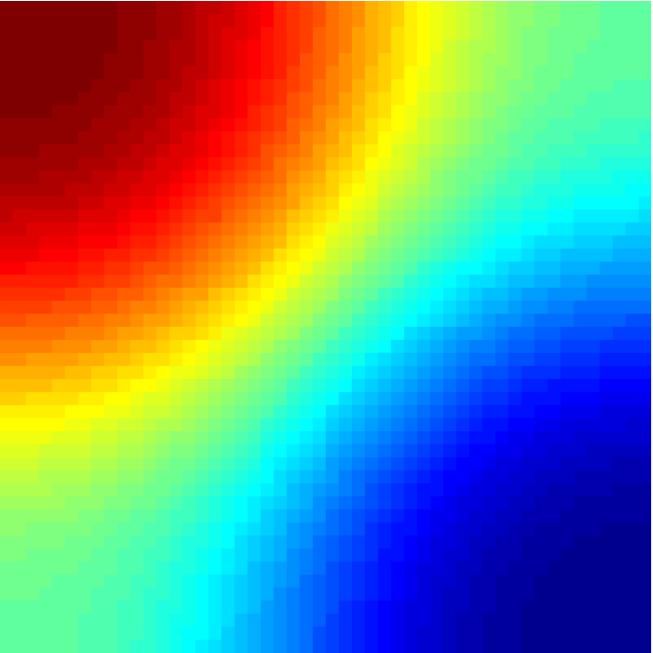}
\label{subfig:EFT0}
}
$\cdots$
\subfloat[$T=12$]
{
\includegraphics[trim=0 -30 0 0, clip,width=0.125\textwidth,valign=c]{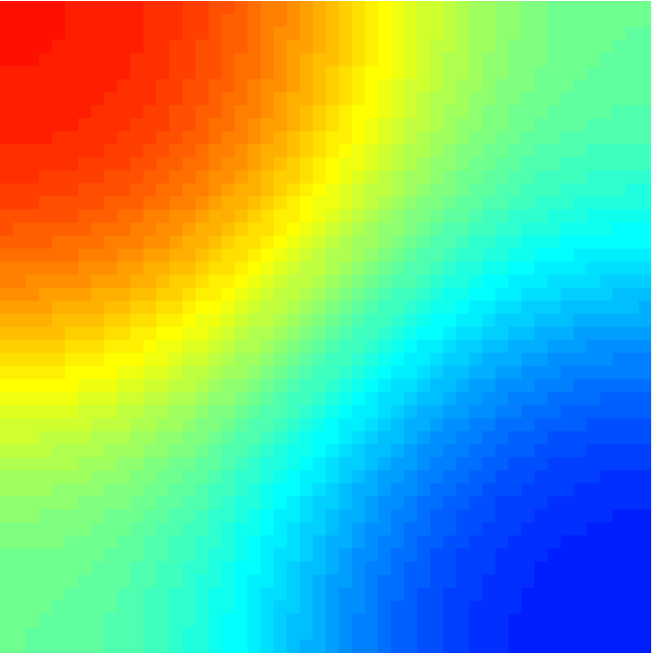}
\label{subfig:EFT20}
}
$\cdots$
\subfloat[$T=28$]
{
\includegraphics[trim=0 -30 0 0, clip,width=0.125\textwidth,valign=c]{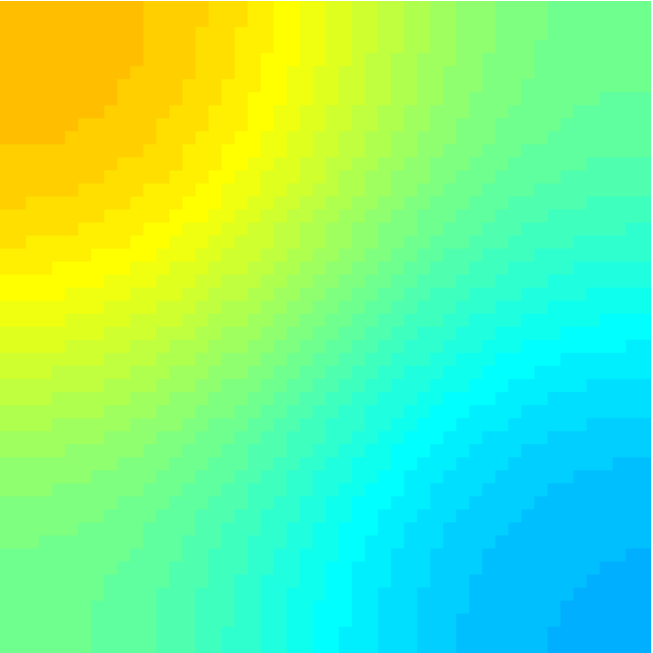}
\label{subfig:EFT40}
}
$\cdots$
\subfloat[$T=74$]
  {
\includegraphics[trim=0 -30 0 0, clip,width=0.125\textwidth,valign=c]{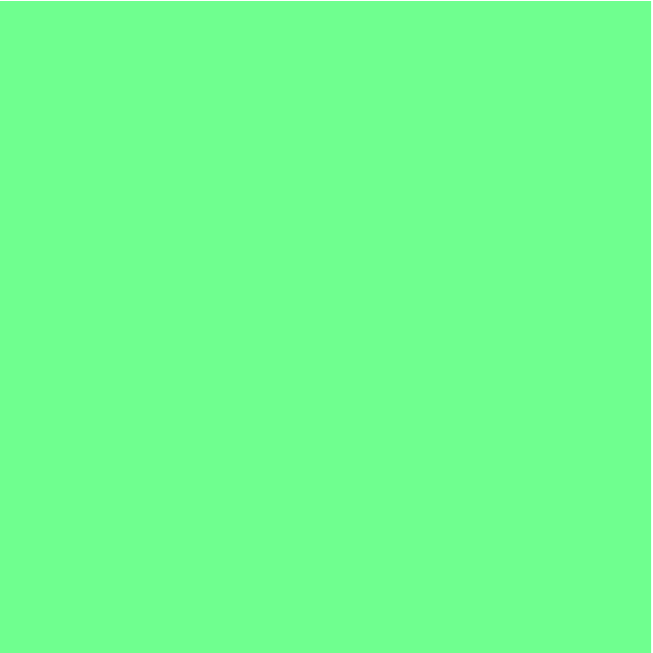}
\label{subfig:EFT74}
}
\caption{{\bf The solution of $\psi_t=\Delta_p \psi$ initialized with an eigenfunction ($p=1.5,\,\lambda=-0.0269$).} Left to right: {\bf(a)} The decay profile, \eqref{eq:decayProfile}, {\bf(b)-(e)} snapshots of the solution, $\psi(t)$, at different time points.}
\label{Fig:EFDecay}
\end{figure}

\subsection{Time discretisation of homogeneous flows}\label{sec:ExplicitScheme}
The explicit scheme of Eq. \eqref{eq:homoFlow} reads,
\begin{equation}\label{eq:homoESPflow}
    \psi_{k+1}=\psi_k+P(\psi_k)\cdot dt_k,\quad \psi_0=f.
\end{equation}
If the initial condition, $f$, is an eigenfunction \eqref{eq:homoEF}, then the solution of  \eqref{eq:homoESPflow} is 
\begin{equation}\label{eq:disShP}
    \psi_k = a_k\cdot f,\quad a_k\in \mathbb{R}.
\end{equation}
For $(p-1)$-homogeneous operator $P$ the recurrence relating $a_{k+1}$ to $a_k$ is \cite{cohen2019stable}
\begin{equation}\label{eq:akrecurr}
    a_{k+1} = a_k\left(1+\abs{a_k}^{p-2}\lambda dt_k\right),\quad a_0=1.
\end{equation}
In standard explicit implementations one needs to regularize the functional (and respective operator) to obtain a practical step size (dictated by the CFL condition). This yields either very small step-sizes or strong deviation from the original flow. In  \cite{cohen2019stable} an alternative scheme was proposed which uses an adaptive step size. We will later see how this scheme directly connects to our proposed time re-sampling. 

\paragraph{Adaptive step size policy}
In \cite{cohen2019stable} an adaptive step-size policy was proposed for the explicit scheme \eqref{eq:homoESPflow}, given by
\begin{equation}\label{eq:adaptiveStepSize}
    dt_k=-\frac{\inp{P(\psi_k)}{\psi_k}}{\norm{P(\psi_k)}^2}\cdot \delta,\quad \delta\in(0,2),
\end{equation}
where $\delta$ is a free parameter controlling the speed of the process. It was shown that the scheme is stable for arbitrary $f$ and $\delta\in(0,2)$. For $f$ which is an eigenfunction the solution is Eq. \eqref{eq:disShP}, where 
\begin{equation}\label{eq:seriesAdaptive}
    a_k=(1-\delta)^k.
\end{equation}
We note that although \cite{cohen2019stable} focused on the $p$-Laplacian flow, the above results are valid for any homogeneous operator of order between zero and one which  is coercive and maximally monotone.
See a recent general study \cite{welk2020pde}
on the relations between statistical estimators of order $p$ and their respective PDE's in the limit. 

\subsection{\acf{DMD}}\label{Sec:DMD}
\ac{DMD} \cite{schmid2010dynamic} is an analysis tool used to recover the main spatial structures in a fluid flow. Its stages are detailed in Algorithm \ref{algo:DMD}. We first present the rationale, notations and definitions of the algorithm. Vectors are denoted by boldface and matrices by capital letters. We depict here the general case, however, in the following sections we focus on the representation of the data matrices over the real field.

\paragraph{\bf Matrices of the dynamics} The data consists of $N+1$ snapshots in time of a flow $\bm{\psi_k}\in \mathbb{R}^M$, $\{\bm{\psi_k}\}_{k=0}^N$. 
We construct two $M\times N$ matrices as follows, 
\begin{equation}\label{eq:DMDdataForming}
    \begin{split}
        \Psi_0^{N-1}&=\begin{bmatrix}\bm{\psi_0} & \cdots & \bm{\psi_{N-1}}\end{bmatrix}\\
        \Psi_1^N&=\begin{bmatrix}\bm{\psi_1}&\cdots&\bm{\psi_{N}}\end{bmatrix}.
    \end{split}
\end{equation}

\paragraph{\bf Dimensionality reduction} A main assumption of \ac{DMD} is that the data can be well represented in a lower dimensional space. To reduce the dimensionality we need to find the singular vectors that span the columns of $\Psi_0^{N-1}$. \ac{SVD} is used to find these vectors since the matrix $\Psi_0^{N-1}$ is not square. This decomposition is an extension of the eigenvector problem for non square matrices (for details see e.g. \cite{strang2019linear} Ch. I.8), and it is given by
\begin{equation}\label{eq:SVD}
    \Psi_0^{N-1}=U\Sigma V^*.
\end{equation}
The superscript $^*$ denotes the conjugate transpose. 
The matrix  $U$  is an $M\times N$ orthogonal matrix ($U^*U=I_{N \times N}$),  $V$  is an $N\times N$ orthogonal matrix ($V^*V=I_{N \times N}$), and $\Sigma$ is an $N\times N$ diagonal matrix, where the entries on the diagonal are the singular values.  

We denote by $U_r$ and $V_r$ the submatrices, containing the first $r$ columns of $U$ and $V$, respectively ($r\ll M$).  $\Sigma_r$ is the a submatrix of $\Sigma$, containing the $r\times r$ left upper entries of $\Sigma$. The dimensionality reduction of the data is obtained by setting
\begin{equation}\label{eq:DMDdimeReduction}
    \begin{split}
        X&= U_r^*\Psi_0^{N-1},\qquad
        Y= U_r^*\Psi_1^{N}.
    \end{split}
\end{equation}
The $k$th column in the matrix $X$, denoted by ${\bm{x_k}}$, is the lower dimensional representation of the $k$th snapshot, ${\bm{\psi_k}}$, i.e. ${\bm{x_k}}=U_r^*{\bm{\psi_k}}$.
%
Note that the columns of $U_r$ are a basis of a linear space and the entries of $\bm{x_k}$ can be viewed as coordinates of the snapshot $\bm{\psi_k}$ in that space.

\paragraph{Mode, spectrum and coordinates calculation}
In the lower-dimensional space we seek a linear mapping $F$ from $X$ to $Y$ that minimizes the Frobenius norm
\begin{equation}\label{eq:DMDERR}
    ERR_{DMD}=\min_{F}\norm{Y-FX}^2_\mathcal{F}.
\end{equation}
The solution of this optimization problem is given by,
\begin{equation}\label{eq:DMDmatrix}
    F=YX^T\cdot\left(XX^T\right)^{-1}= U_r^*\Psi_1^{N+1}V_r\Sigma_r^{-1},
\end{equation}
termed as the \emph{DMD matrix}. 
Then, the $k+1$th sample can be expressed as
\begin{equation}\label{eq:DMDreccu}
    \bm{x_{k}}\approx F\cdot \bm{x_{k-1}}.
\end{equation}
We denote by $\approx$  the linear, dimensionality-reduced approximation of the dynamical system. The linear mapping approximation, $F$, minimizes the Frobenius norm of the error with respect to the first $r$ dominant singular vectors of the data.
Assuming the matrix $F$ is full rank, we can reformulate Eq. \eqref{eq:DMDreccu} as
\begin{equation}\label{eq:DMDreccu2}
    \bm{x_{k}}\approx WDW^*\cdot \bm{x_{k-1}},
\end{equation}
where $D$ is a diagonal matrix containing the eigenvalues of $F$, and the matrix $W$ contains the corresponding eigenvectors. 
\paragraph{Reconstructing the Dynamics}
\subparagraph{Discrete time setting}
We can reconstruct the dynamics projected on the lower dimensional space, e.g. the initial condition is reconstructed by $\bm{\tilde{\psi}_0}=U_r\bm{x_0}$.
More generally, to reconstruct a snapshot at stage $k$ we can apply the mapping $F$ $k$ times, 

\begin{equation}\label{eq:modeSpectrumCoo}
    \begin{split}
    \bm{\tilde{\psi}}_{k}= &U_r\cdot F^k\bm{x_0}= U_r\cdot W D^k W^* U_r^*\bm{\psi}_0\\
    =&U_r\cdot\begin{bmatrix}\bm{w_1}&\cdots&\bm{w_r} \end{bmatrix}\cdot 
    \begin{bmatrix}\mu_1^k& &0\\
                    &\ddots&\\
                    0&  &\mu_r^k 
    \end{bmatrix}\cdot
    \begin{bmatrix}\bm{w_1}^*\\\vdots\\\bm{w_r}^*\end{bmatrix}\cdot U_r^*\bm{\psi}_0=\sum_{i=1}^r\alpha_i\mu_i^k\bm{\phi_i},
    \end{split}
\end{equation} 
where the modes, $\{\bm{\phi_i}\}_{i=1}^r$, and coordinates $\{\alpha_i\}_{i=1}^r$ are
\begin{equation*}
    \bm{\phi_i}=U_r \bm{w_i},\quad \alpha_i={\bm{w_i}}^*U_r^*\psi_0,
\end{equation*}
and $\{\mu_i\}_{i=1}^r$ are the eigenvalues of the matrix $F$ (the spectrum).
\begin{algorithm}[phtb!] \caption{Standard DMD \cite{schmid2010dynamic}}
\begin{algorithmic}[1]
		\Inputs{Data sequence $\{{\bm{\psi}_k}\}_0^{N}$}
		\State Arrange the data into the matrices $\Psi_0^{N-1}$ and $\Psi_1^{N}$ according to Eq. \eqref{eq:DMDdataForming}.
		\State Compute the \acf{SVD} of $\Psi_0^{N-1}$ (see \cite{trefethen1997numerical}) to the multiplication in Eq. \eqref{eq:SVD}.
		\State Dimensionality reduction. Reformulate the data matrices, $\Psi_0^{N-1},\,\Psi_1^{N}$ (denoted by $X$ and $Y$, respectively) with the first $r$ singular vectors from the matrix, $U$,  Eq. \eqref{eq:DMDdimeReduction}.
		\State \label{state:DMDmatrix}Find the optimal linear mapping, $F$, between $X$ and $Y$ in the sense of Eq. \eqref{eq:DMDERR}.
		 The solution is given by Eq. \eqref{eq:DMDmatrix}.
		\State Under the assumption that $F$ is a full rank matrix, compute eigenvalues $\mu$ and right eigenvectors $\bm{v}$ of $F$, the corresponding modes $\bm{\phi}$, and the corresponding coordinates $\alpha$ by 
		\begin{equation}\label{eq:DMDResult}
		F\bm{w}=\mu \bm{w},\quad \bm{\phi}\triangleq U_r\bm{w},\quad\alpha\triangleq \bm{w}^* U_r^*\bm{\psi_0}.    
		\end{equation}
		
		\Outputs{$$\{\mu_i,\bm{\phi_i},\alpha_i\}_1^r$$}
    \end{algorithmic}
    \label{algo:DMD}
\end{algorithm}
Note that, the linear mapping from $\tilde{\psi}_k$ to $\tilde{\psi}_{k+1}$, denoted by $A$ is give  by,
\begin{equation}\label{eq:linearApproxA}
    A = U_r\cdot F\cdot U_r^*.
\end{equation}
This mapping can be interpreted as a linear approximation of the dynamical system. The modes $\{\phi_i\}$ are the right eigenvectors of the matrix $A$ and the corresponding eigenvalues are $\{\mu_i\}$.
For reconstruction, we define the (time-discrete) \emph{reconstruction error}, $ERR^d_{Rec}$, as
\begin{equation}\label{eq:reconErr}
ERR_{Rec}^d=\sum_{k=0}^N\norm{\bm{\tilde{\psi}_k}-\bm{\psi_k}}^2,
\end{equation}
where $\bm{\tilde{\psi}_k}$ is defined in Eq. \eqref{eq:modeSpectrumCoo}. 
\subparagraph{Continuous time setting}
One can expand the reconstruction, Eq. \eqref{eq:modeSpectrumCoo}, to the continuous time setting. The discrete reconstruction is a  sampling of a continuous exponential function, therefore, with the identity, 
\begin{equation*}
    \mu_i=e^{\tilde{\mu}_i dt},
\end{equation*}
where $dt$ is the sampling time step 
Thus, in the time continuous setting, the dynamics reconstruction and the corresponding error take the form,
\begin{equation}\label{eq:DMDreconConti}
    \tilde{\psi}(t)=\sum_{i=1}^r\alpha_i\bm{\phi_i}e^{\tilde{\mu}_it},\quad
    \tilde{\mu}_i=\frac{\ln(\mu_i)}{dt}, \quad
    ERR_{Rec}^c = \int(\tilde{\psi}(t)-\psi(t))^2dt.
\end{equation}

We would like also to consider the limit case, as the step size between consecutive snapshots, $dt$, approaches zero. In that case, the eigenvalue $\tilde{\mu}_i$ is the limit of the quotient written above.

We note that this is the classical algorithm and several variations and extensions were further proposed. It was shown in \cite{kutz2016dynamic} that \ac{DMD} is sensitive to noisy data . Specifically, the spectrum estimation is systematically biased in the presence of noise. This bias is not relaxed when more data is gathered \cite{hemati2017biasing}. The effect of small sensor noise on \ac{DMD} and on
the Koopman expansion was studied and characterized in \cite{bagheri2013effects}. Several attempts have been made to remove this bias. Dawson et al. proposed the forward and backward dynamics to reduce the noise \cite{dawson2016characterizing}. Hemati et al. formulated the problem as a total least squares optimization \cite{hemati2017biasing}. A variational approach was proposed in \cite{azencot2019consistent}. The authors in \cite{nonomura2018dynamic, nonomura2019extended} used Kalman filters to cope with the noise. Williams et al. \cite{williams2015data} suggested to extend the basis of the sampled data snapshots, while \cite{li2017extended} uses deep learning to learn the basis dictionary of the operator. More recent approaches harness the benefits of neural networks to propose effective Koopman-based designs \cite{azencot2020forecasting}. 


\section{\ac{DMD} for homogeneous and symmetric flows}
\label{sec:ourCont} 
This section presents the main novelties of the paper. \ac{DMD} is analyzed for homogeneous flows, its flaws are exposed and a solution is presented in the form of a time re-sampling scheme. For non-oscillatory flows we propose \ac{SDMD}. Finally, we show connections of the modes to nonlinear eigenfunctions and propose the \emph{OrthoNS} analysis and synthesis framework.

\subsection{\ac{DMD} for homogeneous flows} 
There are two prominent assumptions in \ac{DMD}; first, the dynamics can be represented linearly in a lower dimensional space; second, the data is sampled uniformly in time. These assumptions allow us to interpret the system as a linear one \cite{schmid2010dynamic} and to consider \ac{DMD} as an exponential data fitting algorithm \cite{askham2018variable}.
Thus, the finite extinction time of homogeneous flows is inherently hard to model in this framework. In what follows, we show the inconsistency and error of \ac{DMD} applied to flows initiated with eigenfunctions. 

\subsubsection{Closed from solution of \ac{DMD}}
We begin by computing the \ac{DMD} modes of the homogeneous flow, Eq. \eqref{eq:homoFlow}, when the initial condition is an eigenfunction. Let us recall that the solution is Eq. \eqref{eq:SpF}. Therefore, by sampling this solution with respect to time (with a fixed step size) we get
\begin{equation}\label{eq:disShPcloseRef}
    \psi_k = a_k\cdot f,\quad a_k\in \mathbb{R},\, a_0 =1.
\end{equation}
Therefore, the data matrices $\Psi_1^N$ and $\Psi_0^{N-1}$ (Eq. \eqref{eq:DMDdataForming}) are in the form of
\begin{equation}\label{eq:DMDmatrixHomoEF}
    \Psi_1^N=f\cdot \left({\bm{a}_1^N}\right)^T,\quad \Psi_0^{N-1}=f\cdot \left({\bm{a}_0^{N-1}}\right)^T,
\end{equation}
where $\bm{a}_k^m=\begin{bmatrix}
    a_k&\cdots&a_m
    \end{bmatrix}^T$. The following Lemma formulates an analytic solution of the classical \ac{DMD} for these cases.

\begin{lemma}[Analytic solution]\label{lem:inconsistencyDMD} 
Let the dynamical system be Eq. \eqref{eq:homoFlow} and the initial condition $f$ is an eigenfunction per Eq. \eqref{eq:homoEF}.
We obtain the following analytic solution and error for classical \ac{DMD}, Algo. \ref{algo:DMD}, 
for $r=1$: 
    \begin{equation}\label{eq:lemmaFirstResult}
        \mu=\frac{\inp{\bm{a}_1^N}{\bm{a}_0^{N-1}}}{\norm{\bm{a}_0^{N-1}}^2},\quad \phi = \frac{f}{\norm{f}},\quad \alpha= \norm{f}.
    \end{equation}
The \ac{DMD} error (Eq. \eqref{eq:DMDERR}) is
    \begin{equation}\label{eq:lemmaSecondResult}
       ERR_{DMD}=\norm{\bm{a}_1^{N}}^2 - \frac{\inp{\bm{a}_1^N}{\bm{a}_0^{N-1}}^2}{\norm{\bm{a}_0^{N-1}}^2},
    \end{equation}
where $\bm{a}_0^N$ and $\bm{a}_1^{N+1}$ are defined in \eqref{eq:DMDmatrixHomoEF}.
For $r>1$ the solution does not exist.
\end{lemma}
\begin{proof}$\quad$\\
The data matrix $\Psi_0^{N-1} $ (in Eq. \eqref{eq:DMDmatrixHomoEF}) can be reformulated as
\begin{equation*}
    \Psi_0^{N-1} = U\cdot \Sigma \cdot V^*=\frac{f}{\norm{f}}\cdot \norm{f}\norm{\bm{a}_0^{N-1}} \cdot \frac{\bm{a}_0^{N-1}}{\norm{\bm{a}_0^{N-1}}}.
\end{equation*}
The \ac{SVD} of this matrix is $\Psi_0^{N-1}=U\Sigma V^*$ 
where,
$U$ is the column vector ${f}/{\norm{f}}$ concatenated by a zero matrix of size ($M\times (N-1)$), $\Sigma$ is $N\times N$ matrix with zeros everywhere except the entry $(1,1)$ where $\Sigma(1,1)=\norm{f}\norm{\bm{a}_0^{N-1}}$, and $V$ is the column vector ${\bm{a}_0^{N-1}}/{\norm{\bm{a}_0^{N-1}}}$ concatenated by a zero matrix of size ($N\times (N-1)$). More formally,
\begin{equation*}
    U=\begin{bmatrix}
    \frac{f}{\norm{f}}&{\bm{0}}&\cdots&{\bm{0}}
    \end{bmatrix},\, \Sigma = diag\left(\begin{bmatrix}
    \norm{f}\norm{\bm{a}_0^{N-1}}&0&\cdots&0
    \end{bmatrix}\right), \, V=\begin{bmatrix}
    \frac{\bm{a}_0^{N-1}}{\norm{\bm{a}_0^{N-1}}}&{\bm{0}}&\cdots&{\bm{0}}
    \end{bmatrix},
\end{equation*}
where $U\in\mathbb{R}^{M\times N}$, $\Sigma\in\mathbb{R}^{N\times N}$,  $V\in\mathbb{R}^{N\times N}$, and ${\bm{0}}$ is a column zero vector in $\mathbb{R}^M$ or $\mathbb{R}^N$.
If $r=1$ then $U_1,\, V_1$ are vectors and $\Sigma_1$ is a scalar, where they accurately reconstruct $\Psi_0^{N-1}$. This is an expected result since the rank of $\Psi_0^{N-1}$ is one. Consequently, $Y$ and $X$ are the following vectors,
\begin{equation*}
    Y=U_1^*\cdot\Psi_1^{N}= \norm{f}{\bm{a}_1^{N}},\quad X=U_1^*\cdot\Psi_0^{N-1}=\norm{f}{\bm{a}_0^{N-1}}.
\end{equation*}
The \ac{DMD} matrix becomes a scalar, $\mu$, which minimizes the following term 
\begin{equation*}
\begin{split}
    \mu_{min}=\arg\min_{\mu}\{\norm{Y-\mu X}_F^2\}=\arg\min_{\mu}\{\norm{\bm{a}_1^N-\mu\bm{a}_0^{N-1}}^2\}=\frac{\inp{\bm{a}_1^N}{\bm{a}_0^{N-1}}}{\norm{\bm{a}_0^{N-1}}^2}.
\end{split}
\end{equation*}
The eigenvector is $v=1$. Using the above results and \eqref{eq:DMDResult}, \eqref{eq:DMDERR} yields \eqref{eq:lemmaFirstResult} and  \eqref{eq:lemmaSecondResult}. The error is strictly positive unless the vectors $\bm{a}_1^{N}$ and $\bm{a}_0^{N-1}$ are co-linear. This case does not happen for a fixed step size. 


If $r>1$ then neither the matrix $XX^T$ nor $\Sigma_r$ are invertible. Therefore, the solution of Eq. \eqref{eq:DMDmatrix} does not exist. 

\end{proof}

\subsubsection{Sampling and Dimensionality vs. Error - The \ac{DMD} paradox}
The linear mapping, $F$, minimizes the Frobenius norm of the recurrence relation error. 
As a conclusion from Lemma \ref{lem:inconsistencyDMD},  $ERR_{DMD}$ decreases by increasing the sampling rate, since $\mu$ approaches zero as the step size approaches zero. In the next Lemma, we formulate the \ac{DMD} solution when the step size approaches zero.

\begin{lemma}[Time-continuous reconstruction]\label{lem:ConRec} Let the conditions of Lemma \ref{lem:inconsistencyDMD} hold. Let the dynamical system be sampled $N$ times in the interval $[0,T_{ext}]$, where the step size is $dt$. We denote by $\tilde{\psi}(t)= \bm{\tilde{\psi}_k}$, where $t=k\cdot dt$ and $\bm{\tilde{\psi}_k}$ is defined in \eqref{eq:modeSpectrumCoo}.
As $N\to \infty$  ($dt\to 0$),  the time-continuous reconstruction is 
\begin{equation*}
    \tilde{\psi}(t)=f\cdot e^{\tilde{\mu}t},
\end{equation*}
where
\begin{equation}
    \tilde{\mu} = \lambda\frac{4-p}{2}.
\end{equation}
\end{lemma}
\begin{proof}
According to Lemma \ref{lem:inconsistencyDMD} the eigenvalue is
\begin{equation*}
    \mu=\frac{\inp{\bm{a}_1^N}{\bm{a}_0^{N-1}}}{\norm{\bm{a}_0^{N-1}}^2}=\frac{\sum_{k=1}^Na_ka_{k-1}}{\sum_{k=1}^N{a_{k-1}^2}},
\end{equation*}
where the series $\{a_k\}$ is the sampled solution, Eq. \eqref{eq:decayProfile}, i.e. $a_k=a(t_k)=a(k\cdot dt)$. Then,
\begin{equation*}
    \begin{split}
        \tilde{\mu} &=\frac{\ln{\left(\mu\right)}}{dt} = 
        \frac{1}{dt}\ln\left(\frac{\sum_{k=1}^Na_ka_{k-1}}{\sum_{k=1}^N{a_{k-1}^2}}\right)=\frac{1}{dt}\ln\left(\frac{\sum_{k=1}^N{a_{k-1}^2}-\sum_{k=1}^N{a_{k-1}^2}+\sum_{k=1}^Na_ka_{k-1}}{\sum_{k=1}^N{a_{k-1}^2}}\right)\\
        &=\frac{1}{dt}\ln\left(1+\frac{\sum_{k=1}^N{a_{k-1}}\left(a_k-a_{k-1}\right)}{\sum_{k=1}^N{a_{k-1}^2}}\right)=\frac{1}{dt}\ln\left(1+\frac{\sum_{k=1}^N{a_{k-1}}\frac{a_k-a_{k-1}}{dt}dt}{\sum_{k=1}^N{a_{k-1}^2}}\right).
    \end{split}
\end{equation*}
As $N\to \infty$ the denominator approaches  infinity whereas the numerator is finite. We thus use Taylor's series for the $\ln$ function to get
\begin{equation*}
    \begin{split}
        \tilde{\mu} &=\frac{\sum_{k=1}^N{a_{k-1}}\frac{a_k-a_{k-1}}{dt}dt}{\sum_{k=1}^N{a_{k-1}^2}dt}=\frac{\sum_{k=1}^N{a(t_k-dt)}\cdot \frac{a(t_k)-a(t_k-dt)}{dt}dt}{\sum_{k=1}^N{a_{k-1}^2}dt}.
    \end{split}
\end{equation*}
Taking the limit $dt\to 0$ for the above expression yields
$$ \frac{\int_0^{T_{ext}}a(t)a'(t)dt}{\int_0^{T_{ext}}a^2(t)dt}.
$$
Substituting $a(t)$ by the decay profile Eq. \eqref{eq:decayProfile} we obtain 
\begin{equation*}
    \begin{split}
        \tilde{\mu} &=\frac{\frac{1}{2}a^2(t)\eval_0^{T_{ext}}}{\frac{1}{\frac{2}{2-p}+1}\frac{1}{\lambda(2-p)}[0-1]}=\frac{\frac{1}{2}}{\frac{1}{\frac{2}{2-p}+1}\frac{1}{\lambda(2-p)}}=\lambda \frac{4-p}{2}.
    \end{split}
\end{equation*}
\end{proof}

\begin{remark}[Continuous reconstruction from the explicit scheme]
One can reach similar results as in Lemma \ref{lem:ConRec} by taking the explicit scheme  \eqref{eq:homoESPflow} to the limit $dt\to 0$. Here the  recurrence relation between $a_{k}$ and $a_{k+1}$ is Eq. \eqref{eq:akrecurr}. The eigenvalue can be expressed by,
\begin{equation*}
    \begin{split}
        \tilde{\mu}&=\frac{\ln{\left(\mu\right)}}{dt} = 
        \frac{1}{dt}\ln\left(\frac{\sum_{k=1}^Na_ka_{k-1}}{\sum_{k=1}^N{a_{k-1}^2}}\right)\\
        &= \frac{1}{dt}\ln\left(\frac{\sum_{k=1}^{N}a_{k-1}^2+\lambda\sum_{k=1}^{N}\abs{a_{k-1}}^p dt}{\sum_{k=1}^N{a_{k-1}^2}}\right)
        = \frac{1}{dt}\ln\left(1+\lambda dt\frac{\sum_{k=1}^{N}\abs{a_{k-1}}^p }{\sum_{k=1}^N{a_{k-1}^2}}\right)
    \end{split}
\end{equation*}
Taking the limit $dt \to 0$, we have
\begin{equation*}
    \begin{split}
        \lim_{dt\to 0}\frac{\sum_{k=1}^{N}\abs{a_{k-1}}^p  }{\sum_{k=1}^N{a_{k-1}^2}}&=\lim_{dt\to 0}\frac{\sum_{k=1}^{N}\abs{a_{k-1}}^p dt }{\sum_{k=1}^N{a_{k-1}^2}dt}=\frac{\int_0^{T_{ext}} a(t)^p dt}{\int_0^{T_{ext}} a(t)^2 dt}
        =\frac{\frac{1}{\frac{p}{2-p}+1}}{\frac{1}{\frac{2}{2-p}+1}}=\frac{4-p}{2}.
    \end{split}
\end{equation*}
Using Taylor's series for $\ln(1+x)$ we get
\begin{equation*}
    \tilde{\mu} = \lim_{dt\to 0} \frac{1}{dt}\ln\left(1+\lambda dt\frac{4-p}{2}\right)=  \lambda\frac{4-p}{2}.
\end{equation*}

\end{remark}

We show now that though $ERR_{DMD}$ approaches zero as $dt \to 0$, \ac{DMD} does not reconstruct the dynamics correctly and the reconstruction error, $ERR_{Rec}$, is positive. 
It implies that in certain cases, neither increasing the sampling density nor increasing the sub-space  dimensionality improves the recovery of the dynamics. We refer to it as  \emph{the \ac{DMD} paradox}. This is formalized in the following theorem. 
\newpage 
\begin{theorem}[The \ac{DMD} paradox]\label{theo:DMDPara}
    Let the conditions of Lemma \ref{lem:ConRec} hold. 
    \begin{enumerate}
        \item If the dimensionality is one, $r=1$, and, $N\to \infty$, then $ERR_{DMD}\to 0$ (Eq. \ref{eq:DMDERR}), however, the reconstruction error (Eq. \eqref{eq:DMDreconConti}) $ERR_{Rec}^c\ge B$, where
        \begin{equation*}
            B=-\norm{f}^2\frac{1}{\lambda(4-p)}\left[1- \sqrt[]{1-e^{-\frac{4-p}{2-p}}}\right]^2>0.
        \end{equation*}
        \item One cannot reduce  $ERR_{Rec}^c$ by increasing the dimensionality, $r>1$.
    \end{enumerate}
\end{theorem}

\begin{proof} $\quad$
\begin{enumerate}
    \item Following Lemma \ref{lem:ConRec} we have 
    \begin{equation*}
        \begin{split}
            ERR_{DMD}&=\norm{\bm{a}_1^{N}}^2-\frac{\inp{{\bm{a}_1^N}}{\bm{a}_0^{N-1}}^2}{\norm{\bm{a}_0^{N-1}}^2}=\norm{\bm{a}_1^{N}}^2-\frac{\inp{{\bm{a}_1^N}-\bm{a}_0^{N-1}+\bm{a}_0^{N-1}}{\bm{a}_0^{N-1}}^2}{\norm{\bm{a}_0^{N-1}}^2}\\
            &=\norm{\bm{a}_1^{N}}^2-\frac{\left(\inp{{\bm{a}_1^N}-\bm{a}_0^{N-1}}{\bm{a}_0^{N-1}}+\inp{\bm{a}_0^{N-1}}{\bm{a}_0^{N-1}}\right)^2}{\norm{\bm{a}_0^{N-1}}^2}\\
            &=\left(\norm{\bm{a}_1^{N}}^2-\norm{\bm{a}_0^{N-1}}^2\right)-2
            \inp{{\bm{a}_1^N}-\bm{a}_0^{N-1}}{\bm{a}_0^{N-1}}-\frac{\inp{{\bm{a}_1^N}-\bm{a}_0^{N-1}}{\bm{a}_0^{N-1}}^2}{\norm{\bm{a}_0^{N-1}}^2}.
        \end{split}
    \end{equation*}
    We can now calculate the limit of each term as $N \to \infty$. For the first term,
    \begin{equation*}
        \begin{split}
            \lim_{N\to \infty}\left(\norm{\bm{a}_1^{N}}^2-\norm{\bm{a}_0^{N-1}}^2\right) = \lim_{N\to \infty}\left(\sum_{k=1}^Na_k^2-\sum_{k=0}^{N-1}a_k^2\right)=a_0^2=1.
        \end{split}
    \end{equation*}
    For the second term,
    \begin{equation*}
        \begin{split}
            \lim_{N\to \infty}\left(\inp{{\bm{a}_1^N}-\bm{a}_0^{N-1}}{\bm{a}_0^{N-1}}\right)&=\lim_{N\to \infty}\left(\sum_{k=1}^{N}(a_k-a_{k-1})a_{k-1}\right)\\
            &=\lim_{dt\to 0}\int_0^{T_{ext}}\frac{a(t+dt)-a(dt)}{dt}a(t)dt\\
            &=\lim_{dt\to 0}\int_0^{T_{ext}}\frac{d}{dt}\{a(t)\}a(t)dt=\frac{a^2(t)}{2}\eval_0^{T_{ext}}=\frac{1}{2}.
        \end{split}
    \end{equation*}
    For the third term, the inner product in the numerator equals $1/2$ in the limit (as just calculated for the second term). The denominator approaches $\infty$, therefore this term is zero in the limit and we get  
    $ERR_{DMD}= 1-2\frac{1}{2}-0=0$.
    
    For the reconstruction error, we have
    $\tilde{\mu}=\lambda(4-p)/2$ and according to Lemma \ref{lem:inconsistencyDMD} the reconstructed dynamics is
\begin{equation*}
    \tilde{\psi}(t)=f\cdot  e^{\tilde{\mu}t}=f\cdot e^{\frac{\lambda(4-p)}{2}t}.
\end{equation*}
The reconstruction error \eqref{eq:DMDreconConti} is 
\begin{equation*}
    ERR_{Rec}^c= \norm{f}^2\int_0^{T_{ext}} \left[a(t)-\hat{a}(t)\right]^2dt.
\end{equation*}
Using the expressions for $a(t)$ and $\hat{a}(t)$ we get
\begin{equation*}
\begin{split}
ERR_{Rec}^c &= \norm{f}^2\int_0^{T_{ext}} \left[a(t)-\hat{a}(t)\right]^2dt\\
&=\norm{f}^2\int_0^{T_{ext}} \left[\left[(1+\lambda(2-p)t)^+\right]^{\frac{1}{2-p}}- e^{\lambda\frac{4-p}{2}t}\right]^2dt\\
&=\norm{f}^2\int_0^{T_{ext}} (1+\lambda(2-p)t)^{\frac{2}{2-p}}dt
+\int_0^{T_{ext}} e^{\lambda(4-p)t}dt\\
&\qquad -2\int_0^{T_{ext}} (1+\lambda(2-p)t)^{\frac{1}{2-p}} e^{\lambda\frac{4-p}{2}t}dt\\
&\ge\norm{f}^2\left[-\frac{1}{\lambda(4-p)}+ \frac{1}{\lambda(4-p)}\left(e^{-\frac{4-p}{2-p}}-1\right)\right]\\
&\qquad-2\norm{f}^2\sqrt[]{\int_0^{T_{ext}} (1+\lambda(2-p)t)^{\frac{2}{2-p}}dt} \sqrt[]{\int_0^{T_{ext}} e^{\lambda(4-p)t}dt}\\
&=\norm{f}^2\left[\sqrt[]{-\frac{1}{\lambda(4-p)}}- \sqrt[]{\frac{1}{\lambda(4-p)}\left(e^{-\frac{4-p}{2-p}}-1\right)}\right]^2\\
&=-\norm{f}^2\frac{1}{\lambda(4-p)}\left[1- \sqrt[]{1-e^{-\frac{4-p}{2-p}}}\right]^2>0.
\end{split}
\end{equation*}
\item According to Lemma \ref{lem:inconsistencyDMD}, we do not obtain solutions for $r>1$.
\end{enumerate}
$\quad$
\end{proof}

\paragraph{Error in extinction time} Other than the inherent error in the decay profile, the extinction time cannot be restored with classical \ac{DMD}. The difference between the analytic decay profile and the approximated exponential function is shown in Fig. \ref{Fig:errorDMDEF}. 
\begin{figure}[phtb]
\centering 
\includegraphics[trim=0 0 0 0, clip,width=0.5\textwidth]{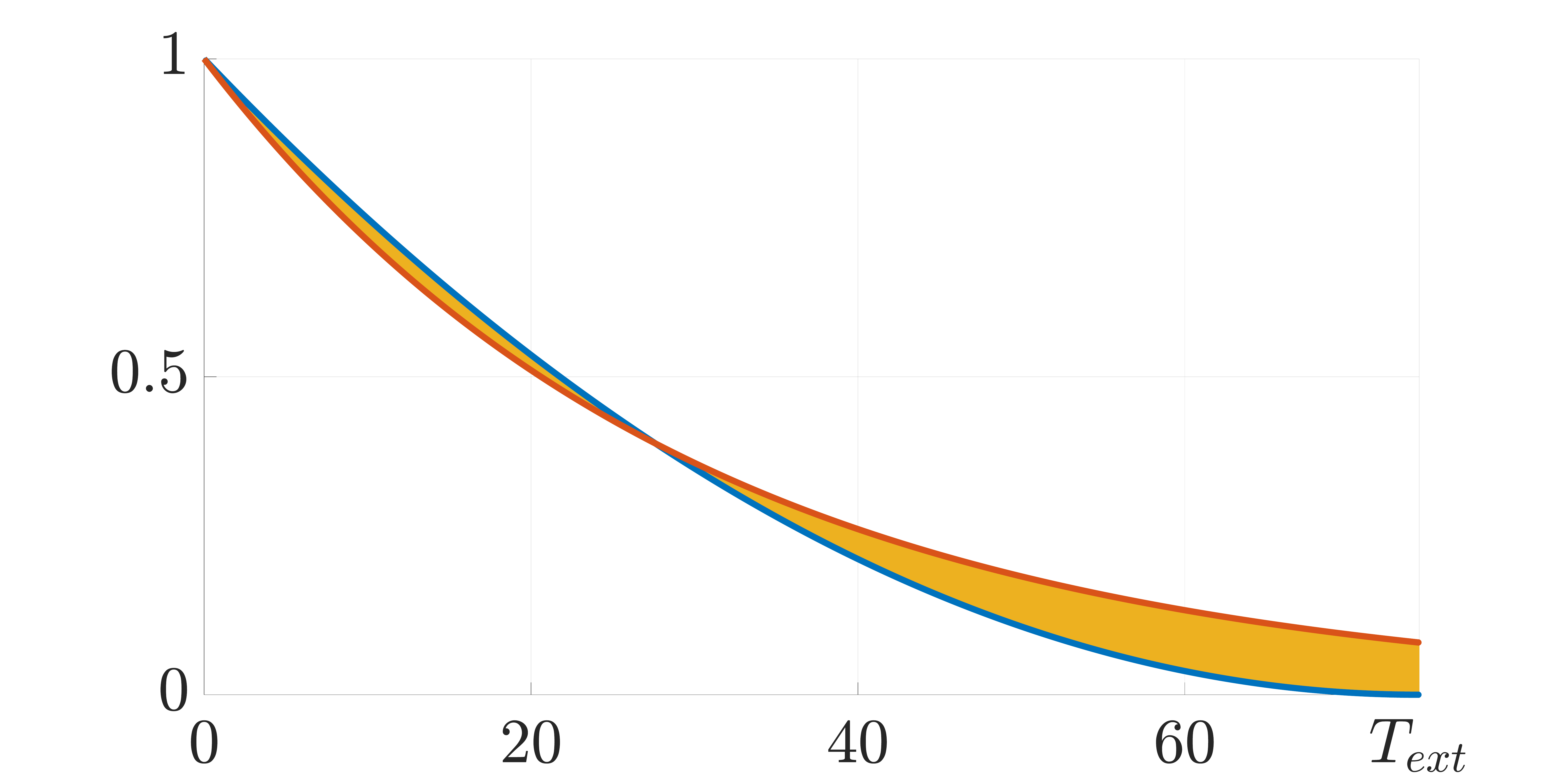}
\caption{{\bf{The \ac{DMD} paradox.}} The blue line is the polynomial decay. The red line depicts the closest exponential function in the sense of $ERR_{DMD}$  \eqref{eq:DMDERR}. Though  $ERR_{DMD} \to 0$ the reconstruction error, $ERR_{Rec}^c$  \eqref{eq:DMDreconConti}, (the orange area) is not.}
\label{Fig:errorDMDEF}
\end{figure}

\begin{corollary}[Properties of the time-continuous setting]
\begin{enumerate}
    \item As $dt \to 0$ the eigenvalue $\mu\to 1$ (coincides with Eq. \eqref{eq:akrecurr}), $a_k \to 1$ for all $k$, and $ERR_{DMD} \to 0$.
    \item For $p=2$, $\tilde{\mu}=\lambda$, as expected in the linear case.
    \item To solve this paradox the vectors $\bm{a}_0^{N-1}$ and $\bm{a}_1^{N}$ must be co-linear. Therefore, the recurrence relating $a_{k+1}$ to $a_k$ should be geometric. 
\end{enumerate}
\end{corollary}

The third part of the corollary implies the \ac{DMD} paradox can be solved by sampling the data non-uniformly. For example, we can sample the dynamics at time points
\begin{equation}\label{eq:adapSampling}
    t_k = \frac{\abs{1-\delta}^{k(2-p)}}{\lambda(2-p)}-\frac{1}{\lambda(2-p)}.
\end{equation}
With this time sampling policy the solution, Eq. \eqref{eq:SpF}, gets the geometric decay form
\begin{equation*}
    \psi_k=a_k\cdot f,\quad a_k=\abs{1-\delta}^k.
\end{equation*}
Evolving  the explicit scheme  \eqref{eq:homoESPflow} with the adaptive step size 
\begin{equation}\label{eq:adapStepSizeClose}
    dt_k=-\frac{\inp{P(\psi_k)}{\psi_k}}{\norm{P(\psi_k)}^2}\delta,
\end{equation} 
yields the solution as in \eqref{eq:seriesAdaptive},
\begin{equation}\label{eq:ESEFsolution}
    \psi_k=a_k\cdot f,\quad a_k=(1-\delta)^k.
\end{equation}
In both cases, the solution converges when $\delta\in(0,2)$ and they are identical when $\delta\in(0,1]$.

\begin{theorem}[Zero reconstruction error for time-rescaled \ac{DMD}]\label{theo:EFadaptiveDMD} 
For the following two cases:
\begin{enumerate}
\item A time-continuous homogeneous flow  \eqref{eq:homoFlow} initialized with an eigenfunction and sampled at time-points as in Eq. \eqref{eq:adapSampling}.
\item An explicit scheme of a homogeneous flow \eqref{eq:homoESPflow} initialized with an eigenfunction with step-size $dt_k$ as in \eqref{eq:adaptiveStepSize}. 
\end{enumerate}
Applying \ac{DMD} perfectly reconstructs the  flow, $ERR^d_{Rec}=0$.
\end{theorem}
\begin{proof} We prove this theorem for the explicit scheme. The proof for adaptive sampling (the first case) is similar, by replacing $(1-\delta)$ with $\abs{1-\delta}$.

The solution of the explicit scheme, following \eqref{eq:ESEFsolution}, is 
\begin{equation*}
    {\bm{\psi_k}}=(1-\delta)^k\cdot f.
\end{equation*}
Consequently, if there is a linear mapping, $A$, from ${\bm{\psi_{k-1}}}$ to ${\bm{\psi_{k}}}$ then the initial condition, $\psi_0=f$, should be its right eigenvector and the corresponding eigenvalue is $1-\delta$. In addition, the requirement for $A$ is to be with minimal rank which is one in this case. Therefore, the linear mapping $A$ is

\begin{equation*}
    A=(1-\delta)\frac{1}{\norm{f}^2}f\cdot f^T.
\end{equation*}
The solution of \ac{DMD} is given in Lemma \ref{lem:inconsistencyDMD} when the eigenvalue is
\begin{equation*}
    \mu=1-\delta.
\end{equation*}
With this solution both \ac{DMD} and reconstruction errors are zero.

$\quad$
\end{proof}

\begin{remark}
One can conclude that the data must exponentially decay to be precisely reconstructed by \ac{DMD}. The decay profile is a result of the homogeneity of the system. Therefore, not only linear systems can be precisely reconstructed, but also one homogeneous ones. And, the adaptive step size policy can be understood as a homogeneity normalization. Namely, this policy mimics a flow derived by a one-homogeneous operator, with a fixed time step.
\end{remark}

\subsection{Proposed general time re-scaling} 
Following the insights gained by the above analysis for the case of initialization with eigenfunctions, we  formulate a general scheme, which applies to any initial condition. Naturally, we cannot expect that a linear approximation as \ac{DMD} will maintain a zero reconstruction error in the general case. However, the proposed solution models much better the nonlinear dynamics. We require that for the case of a single eigenfunction, we obtain the solutions stated in the previous section, yielding perfect reconstruction.

\subsubsection{Prior time re-scaling}
\paragraph{Continuous setting} We introduce first the time re-scaling in the continuous time setting. Specifically, the original flow is factorized by the functional $\lambda_{\psi}^{-1}$, 
\begin{equation}\label{eq:NHF}\tag{\bf TRC}
    \psi_t=G(\psi)=\lambda_{\psi}^{-1}\cdot P(\psi),\quad \psi(0)=f,
\end{equation}
where
\begin{equation}\label{eq:RayleighQuotient}
    \lambda_{\psi}^{-1}=-\inp{P (\psi)}{\psi}/\norm{P (\psi)}^2.
\end{equation}
Note that an eigenfunction of $P$ is an eigenfunction of $G$ with the corresponding eigenvalue $1$. The factorization term $\lambda_{\psi}$ can be viewed as a generalized Rayleigh quotient, as discussed, for instance, in \cite{cohen2018energy}.

\begin{theorem}[Convergence of \eqref{eq:NHF}]\label{theo:decayAndConvergence}
    Let $R(\psi)$ be a convex functional, $-P(\psi)$ be the gradient of $R(\psi)$, and $\psi(t)$ be the solution of \eqref{eq:NHF}. Then,
    \begin{enumerate}
        \item If the zero function belongs to the kernel of the functional $R$ then $R(\psi)$ converges to zero exponentially.
        \item If $P$ is a homogeneous operator and the initial condition, $f$, is an eigenfunction of $P$ (and it is not trivial) the solution is $\psi(t)=f\cdot e^{-t}$.
    \end{enumerate}
\end{theorem}
The proof is in Appendix \ref{sec:proofTheoConv}.


\paragraph{Discrete setting}
The discrete time re-scaling is done by plugging $\lambda^{-1}_{\psi_k} \delta$ as the adaptive step size of the explicit scheme, 
\begin{equation}\label{eq:NHFES}\tag{\bf TRD}
    \psi_{k+1}=\psi_k-P(\psi_k)\cdot\lambda^{-1}_{\psi_k} \delta,\quad \psi_0=f,\delta\in\mathbb{R}.
\end{equation}
This explicit scheme was studied in \cite{cohen2019stable}. The flow is proven to converge to the steady state exponentially under this scheme when $\delta\in(0,2)$. 

This adaptation is possible if we know the operator in advance and we can control the step size as well, which is not always the case. In what follows, we suggest two ways of time re-scaling when the data is already sampled. We propose ways to re-scale the time axis also in cases when the operator or step size are not known.

\subsubsection{Posterior time re-scaling}

To adapt arbitrary data snapshots to our time-rescaled \ac{DMD} we first associate a datum to a certain time point. Then, we interpolate the data and sample at the appropriate time points.
Let us denote by  $dt_k$ the original step size, representing the time difference between sample ${\bm{\psi_k}}$ and sample ${\bm{\psi_{k+1}}}$. The point at time, associated with ${\bm{\psi_k}}$ is $t_k=\sum_{i=0}^{k-1}dt_i$. We rescale the time axis such that $dt_k$ is mapped to $\tilde{dt}_k$ and accordingly $t_k$ to $\tilde{t}_k=\sum_{i=0}^{k-1}\tilde{dt}_i$.
%
%
We do that by reformulating the explicit scheme of the original flow, i.e.
\begin{equation*}
    \begin{split}
        \psi_{k+1}&=\psi_k+P(\psi_k)\cdot dt_k\\
        &=\psi_k-P(\psi_k)\cdot\frac{\inp{P(\psi_k)}{\psi_k}}{\norm{P(\psi_k)}^2} \cdot \left(-\frac{\norm{P(\psi_k)}^2}{\inp{P(\psi_k)}{\psi_k}}dt_k\right).
    \end{split}
\end{equation*}
Thus, the following mapping is obtained
\begin{equation}\label{eq:rescaleHB}
    \tilde{dt}_k = -\frac{\norm{P(\psi_k)}^2}{\inp{P(\psi_k)}{\psi_k}}dt_k
,\quad \tilde{t}_k=\sum_{i=0}^{k-1}\tilde{dt}_i=-\sum_{i=0}^{k-1}\frac{\norm{P(\psi_i)}^2}{\inp{P(\psi_i)}{\psi_i}}dt_k.
\end{equation}

\paragraph{Blind time sampling and flow}
When $P$ and $dt$ are unknown, the step size rescaling, Eq. \eqref{eq:rescaleHB}, can be reformulated as
\begin{equation*}
    \begin{split}
            \tilde{dt}_k &= -\frac{\norm{P(\psi_k)}^2}{\inp{P(\psi_k)}{\psi_k}}dt_k
            =-\frac{\norm{P(\psi_k)}^2}{\inp{P(\psi_k)}{\psi_k}}\frac{dt_k^2}{dt_k}
            =-\frac{\norm{P(\psi_k)\cdot dt_k}^2}{\inp{P(\psi_k)\cdot dt_k}{\psi_k}}.
    \end{split}
\end{equation*}
Since $P(\psi_k)\cdot dt_k = \psi_{k+1}-\psi_{k}$ we get
\begin{equation}\label{eq:rescalingB}
    \tilde{dt}_k = -\frac{\norm{\psi_{k+1}-\psi_k}^2}{\inp{\psi_{k+1}-\psi_k}{\psi_k}},\quad \tilde{t}_k=\sum_{i=0}^{k-1}\tilde{dt}_i=-\sum_{i=0}^{k-1}\frac{\norm{\psi_{i+1}-\psi_{i}}^2}{\inp{\psi_{i+1}-\psi_{i}}{\psi_i}}.
\end{equation}
Our data is now framed within a proper time-rescale, $\{{\bm{\psi_k}},\tilde{t}_k\}_{k=0}^N$. In order to make it suitable for \ac{DMD}, we interpolate the data and sample it with a fixed step size (for instance, by linear interpolation).
We note that when the operator is known (non-blind case), naturally the estimations are better and we obtain less reconstruction errors, as our experiments show in the Results section. 
We note that previous works have proposed non-uniform data sampling
(e.g. \cite{gueniat2015dynamic, leroux2016dynamic}). However, the context and motivation are different.


\begin{remark}
The time rescaling, Eq. \eqref{eq:NHF}, can be seen as \emph{homogeneity normalization}. The new operator $G(\cdot)$ is one-homogeneous for any homogeneity order of $P$, thus, it decays exponentially. In addition, the factorization term, $\lambda^{-1}_\psi$, is not unique to change the homogeneity of the flow to one. It can be any $(2-p)$-homogeneous functional. For example, the term $\norm{\psi}^{2-p}$ changes the homogeneity to one but the convergence (Theorem \ref{theo:decayAndConvergence}) is not clear. 
\end{remark}

\subsection{\ac{DMD} for symmetric operators}
According to Theorem \ref{theo:decayAndConvergence}, one can conclude that a homogeneous functional decays exponentially under the adaptive step size policy. 
It was discussed in \cite{cohen2019stable} that the adaptive step size flow can yield negative eigenvalues but not complex. Thus, it is natural to constrain the \ac{DMD} matrix to be symmetric and real.
This requirement coincides with the analytic expression of nonlinear diffusion in \cite[chapter~3.4]{weickert1998anisotropic}. In this monograph, Weickert investigates the following nonlinear PDE for the continuous and semi-discrete settings,
\begin{equation*}
    \psi_t=\div{\left(D(\nabla \psi)\nabla \psi\right)},    
\end{equation*}
where $D$ is a tensor $\mathbb{R}^{2\times 2}\to \mathbb{R}^{2\times 2}$. Weickert shows that the above flow can be discretized by 
\begin{equation*}
    \psi_{k+1}=\mathcal{A}(\psi_k)\psi_k,
\end{equation*}
where $\mathcal{A}(\cdot)$ is symmetric. Note that this evolution is not linear and the operator $\mathcal{A}$ change with iterations.
However, the \ac{DMD} matrix, $F$ (Algo. \ref{algo:DMD}, State \ref{state:DMDmatrix}), is not limited to be symmetric. If $F$ is symmetric, it can be written as 
\begin{equation} \label{eq:DMDconstraints}
    F=B^TB,
\end{equation}
where $B$ belongs to $\mathbb{C}^{r \times r}$. Then, we can embed this requirement in the standard \ac{DMD} (see Appendix \ref{sec:LRUC}). 
We refer to this as \acl{SDMD} (Algo. \ref{algo:DMD_forced}). This algorithm is identical to the classic \ac{DMD} other than calculating $F$. Restricting $F$ to be symmetric is equivalent to solving
\begin{equation*}
    FXX^T+XX^TF = XY^T+YX^T.
\end{equation*}
For more details we refer the reader to Appendix \ref{sec:LRUC}.

\begin{algorithm} \caption{\acf{SDMD} 
}
\begin{algorithmic}[1]
		\Inputs{Data sequence $\{\bm{\psi_k}\}_0^{N+1}$}.
		\State Repeat the steps detailed in Algorithm \ref{algo:DMD}. The fifth step is changed to: 
		 Find the optimal linear mapping, $F$, between $X$ and $Y$ in the sense of 
		\begin{equation*}
		    \min_{F}\norm{Y-FX}^2_F,\quad s.t.\,\, F=B^TB.
		\end{equation*}
		Equivalently, solve the Sylvester equation 
		\begin{equation*}
            FXX^T+XX^TF = XY^T+YX^T.
        \end{equation*}
		Algorithms for solving this are given in Appendix \ref{sec:LRUC}.
		\Outputs{$$\{\mu_i,\bm{\phi_i},\alpha_i\}_1^r$$}.
    \end{algorithmic}
    \label{algo:DMD_forced}
\end{algorithm}

\subsection{Modes as nonlinear eigenfunctions}
In what follows, we examine the modes of applying \ac{DMD} on the time-rescaled snapshots, as explained above. We attempt to draw a relation between the modes and the eigenfunctions of the operator $P$.  Let us recall, first, the definition of the modes $\{\bm{\phi_i}\}$
\begin{equation*}
    \Phi = \begin{bmatrix}
    \bm{\phi_1}&\cdots&\bm{\phi_r}
    \end{bmatrix}=U_r\begin{bmatrix}
    \bm{w_1}&\cdots&\bm{w_r}
    \end{bmatrix}=U_r\cdot W,
\end{equation*}
where $\{\bm{w_i}\}_{i=1}^r$ is the eigenvector set of the \ac{DMD} matrix, $F$. The mode set $\{\phi_i\}_{i=1}^r$, is an orthonormal set,
\begin{equation*}
    \inp{\phi_i}{\phi_j}=\phi_i^*\phi_j=\left(U_r \bf{w_i}\right)^*U_r {\bf{w_j}}={\bf{w_i}}^*U_r^*U_r {\bf{w_j}}={\bf{w_i}}^*{\bf{w_j}}=\delta_{i,j},
\end{equation*}
where $\delta_{i,j}$ is the Kronecker delta.
In general, the dynamical system reconstruction with \ac{DMD} is given by (Eqs. \eqref{eq:linearApproxA} and \eqref{eq:modeSpectrumCoo})
\begin{equation*}
    A = U_r\cdot F\cdot U_r^*.
\end{equation*}
The modes are eigenvectors of the dynamics matrix, $A$
\begin{equation*}
    A\cdot {\bm{\phi_i}} = U_r\cdot F\cdot U_r^*\cdot U_r{\bm{w_i}}=U_r\cdot F\cdot {\bm{w_i}}=\mu_i U_r {\bm{w_i}}=\mu_i {\bm{\phi_i}}.
\end{equation*}
The eigenvalues correspond to those of the \ac{DMD} matrix.
We focus on the adaptive step-size explicit scheme \eqref{eq:NHFES}. The linear system defined by $A$ is its approximation. We can thus write
\begin{equation*}
    \psi_{k+1}\approx A\psi_{k}.
\end{equation*}
Using \eqref{eq:NHFES} for  $\psi_{k+1}$ we get
\begin{equation*}
    A\psi_{k}\approx \psi_k-\frac{\inp{P(\psi_k)}{P(\psi_k)}}{\norm{P(\psi_k)}^2}P(\psi_k)\delta.
\end{equation*}
This approximation is valid for any input, in particular for a mode $\phi$, 
\begin{equation*}
    \mu\phi=A\phi\approx \phi-\frac{\inp{P(\phi)}{\phi}}{\norm{P( \phi)}^2} P(\phi)\cdot \delta.
\end{equation*}
Rearranging this equation yields
\begin{equation}\label{eq:phiPrelation}
    P(\phi)\approx \left[\frac{1-\mu}{\delta}\frac{\norm{P( \phi)}^2}{\inp{P(\phi)}{\phi}}\right]\phi.
\end{equation}
The expression in the brackets is a (real) number. 
Thus, we can conclude that \emph{the orthonormal mode set  approximates a set of nonlinear eigenfunctions of the operator $P$.} This introduces an interesting new relation between \ac{DMD} and nonlinear spectral theory \cite{book2008nonlinearSpectral,gilboa2018nonlinear}.

\subsection{Eigenvalue evaluation}
The interpretation of the nonlinear eigenvalue, $\lambda$, is twofold. The first, the eigenvalue is the value of the generalized Rayleigh quotient at a local extremum. The second, the eigenvalue dictates the decay profile as well as the extinction time. Thus, the nonlinear eigenvalues can be approximated accordingly.

\paragraph{Generalized Rayleigh quotient} An eigenvalue approximation according to the eigenvector is straightforward. The approximated relation between the operator $P$ and the eigenvector $\phi$ is defined in Eq. \eqref{eq:phiPrelation}. Thus, the nonlinear eigenvalue is the coefficient of the mode $\phi$, i.e.
\begin{equation}\label{eq:laPhi}
    \lambda_{\phi} = \frac{1-\mu}{\delta}\frac{\norm{P(\phi)}^2}{\inp{P(\phi)}{\phi}}.
\end{equation}
The subscript denotes that the eigenvalue is based on the eigenvector $\phi$.
Notice that if the initial condition $f$ is an eigenfunction, the solution of \ac{DMD} is given by Eq. \eqref{eq:lemmaFirstResult}. Then, the eigenvalue, $\mu$, is equal to $1-\delta$ and the nonlinear eigenvalue, $\lambda_\phi$, is equal to the generalized Rayleigh quotient, which is precisely the eigenvalue $\lambda$.


\paragraph{Decay profile} As discussed above, the adaptive step size policy causes the flow to decay similarly to a one homogeneous flow. Thus, we can compare between the decay profiles. We compare the decay profile function at time $t_k$, $a(t_k)$ to the exponential attenuation by the eigenvalue $\mu^k$. We propose here to compute the  eigenvalue by minimizing,
\begin{equation}\label{eq:optimizationLam}
    E(\lambda) = \sum_{k=0}^N\left(a(t_k)^
        {2-p}-\mu^{k(2-p)}\right)^2= \sum_{k=0}^N\left[1+\lambda_\mu(2-p)t_k-\mu^{k(2-p)}\right]^2.
\end{equation}
Solving $\partial_\lambda E(\lambda)=0$ yields
\begin{equation}\label{eq:laMu}
    \lambda_{\mu}=\frac{\sum_{k=0}^N\mu^{k(2-p)}t_k - \sum_{k=0}^Nt_k}{(2-p)\sum_{k=0}^Nt_k^2}.
\end{equation}
The subscript denotes that the eigenvalue is based on $\mu$. 
Note that, if the flow is initialized with an eigenfunction and sampled at time points $t_k$ as in Eq. \eqref{eq:adapSampling} then $\mu=\abs{1-\delta}$. Using \eqref{eq:optimizationLam} we can solve for $E(\lambda)$ by, 
\begin{equation*}
\begin{split}
    E(\lambda) &= \sum_{k=0}^N\left[1+\lambda_\mu(2-p)t_k-\mu^{k(2-p)}\right]^2\\
    &=\sum_{k=0}^N\left[1+\lambda_\mu(2-p)\left(\frac{\abs{1-\delta}^{k(2-p)}}{\lambda(2-p)}-\frac{1}{\lambda(2-p)}\right)-\abs{1-\delta}^{k(2-p)}\right]^2\\
    &=\sum_{k=0}^N\left[1-\frac{\lambda_\mu}{\lambda}+\frac{\lambda_\mu}{\lambda}\abs{1-\delta}^{k(2-p)}-\abs{1-\delta}^{k(2-p)}\right]^2\\
    &=\sum_{k=0}^N\left[\left(1-\frac{\lambda_\mu}{\lambda}\right)+\left(\frac{\lambda_\mu}{\lambda}-1\right)\abs{1-\delta}^{k(2-p)}\right]^2\\
    &=\left(1-\frac{\lambda_\mu}{\lambda}\right)^2\sum_{k=0}^N\left[1-\abs{1-\delta}^{k(2-p)}\right]^2
\end{split}
\end{equation*}
Therefore, the optimal $\lambda_\mu$ is $\lambda$.

\subsection{\acf{OrthoNS}}\label{sec:pDecoUseDMD} 
We now summarize the above results and methods into a simple coherent analysis and synthesis framework for homogeneous gradient flows, based on \ac{DMD}.
Two versions of the mode decomposition are given in Algorithms \ref{algo:OrthoNS} and \ref{algo:BlindOrthoNS}.
We distinguish between two scenarios. The first, when the operator is known and the step size is controllable. In that case, we refer the reader to Algorithm \ref{algo:OrthoNS}.
\begin{algorithm}[phtb!] \caption{\acl{OrthoNS}}
\begin{algorithmic}[1]
        \Inputs{Given a dynamical system Eq. \eqref{eq:homoFlow}}
        \State \label{state:DataGathring}Evolve the solution of Eq. \eqref{eq:homoFlow} explicitly by Eq. \eqref{eq:homoESPflow} or sample the dynamics, where the step size $dt_k$ is given by
		\begin{equation*}
		    dt_k=\frac{\inp{P(\psi_k)}{\psi_k}}{\norm{P(\psi_k)}^2}\cdot \delta, \quad \delta \in (0,2).
		\end{equation*}
		And we get the data sequence, $\{{\bm{\psi_k}}\}_{k=0}^{N}$, homogeneously normalized.
		\State \label{state:SDMD}Apply the \acl{SDMD} Algorithm \ref{algo:DMD_forced}. 
		The result is $\{\mu_i,{\bm{\phi_i},\alpha_i}\}$.
		\State \label{State:nonlinearEigenvalue}Relate the eigenvalues $\{\mu_i\}_{i=1}^r$ or the modes $\{{\bm{\phi_i}}\}_{i=1}^r$ to the nonlinear eigenvalues $\{\lambda_i\}_{i=1}^r$ with Eqs. \eqref{eq:laPhi} or \eqref{eq:laMu}, accordingly.

		\Outputs{$$\{\lambda_i,\bm{\phi_i},\alpha_i\}_1^r$$}
    \end{algorithmic}
    \label{algo:OrthoNS}
\end{algorithm}
For a data provided beforehand the posterior rescaling is called for, detailed in Algorithm \ref{algo:BlindOrthoNS}.

\begin{algorithm}[phtb!] \caption{Posterior \acl{OrthoNS}}
\begin{algorithmic}[1]
    \Inputs{Given the data sequence
    $\{{\bm{\phi_k}}\}_{k=0}^N$.}
    \If{The operator and the sample times are known}
    \State Rescale the time axis according to Eq. \eqref{eq:rescaleHB}.
    \Else
    \State Rescale the time axis according to Eq. \eqref{eq:rescalingB}.
    \EndIf
    	\State \label{state:BlindDataGathring}Interpolate the data according to the new time axis at fixed step size. Then, we get a new sequence of data, homogeneously normalized, $\{{\bm{\bar{\psi}_k}}\}_{k=0}^N$.
		\State \label{state:BlindSDMD}Apply the \acl{SDMD} Algorithm \ref{algo:DMD_forced}. 
		The result is $\{\mu_i,{\bm{\phi_i}},\alpha_i\}$.
		\State \label{State:BlindnonlinearEigenvalue}Relate the eigenvalues $\{\mu_i\}_{i=1}^r$ or the modes $\{{\bm{\phi_i}}\}_{i=1}^r$ to the nonlinear eigenvalues $\{\lambda_i\}_{i=1}^r$ with Eqs. \eqref{eq:laPhi} or \eqref{eq:laMu}, accordingly.
		\Outputs{$$\{\lambda_i,\bm{\phi_i},\alpha_i\}_1^r$$}
    \end{algorithmic}
    \label{algo:BlindOrthoNS}
\end{algorithm}

\paragraph{A dynamical system reconstruction}
Here we propose a simple alternative reconstruction of the flow, which can be expressed analytically using the modes. Each mode is related to an eigenfunction with its corresponding eigenvalue and decay profile.  A linear approximation of the flow \eqref{eq:homoFlow} is expressed as a weighted summation of the orthonormal modes,
\begin{equation}\label{eq:approPFlowSolution}
    \hat{\psi}(t)=\sum_{i=1}^r\alpha_i\phi_i\left[\left(1+\lambda_i(2-p)t\right)^+\right]^{\frac{1}{2-p}}.
\end{equation}

\begin{definition}[\acf{OrthoNS}]\label{def:dispDec}
The \ac{OrthoNS} of an image $f$ and a homogneous operator $P$ is the set  $\{\phi_i,T_i,\alpha_i\}_{i=1}^r$ (Algorithms \ref{algo:OrthoNS}, \ref{algo:BlindOrthoNS}).
\end{definition}
With this definition, we can reconstruct the initial condition, $f$, as $\hat{f}=\hat{\psi}(0)=\sum_i \alpha_i\phi_i$. The error, $\norm{f-\hat{f}}$, depends on the dimensionality, $r$ \cite{gavish2014optimal}.

\begin{theorem}[Parseval Identity]
    The \ac{OrthoNS} admits the Parseval's identity with respect to $\hat{f}$.
\end{theorem}
\begin{proof}
\begin{equation*}
    \begin{split}
        \norm{\hat{f}}^2=\hat{f}^T\hat{f}=\left(\sum_{i=1}^r\alpha_i \phi_i\right)^T \sum_{j=1}^r\alpha_j \phi_j=\sum_{i=1}^r\sum_{j=1}^{r}\alpha_i\alpha_j\inp{\phi_i}{\phi_j}=\sum_{i=1}^r\alpha_i^2
    \end{split}
\end{equation*}
\end{proof}

\begin{definition}[Spectrum]\label{def:dispSpe}
The \ac{OrthoNS} spectrum of a function  $f$ is the set $\{T_i, \abs{\alpha_i}^2\}_{i=1}^r$, where $T_i$ and $\alpha_i$ are the extinction time and the coefficient of the mode $\phi_i$, respectively.
\end{definition}

\begin{definition}[Filtering]\label{def:dispFil}
Given a filter $h\in \mathbb{R}^r$, the \ac{OrthoNS} filtering is,
\begin{equation}\label{eq:dispFil}
    f_h=\sum_{i=1}^r{\bm{\phi_i}}{\alpha}_ih_i.
\end{equation}
This yields an amplification or attenuation of the modes.
\end{definition}

We note that the coefficients $\{\alpha_i\}_{i=1}^r$ are optimal with respect to the initial condition, $f$. Hence, the reconstruction accuracy is excellent near $t \approx 0$ but may detoriorate as time increases (as seen in our experiments). We currently examine ways to improve this.

\section{Results}\label{sec:results}
In this section we show numerical implementations of the theory presented above. We choose the operator $P$ to be the $p$-Laplacian operator where $p\in(1,2)$, and assume Neumann boundary conditions. We follow the numerical implementation as detailed in  \cite{cohen2020Introducing,cohen2019stable}. The eigenvectors presented here were generated numerically by the algorithm of \cite{cohen2018energy}.
In this section, we show results of the (posterior) time rescaling 
(Sec. \ref{subsec:ResHN}); the dynamical system reconstruction (Sec. \ref{subsec:ResDSR}); and image analysis and processing with \ac{OrthoNS} (Sec. \ref{subsec:SAP}). We compare the running time of \ac{OrthoNS} and nonlinear $p$-decomposition \cite{cohen2020Introducing} as a function of image size in Sec. \ref{subsec:ResRT}. Finally, we examine the robustness to noise of \ac{SDMD} compared to \ac{DMD} \cite{schmid2010dynamic}, tls\ac{DMD} \cite{hemati2017biasing} and fb\ac{DMD} \cite{dawson2016characterizing} in Appendix \ref{sec:SDMDResults}. All experiments were run on an i7-8700k CPU machine @ 3.70 GHz, 64 GB RAM.

\subsection{\ac{DMD} with time rescaling}\label{subsec:ResHN}
We demonstrate the time rescaling techniques and show quantitative results of Theorem \ref{theo:EFadaptiveDMD},
along results of Algorithms \ref{algo:OrthoNS} and \ref{algo:BlindOrthoNS}.
In Tables \ref{tab:HN} and \ref{tab:BHN} we show results of an experiment in which
the $p-$Laplacian flow \eqref{eq:pFlow} is evolved, initialized with an eigenfunction,  $p=1.5,\,\lambda = -0.0269,\, \norm{f}^2=249.1$, $T_{ext}=74.3$. We set $r=20$ and, as expected, the system is close to singularity. We list the five most significant modes sorted according to their coefficients. In both cases, a single significant mode is obtained which perfectly matches the parameters of the initial condition (power, eigenvalue and extinction time).
\begin{table}[phtb!]
\begin{minipage}[phtb!]{0.475\textwidth}
\centering
 \begin{tabular}{||c c c||} 
 \hline
 $\alpha^2$ 
 & T & $\lambda$ \\ [0.5ex] 
 \hline\hline
      249.1
      & 74.3&  -2.69e-2\\
 \hline
 5.8e-07
 & 0.2& -1e1 \\
 \hline
 8.2e-17
 & 0.3e-3& -7.22e3\\
 \hline
4.9e-23
&2.6e-10& -7.56e9\\
\hline
3.2e-23
& 5.8e-10& -3.44e9\\ [1ex] 
\hline
\end{tabular}
\caption{{\bf{\ac{OrthoNS} of an eigenfunction.}} The result of Algo. \ref{algo:OrthoNS}, $P$ is the $p-$Laplacian and $f$ is an eigenfunction.}
\label{tab:HN}
\end{minipage}
$\quad$
\begin{minipage}[phtb!]{0.475\textwidth}
\centering
 \begin{tabular}{||c c c||} 
 \hline
 $\alpha^2$ 
 & T & $\lambda$ \\ [0.5ex] 
 \hline\hline
    249.1 
    &  74.3& -2.69e-2\\
 \hline
 5.8e-07
 & 0.2& -1e1 \\
 \hline
 2.2e-17
 & 0.2e-3& -1.15e4\\
 \hline
5e-23
& 2.6e-10& -7.74e9\\
\hline
4.7e-23
&  4.8e-09& -4.13e8\\ [1ex] 
\hline
\end{tabular}
\caption{{\bf{Posterior \ac{OrthoNS} of an eigenfunction.}} The result of Algo. \ref{algo:BlindOrthoNS} for the same case as in Table \ref{tab:HN}.}
\label{tab:BHN}
\end{minipage}
\end{table}

\subsection{\ac{OrthoNS} and dynamical system reconstruction}\label{subsec:ResDSR}
We examine the main modes and the reconstruction (Eq. \eqref{eq:approPFlowSolution}) on a simple 1D pulse signal. The experiment is repeated twice, with $p=1.01$ and with  $p=1.5$. 
In both cases, the dimensionality is set to $r=5$.
In Fig. \ref{fig:modes101} the initial condition reconstruction and the \ac{OrthoNS} modes (Algorith \ref{algo:OrthoNS}) are shown for $p=1.01$. 
\begin{figure}[phtb!]
    \centering
    \includegraphics[trim=350 0 300 0, clip,width=1\textwidth,valign=c]{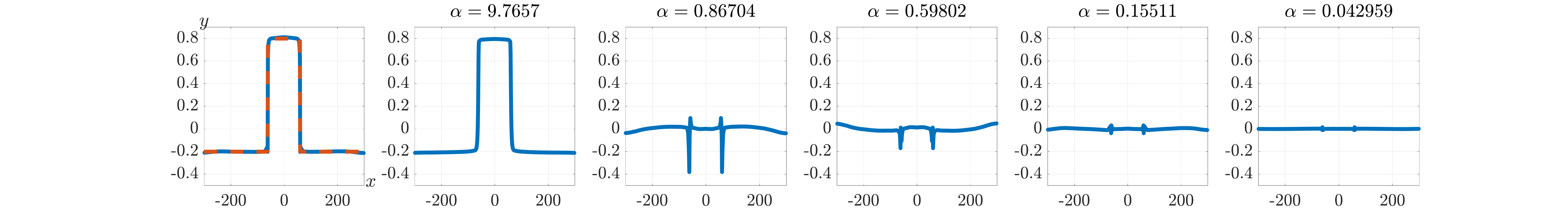}
    \caption{{\bf{\ac{OrthoNS}} of a pulse}, when $p=1.01$ and $r=5$. Left plot, the pulse is in red, the reconstruction is blue. From the second left to the right, the five main \ac{OrthoNS} modes $\phi_i$, sorted according to $\alpha$.}
    \label{fig:modes101}
\end{figure}
On the left, the initial condition is in red (dashed) and its reconstruction is in blue. The five plots on the right are the modes, in decreasing order with respect to $\alpha$ (their ``power'').
A pulse is very close to an eigenfunction of the $p$-Laplacian for $p\to 1$ (total variation). As expected, the first mode is very dominant, containing most of the signal's power.

In Fig. \ref{fig:modes3D101} we illustrate how each mode evolves separately, according to its decay profile, based on the nonlinear eigenvalue $\lambda_{\mu}$ (Eq. \eqref{eq:laMu}). Fig. \ref{fig:flowRecDiffp101} shows the ground truth flow, an approximation of the flow by a linear combination of the modes (Eq. \eqref{eq:approPFlowSolution}), and the difference between them.

\begin{figure}[phtb!]
    \centering
    \includegraphics[trim=350 0 300 0, clip,width=1\textwidth,valign=c]{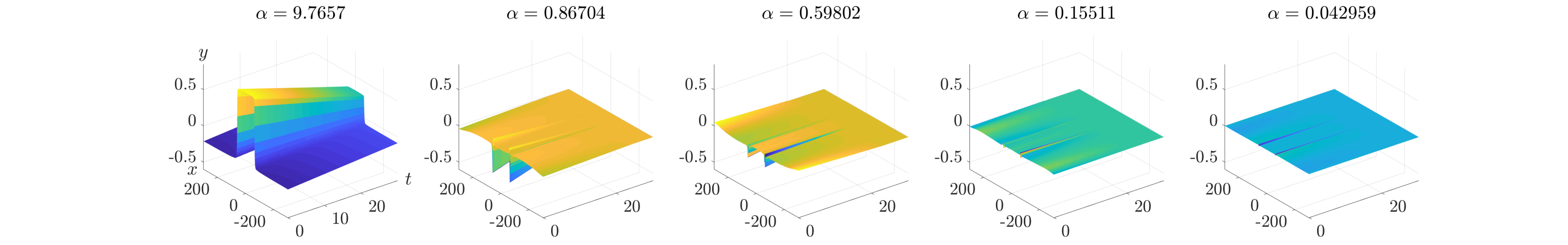}
\label{subfig:modes3D101}
    \caption{{\bf{Dynamics reconstruction.}} The change of the modes (from Fig. \ref{fig:modes101}) over the time according to the approximated decay profile, Eq. \eqref{eq:approPFlowSolution}. Five significant modes for pulse smoothing when $p=1.01$.}
    \label{fig:modes3D101}
\end{figure}

\begin{figure}[phtb!]
\centering
\captionsetup[subfigure]{justification=centering}
\subfloat[Diffusion $p=1.01$]{
\includegraphics[trim=0 0 0 0, clip,width=0.33\textwidth,valign=c]{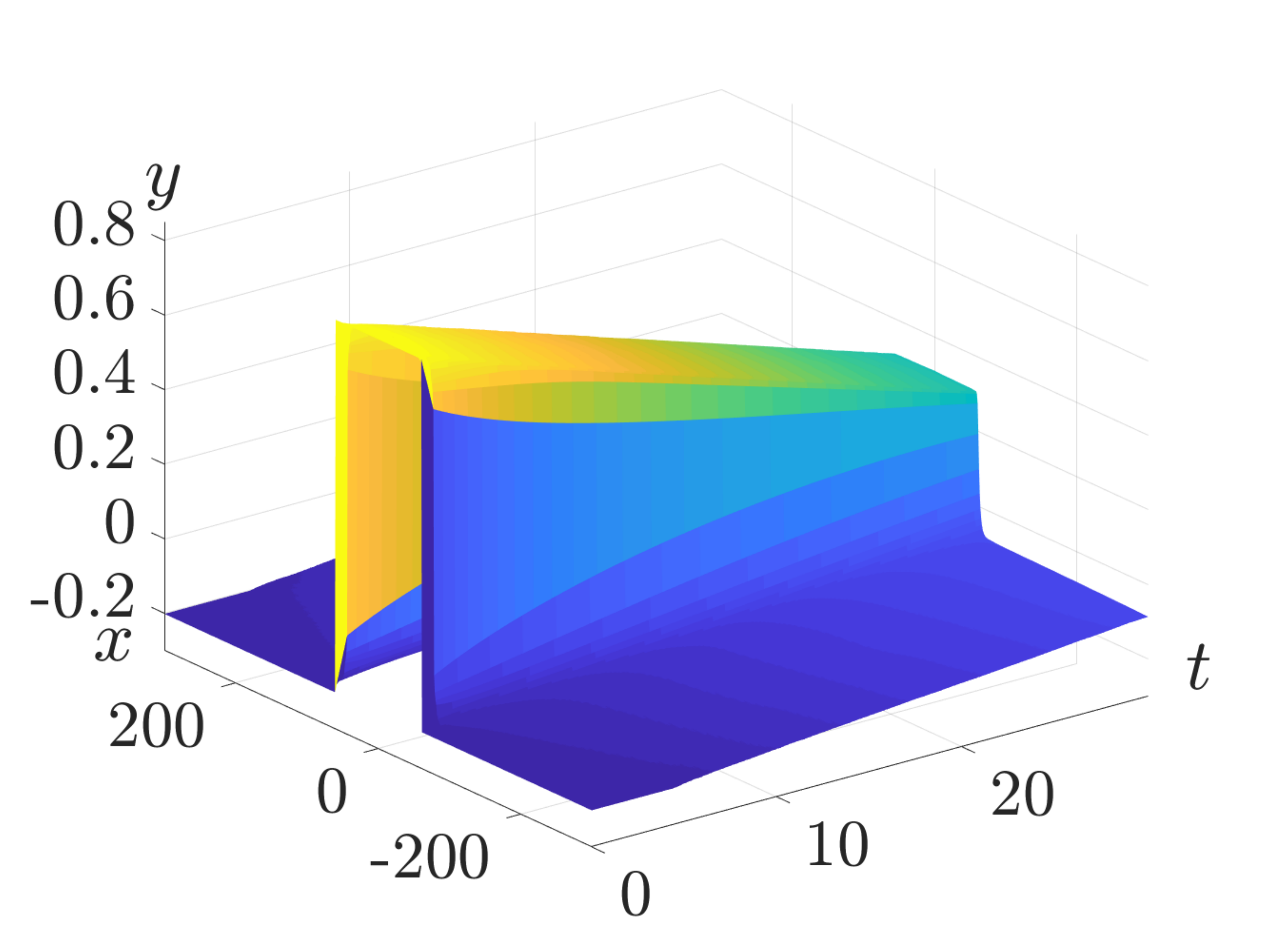}
    \label{subfig:uT101}
    }
    \subfloat[Reconstructed flow $p=1.01$]{
    \includegraphics[trim=0 0 0 0, clip,width=0.33\textwidth,valign=c]{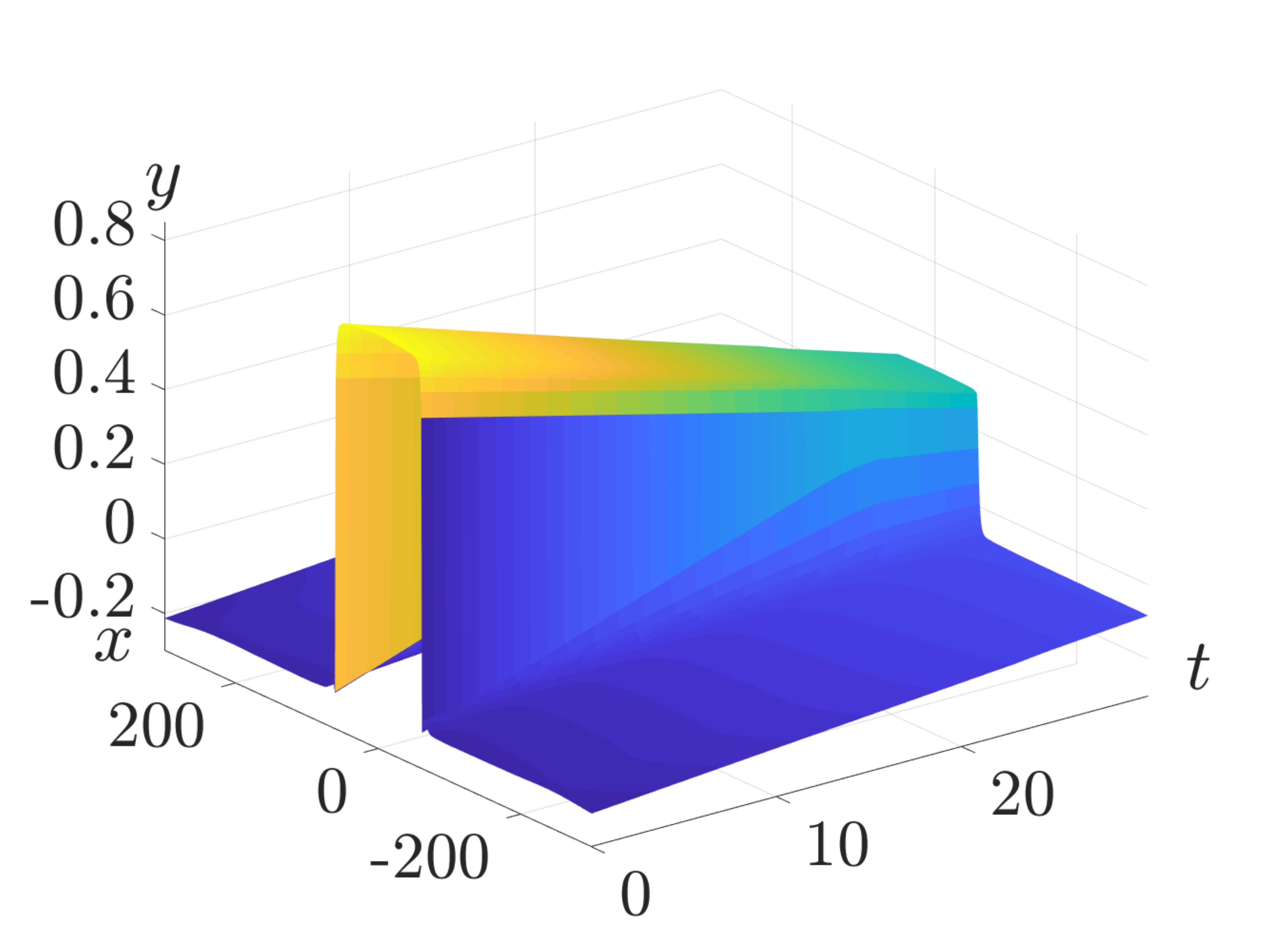}
    \label{subfig:uTcons101}
    }
    \subfloat[Difference]{
    \includegraphics[trim=0 0 0 0, clip,width=0.33\textwidth,valign=c]{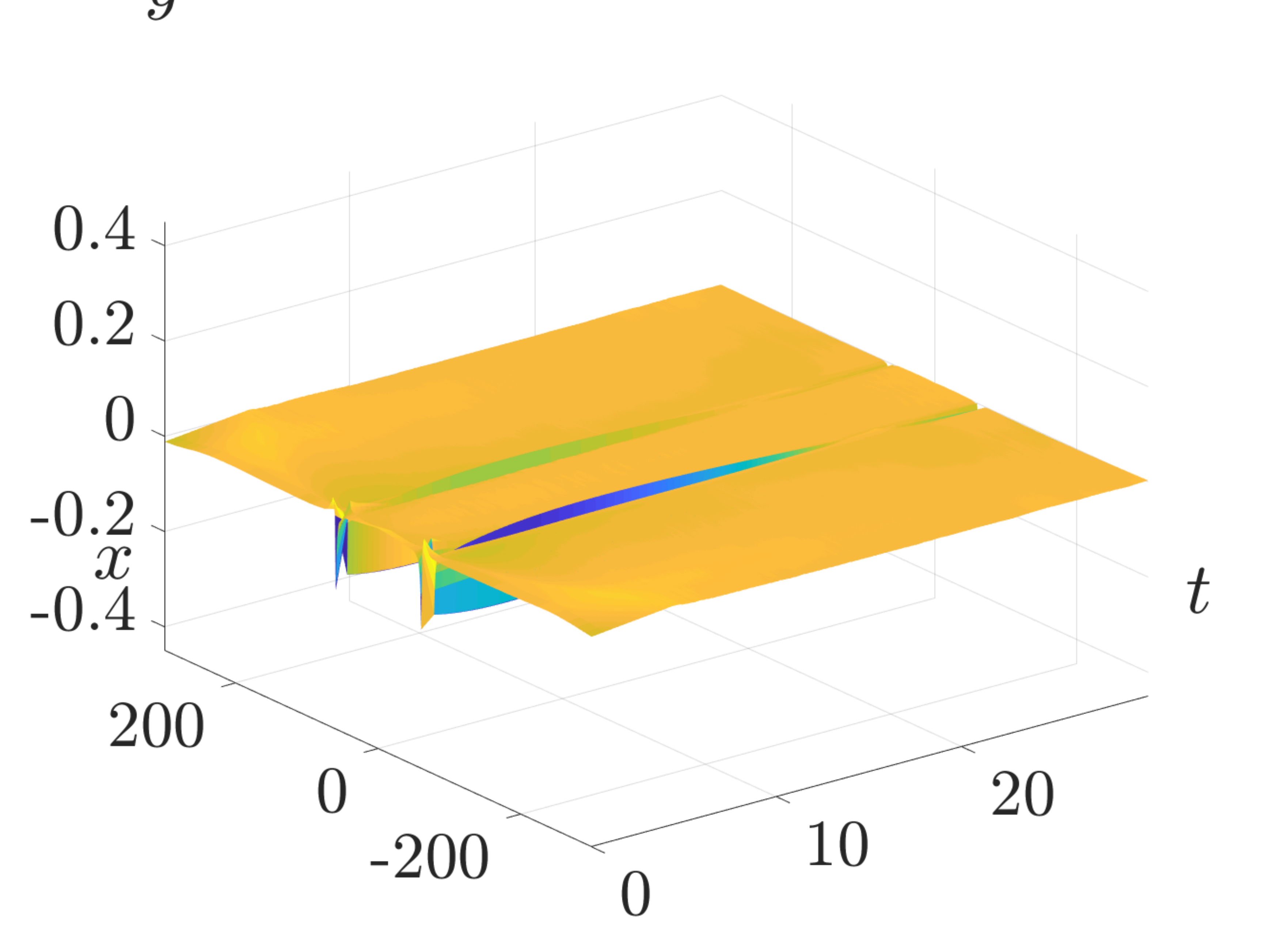}
    \label{subfig:uTconsDiff101}
    }
    \caption{{\bf{Dynamics reconstruction.}} Fig. \ref{subfig:uT101} is the GT of $\psi_t=\Delta_p \psi$ initialized with a pulse. Fig. \ref{subfig:uTcons101} the reconstructed pulse smoothing by Eq. \eqref{eq:approPFlowSolution}. Fig. \ref{subfig:uTconsDiff101} the difference between these two signals.}
    \label{fig:flowRecDiffp101}
\end{figure}

We repeat this experiment for $p=1.5$ with the same initial condition, as shown in Figs. \ref{fig:modes150}, \ref{fig:modeds3D150}, and  \ref{fig:flowRecDiffp150}.
In this case, we obtain less accuracy in the reconstruction with the same number of modes, as the modes are smoother and do not resemble a pulse. As expected, the main mode, in terms of $\alpha$, is not that dominant, where power  is scattered more evenly between modes compared to the case with $p=1.01$. 
\begin{figure}[phtb!]
    \centering
    \includegraphics[trim=350 0 300 0, clip,width=1\textwidth,valign=c]{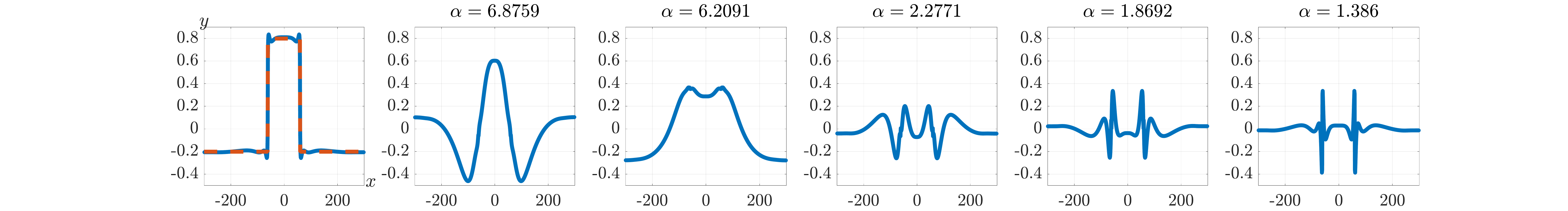}
    \caption{{\bf{\ac{OrthoNS}} of a pulse}, when $p=1.5$ and $r=5$. In the left plot, the pulse is the red line, the reconstruction one is the blue line. From the second left to the right, the five modes are sorted according to their amplitude.}
    \label{fig:modes150}
\end{figure}

%

\begin{figure}[phtb!]
    \centering
    \includegraphics[trim=350 0 280 0, clip,width=1\textwidth,valign=c]{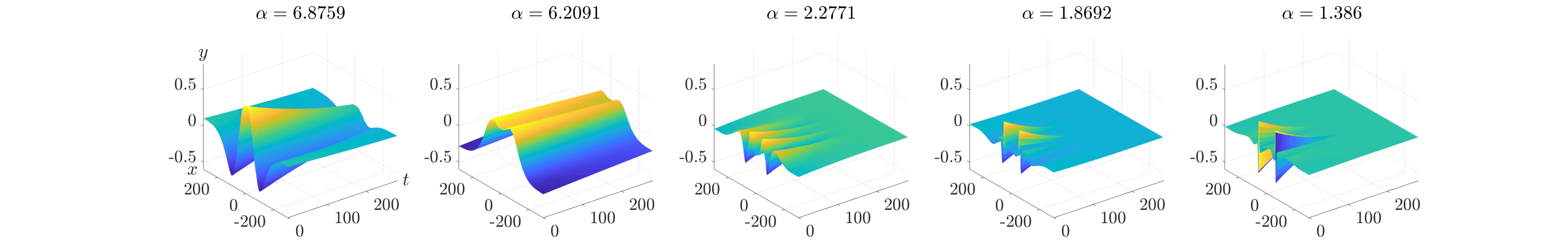}
    \caption{{\bf{Dynamics reconstruction.}} The change of the modes (from Fig. \ref{fig:modes150}) over the time according to the approximated decay profile, Eq. \eqref{eq:approPFlowSolution}. Five significant modes for pulse smoothing when $p=1.5$.}
    \label{fig:modeds3D150}
\end{figure}

\begin{figure}[phtb!]
\centering
\captionsetup[subfigure]{justification=centering}
\subfloat[Diffusion $p=1.5$]{
\includegraphics[trim=0 0 0 0, clip,width=0.33\textwidth,valign=c]{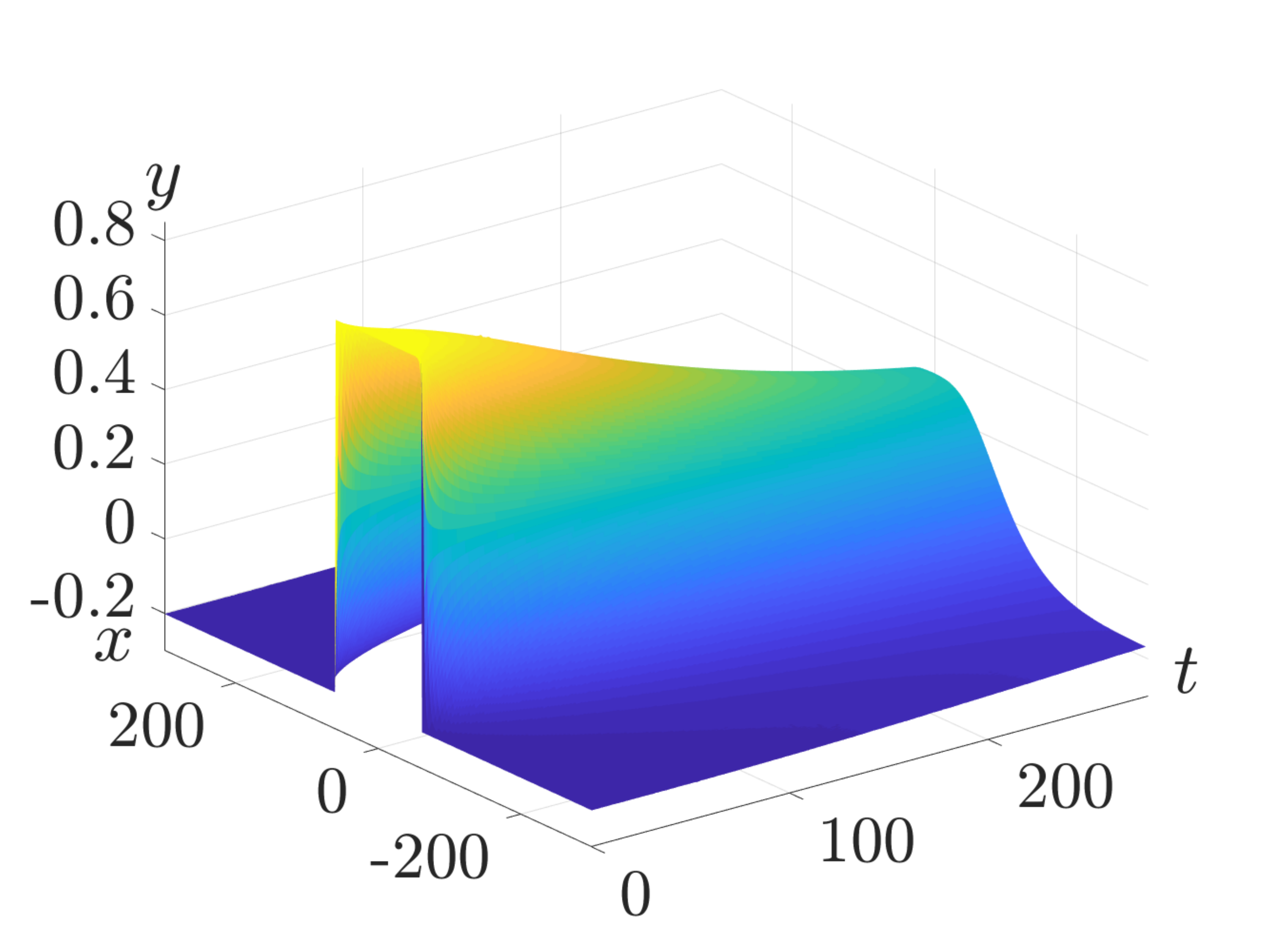}
    \label{subfig:uT150}
    }
    \subfloat[Reconstructed flow $p=1.5$]{
    \includegraphics[trim=0 0 0 0, clip,width=0.33\textwidth,valign=c]{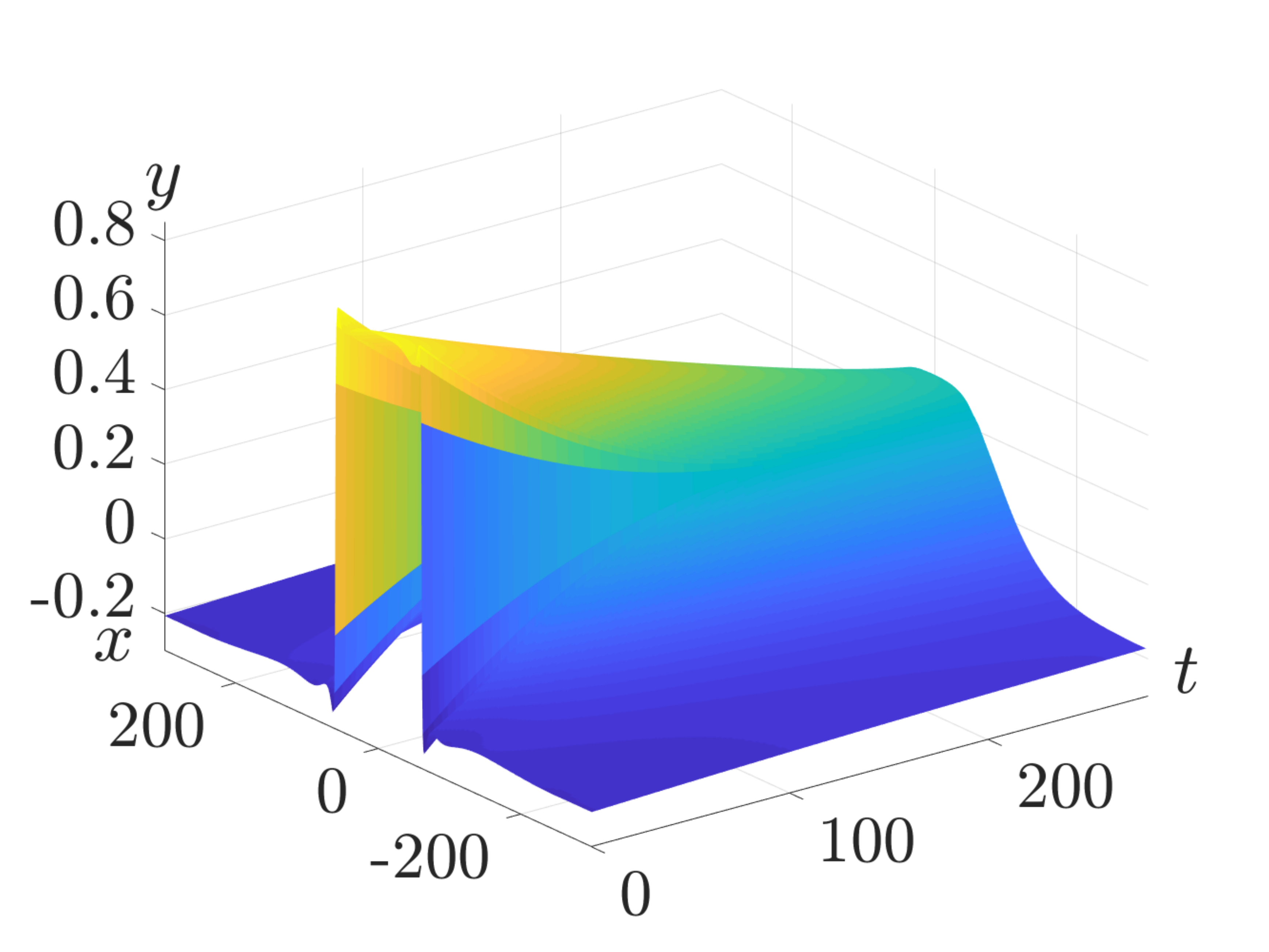}
    \label{subfig:uTcons150}
    }
    \subfloat[Difference]{
    \includegraphics[trim=0 0 0 0, clip,width=0.33\textwidth,valign=c]{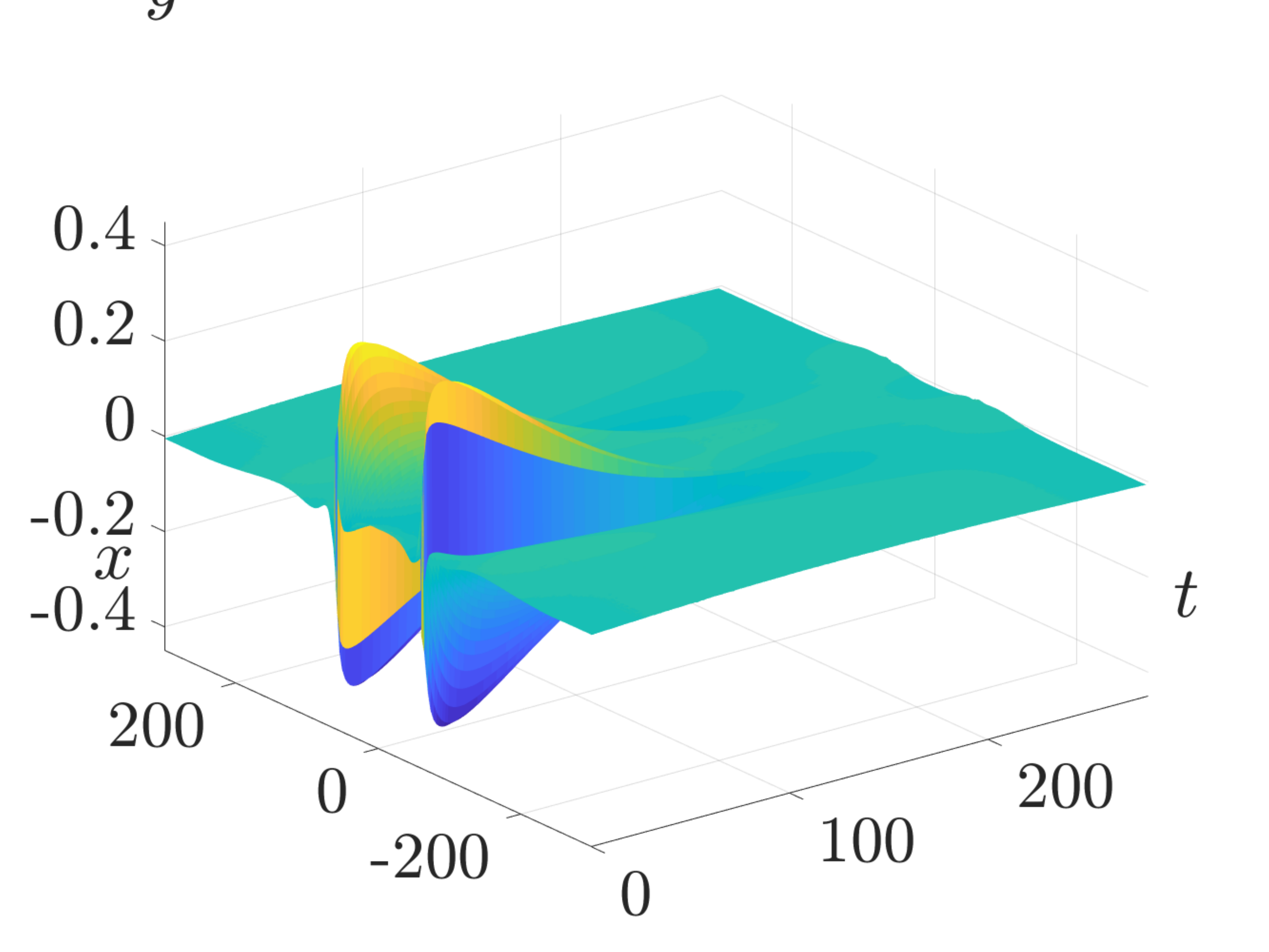}
    \label{subfig:uTconsDiff150}
    }
    \caption{{\bf{Dynamics reconstruction.}} Fig. \ref{subfig:uT150} is the GT of $\psi_t=\Delta_p \psi$ initialized with a pulse. Fig. \ref{subfig:uTcons150} the reconstructed pulse smoothing by Eq. \eqref{eq:approPFlowSolution}. Fig. \ref{subfig:uTconsDiff150} the difference between these two signals.}
    \label{fig:flowRecDiffp150}
\end{figure}

Note, that the error between the flow and the reconstruction is almost zero for $t=0$ (Figs. \ref{subfig:uTconsDiff101} and \ref{subfig:uTconsDiff150}). The reason is that the coefficients $\{\alpha_i\}_{i=1}^r$ are optimized based on the initial condition. However, the error increases fast and then diminishes with time. We plan to investigate in the future alternative optimization models to improve the reconstruction of the entire flow.

\subsection{Signal analysis and processing via \ac{OrthoNS}}\label{subsec:SAP}
We show here that \ac{OrthoNS} precisely and quickly distinguishes between different parts of data. This distinction is done by filtering, as defined in Def. \ref{def:dispFil}. It is demonstrated here by denoising artificial and natural images.

\subsubsection{Denoising an eigenfunction with noise}
In Fig. \ref{Fig:filtering} an eigenfunction of the $p$-Laplacian with additive Gaussian noise ($N\sim\mathcal{N}(0,0.3)$) is denoised.
We apply \ac{OrthoNS} and get the spectrum as shown in Fig. \ref{subfig:filtering_discSpec}. By filtering out the blue component from the spectrum, we restore the eigenfunction (Fig. \ref{subfig:filtering_EF}) and the noise image (Fig. \ref{subfig:filtering_noise}).

\begin{figure}[phtb!]
\centering
\begin{minipage}{1\textwidth}
\centering
\captionsetup[subfigure]{justification=centering}
\subfloat[A $p$-Laplacian eigenfunction $p=1.5$, $\lambda=0.0269$]
 { 
\includegraphics[trim=0 -10 0 -10, clip,width=0.1875\textwidth,valign=c]{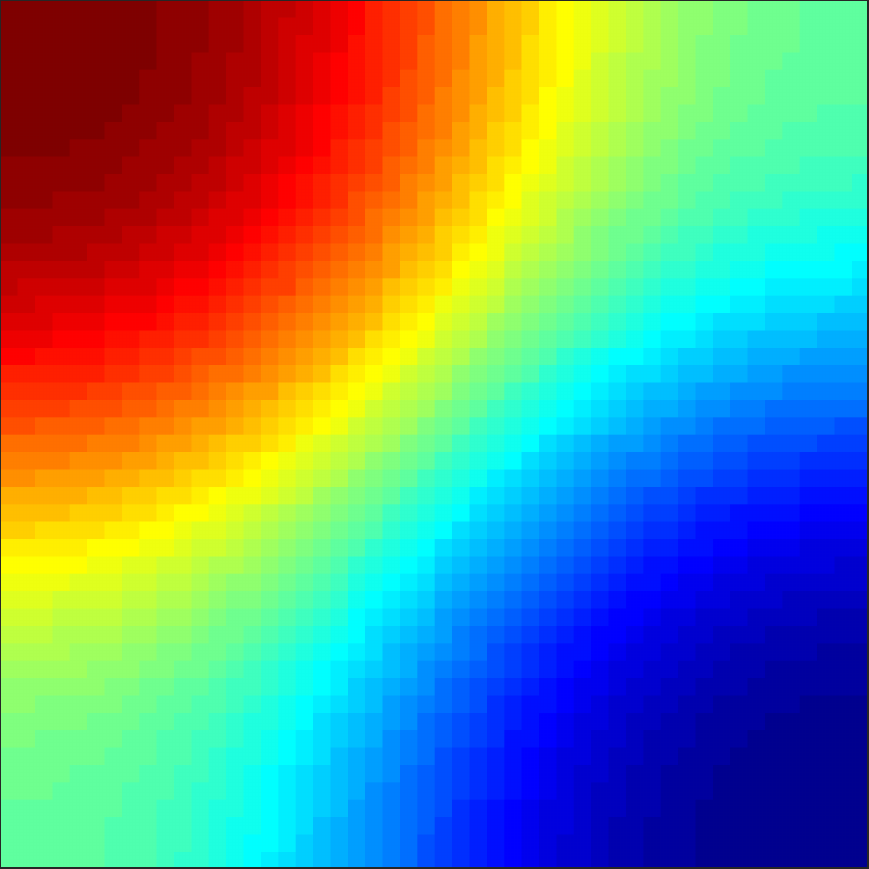}
\label{subfig:origEN}}$\quad+\quad$
\subfloat[Noise Image]
 { 
\includegraphics[trim=0 -10 0 -10, clip,width=0.1875\textwidth,valign=c]{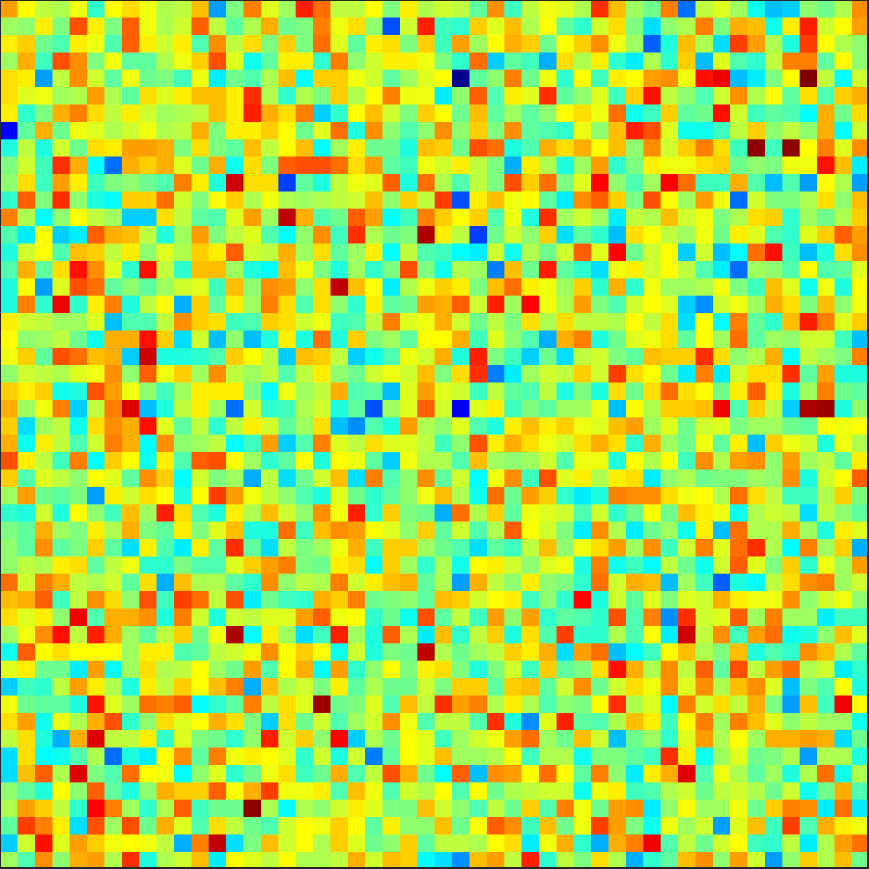}
\label{subfig:noiseImage}}$\quad=\quad$
\subfloat[An eigenfunction with additional noise $PSNR=10.5dB$]
 { 
\includegraphics[trim=0 -10 0 -10, clip,width=0.1875\textwidth,valign=c]{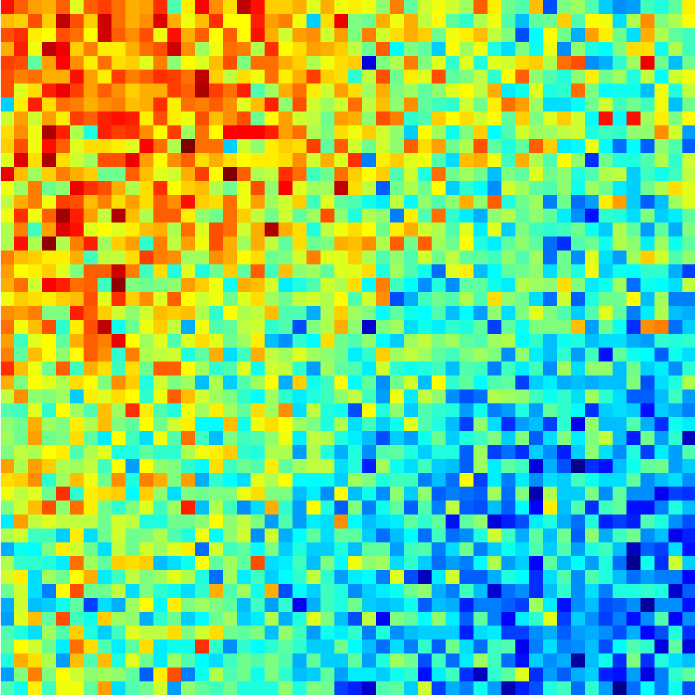}
\label{subfig:filtering_f}}\\
\subfloat[Filtered noise]{
\includegraphics[trim=0 -10 0 -10, clip,width=0.1875\textwidth,valign=c]{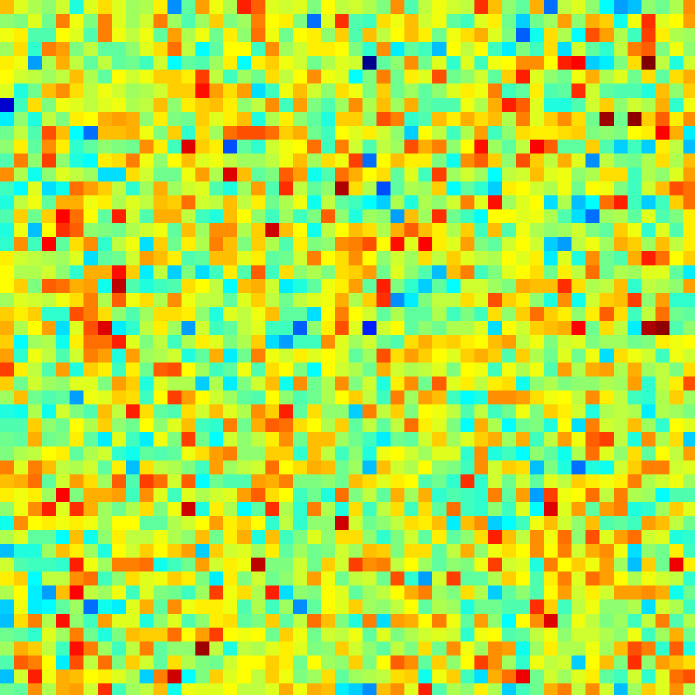}
\label{subfig:filtering_noise}}
$\quad$
\subfloat[Filtered eigenfunction $PSNR=30.2dB$]{
\includegraphics[trim=0 -10 0 -10, clip,width=0.1875\textwidth,valign=c]{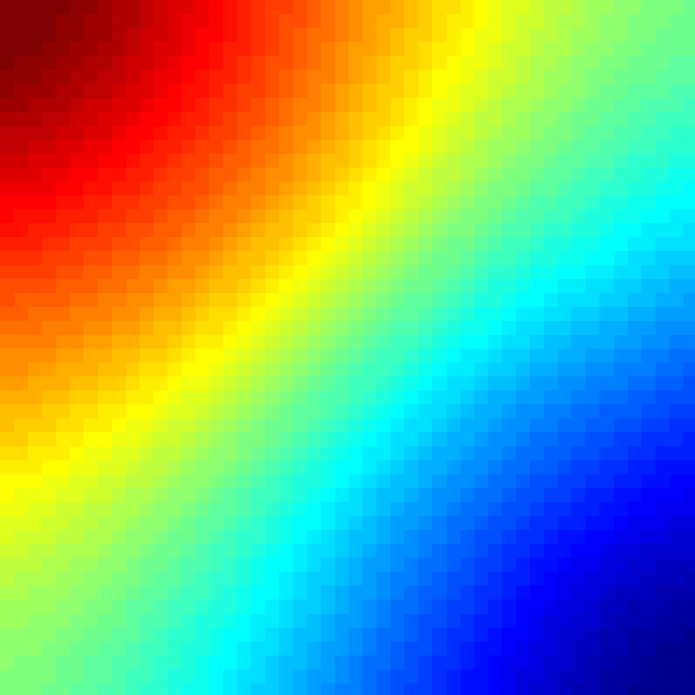}
\label{subfig:filtering_EF}
}
\subfloat[The discrete spectrum]
  {
  \includegraphics[trim=0 80 0 0, clip,width=0.325\textwidth,valign=c]{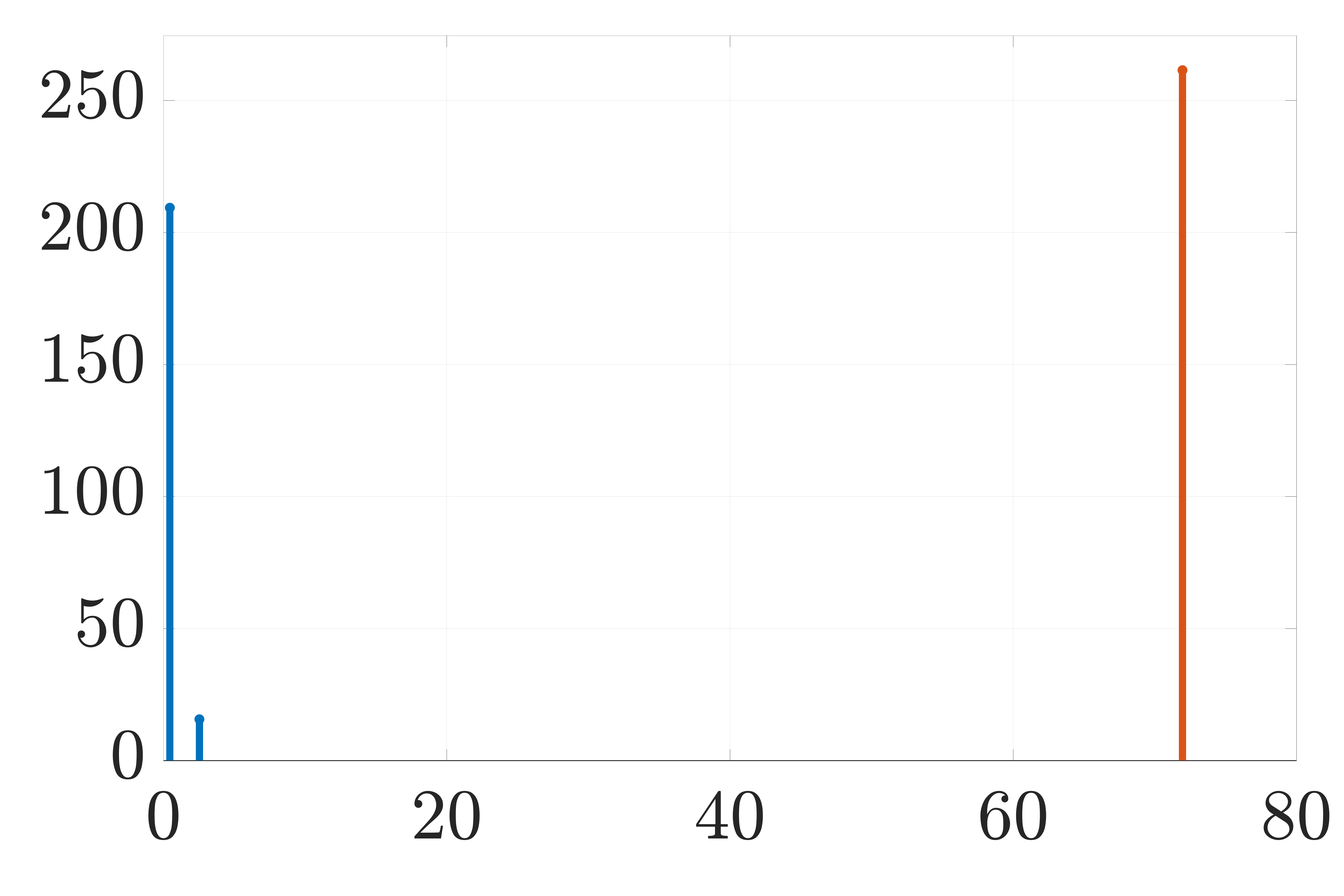}
\label{subfig:filtering_discSpec}
}
\end{minipage}
\begin{minipage}{1\textwidth}
\caption{{\bf{Filtering via \ac{OrthoNS} and Def. \ref{def:dispFil}}} - Recovering an eigenfunction, corrupted with Gaussian noise. See the text for further details.}
\label{Fig:filtering}
\end{minipage}
\end{figure}


\subsubsection{Denoising a natural image} 
In Fig. \ref{Fig:denoising} a natural image with additive white Gaussian noise ($N\sim \mathcal{N}(0,0.2)$, $PSNR = 14dB$) is denoised. 
Filtering is performed using \ac{OrthoNS}, based on the $p$-Laplacian with $p=1.01$. Modes with high eigenvalues (lower extinction time) are filtered out. This yields an edge preserving denoising. As expected, we lose some small details but the zebra's texture in general is preserved. 
  
\begin{figure}[phtb!]
\centering
\captionsetup[subfigure]{justification=centering}
\subfloat[A zebra]
  {
\includegraphics[width=0.2\textwidth,valign=c]{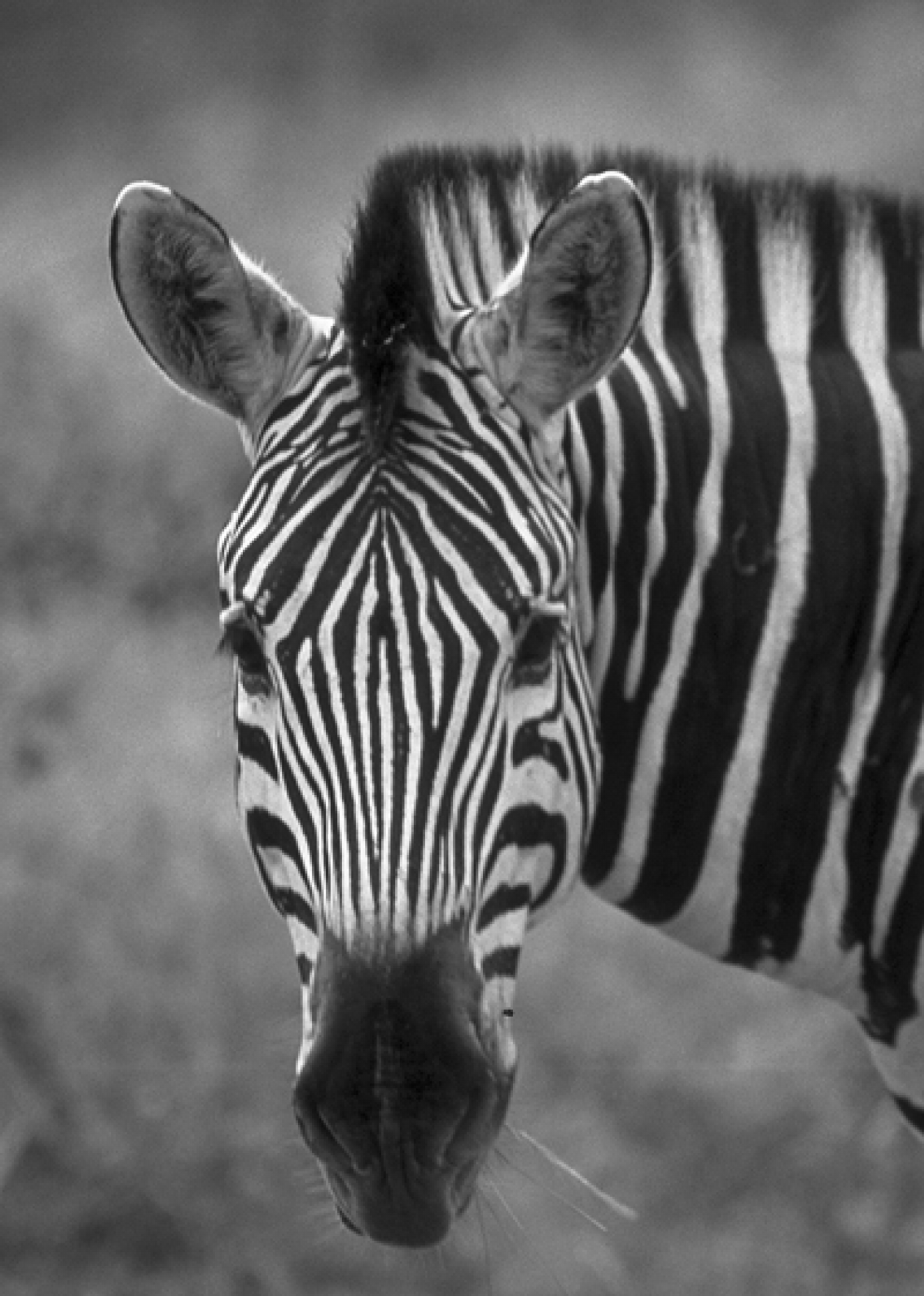}
\label{subfig:fullZebra}
}
\subfloat[A noisy zebra $PSNR = 14dB$]
  {
\includegraphics[width=0.2\textwidth,valign=c]{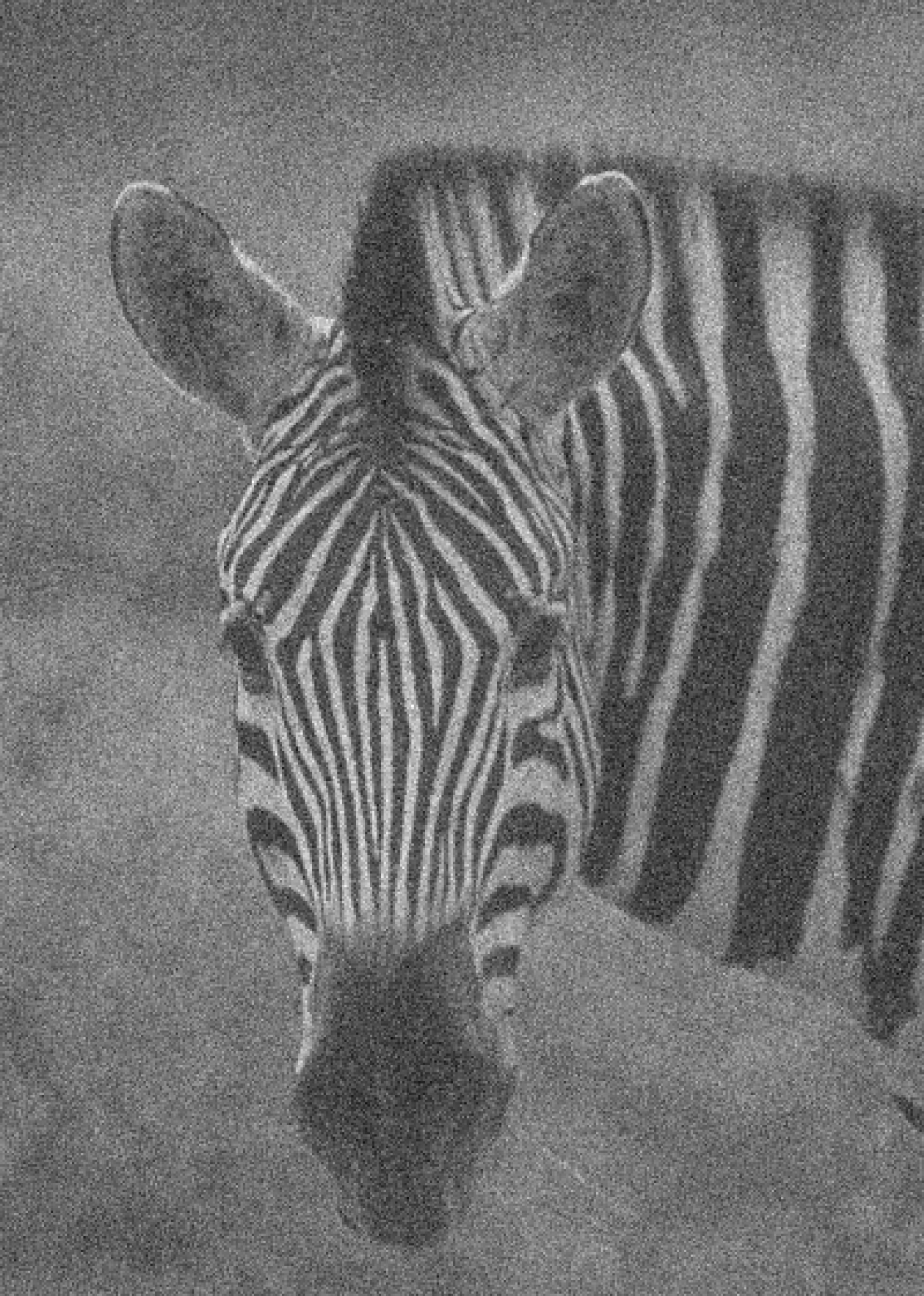}
\label{subfig:fullNoiseZebra}
}
\subfloat[A filtered zebra $PSNR = 22.6dB$]
  {
\includegraphics[width=0.2\textwidth,valign=c]{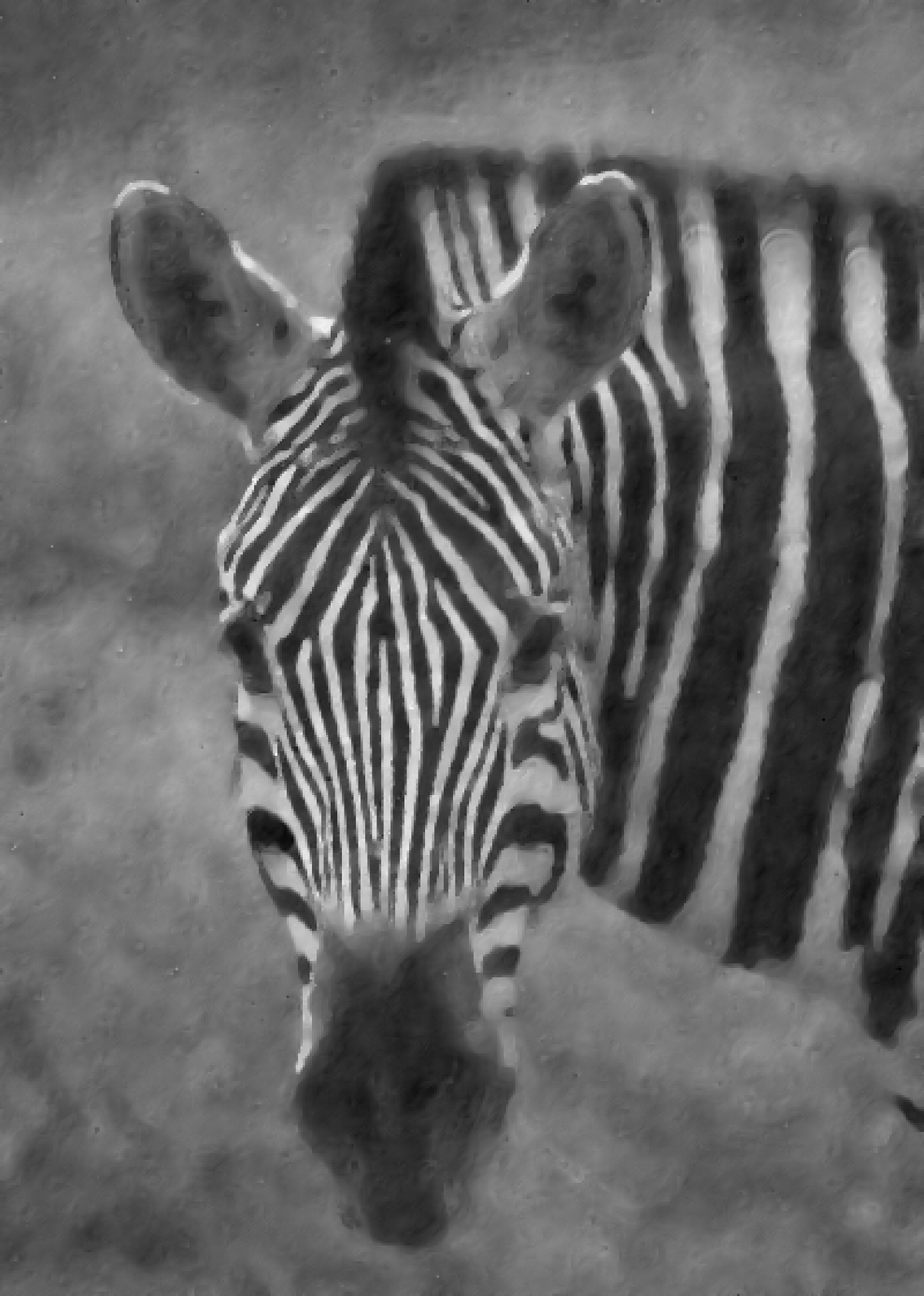}
\label{subfig:filteredZebra}
}
\caption{\emph{Denoiser} - Using the definition of filtering, Def. \ref{def:dispFil}, we filter out the noise.}
\label{Fig:denoising}
\end{figure}

\subsubsection{Spectrum} In Fig. \ref{Fig:ZebraDecompositionDMD_101} we demonstrate a qualitative comparison between the nonlinear spectral $p$-decomposition of \cite{cohen2020Introducing}, the \ac{OrthoNS} decomposition and the posterior \ac{OrthoNS}. We apply these methods on the zebra (Fig. \ref{subfig:fullZebra}) for $p=1.01$. 
On the left column, the spectra of the $p$-decomposition, the \ac{OrthoNS}, and the posterior scheme are given. The images from the second left column to the right represent four different bands in the spectra. One can see that the bands are automatically sorted from fine to coarse spatial structures.
%
\begin{figure}[phtb!]
\centering
\captionsetup[subfigure]{justification=centering}
\subfloat[The $p$-spectrum $S$ vs. time]
  {
\includegraphics[width=0.3\textwidth,height=0.2\textwidth]{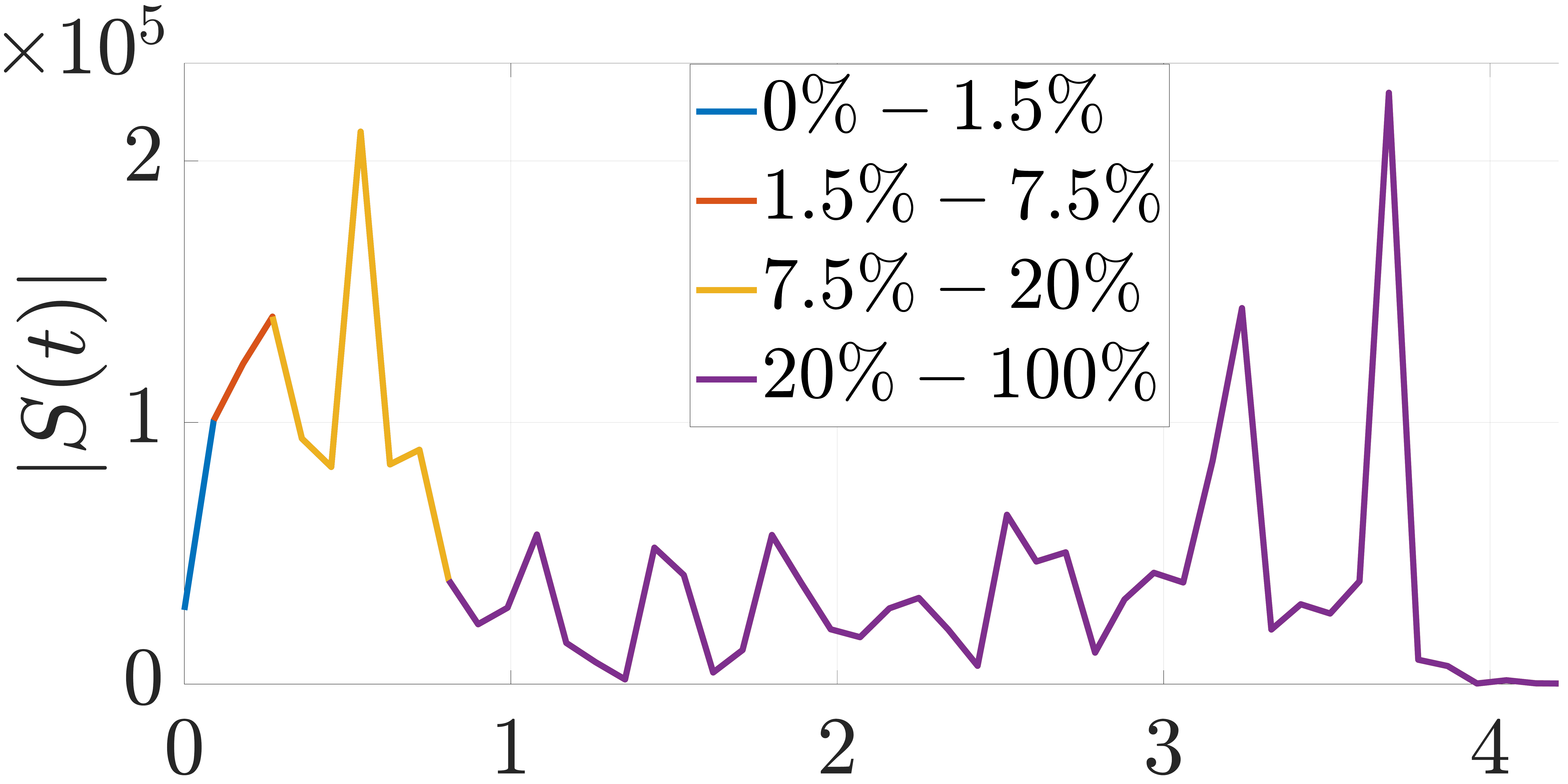}
\label{subfig:specZebra1_01color}
}
\subfloat[$0\%-1.5\%$]
  {
\includegraphics[width=0.15\textwidth]{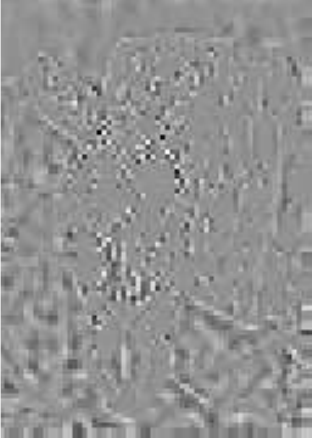}
\label{subfig:z1_101}
}
\subfloat[$1.5\%-7.5\%$]
  {
\includegraphics[width=0.15\textwidth]{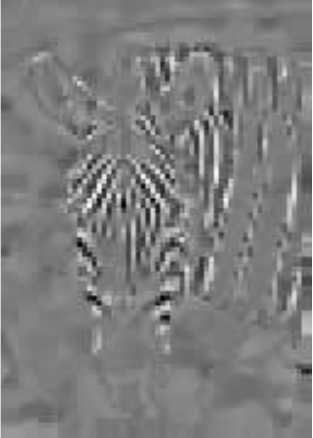}
\label{subfig:z2_101}
}
\subfloat[$7.5\%-20\%$]
  {
\includegraphics[width=0.15\textwidth]{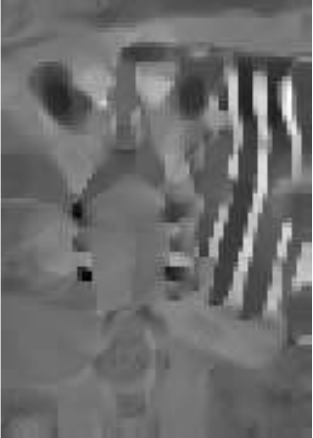}
\label{subfig:z3_101}
}
\subfloat[$20\%-100\%$]
  {
\includegraphics[width=0.15\textwidth]{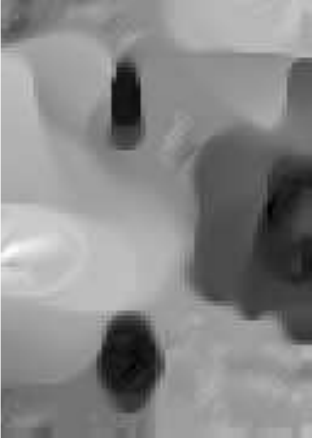}
\label{subfig:z4_101}
}
\label{fig:pDecomZebra101ConDisBHN}

\centering
\captionsetup[subfigure]{justification=centering}
\subfloat[The \ac{OrthoNS}]
  {
\includegraphics[width=0.3\textwidth]{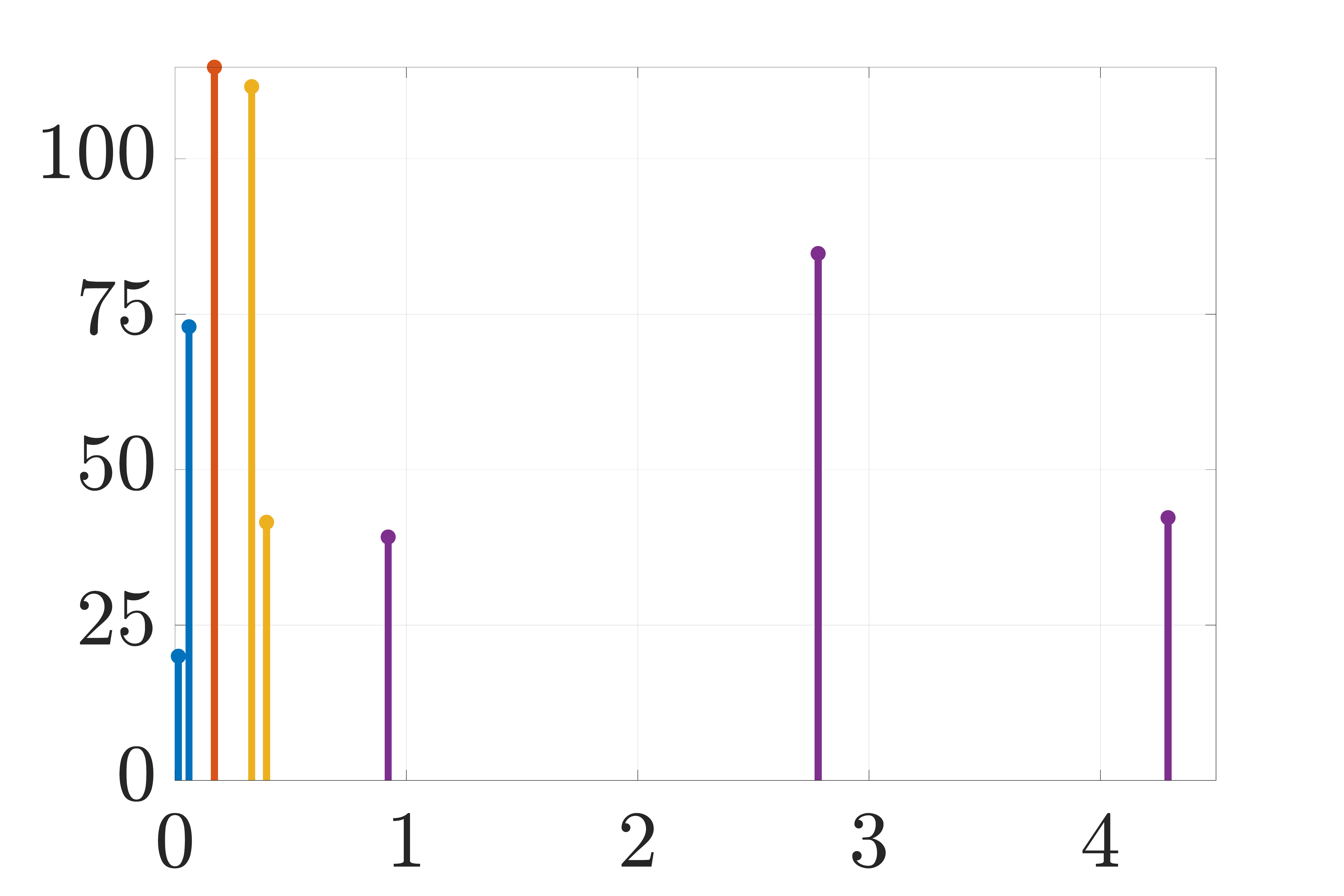}
\label{subfig:specDMDZebra1_01color}
}
\subfloat[Blue items]
  {
\includegraphics[width=0.15\textwidth]{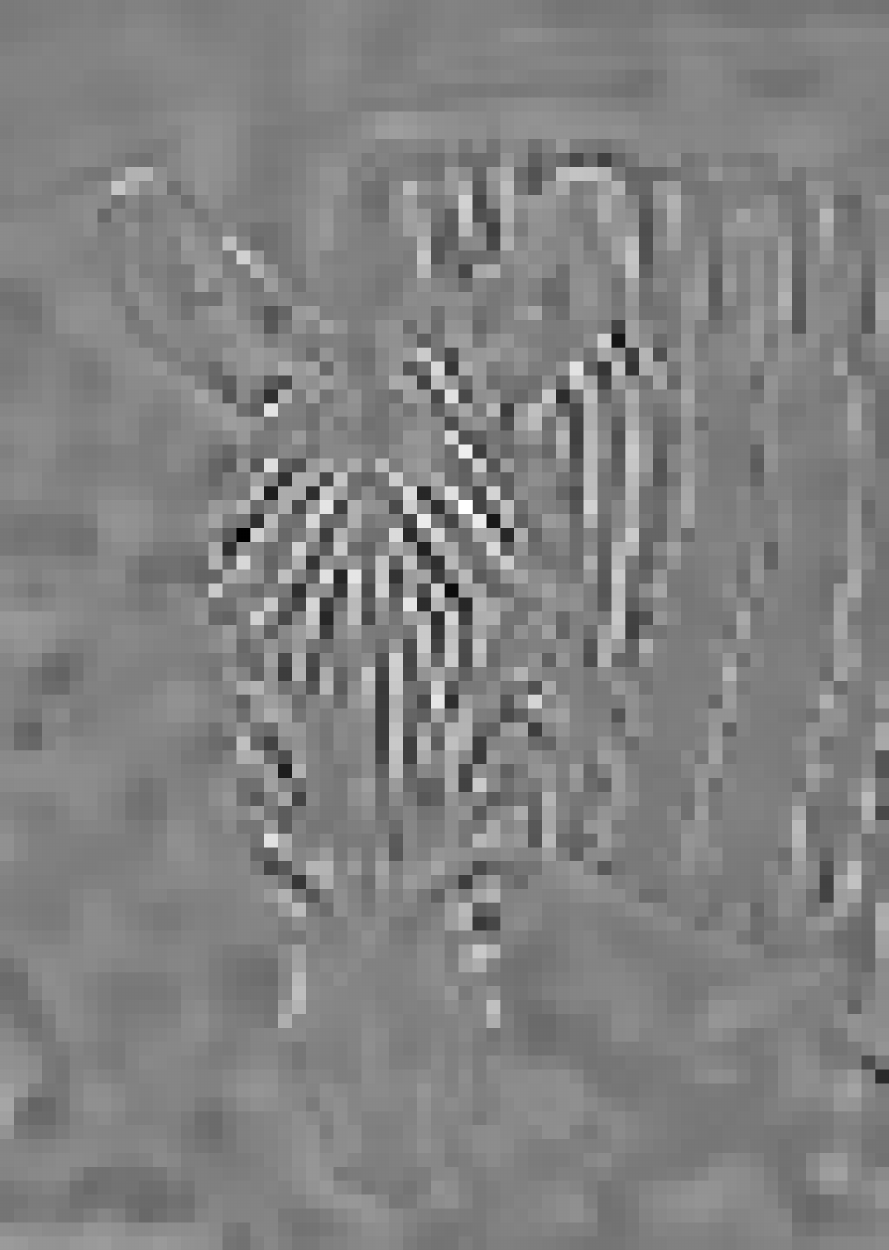}
\label{subfig:z1_DMD_101}
}
\subfloat[Red items]
  {
\includegraphics[width=0.15\textwidth]{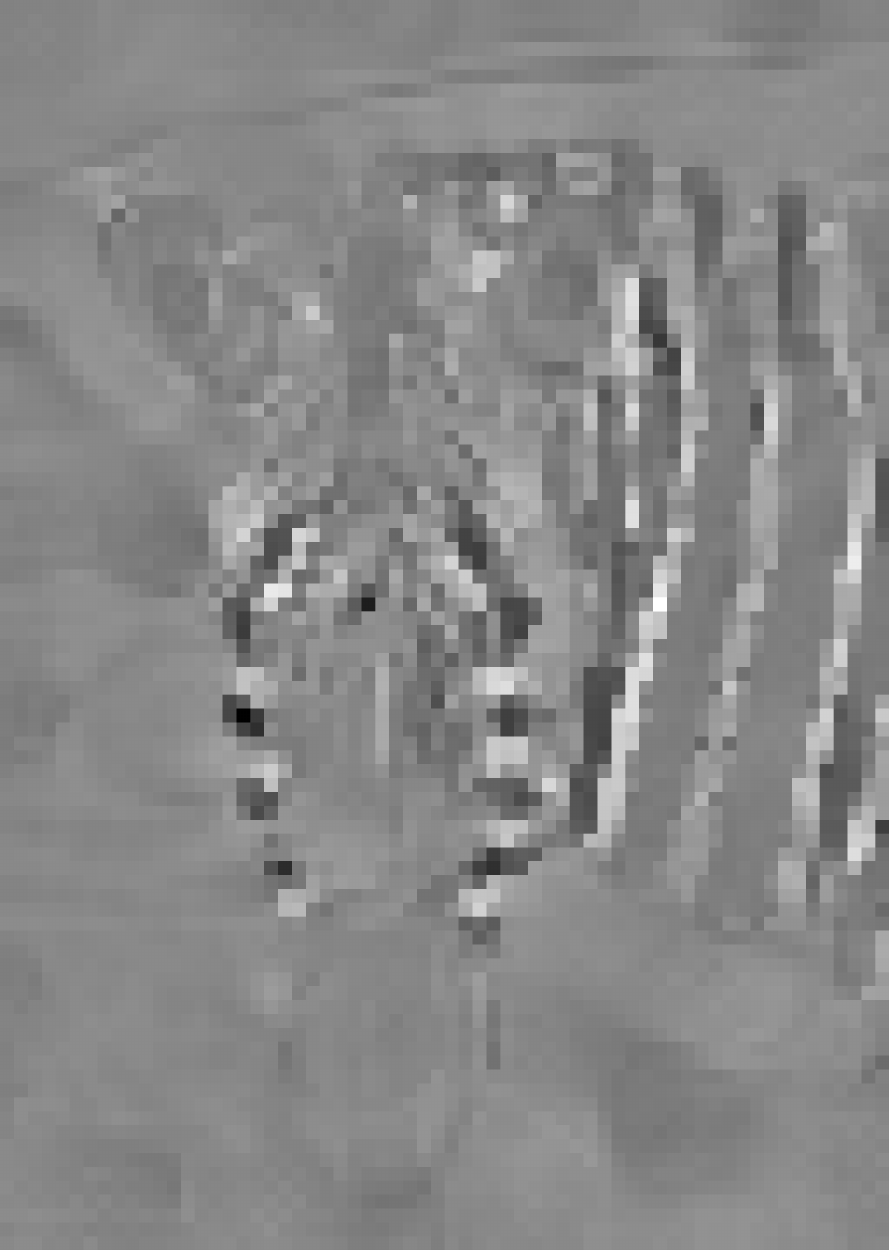}
\label{subfig:z2_DMD_101}
}
\subfloat[Yellow items]
  {
\includegraphics[width=0.15\textwidth]{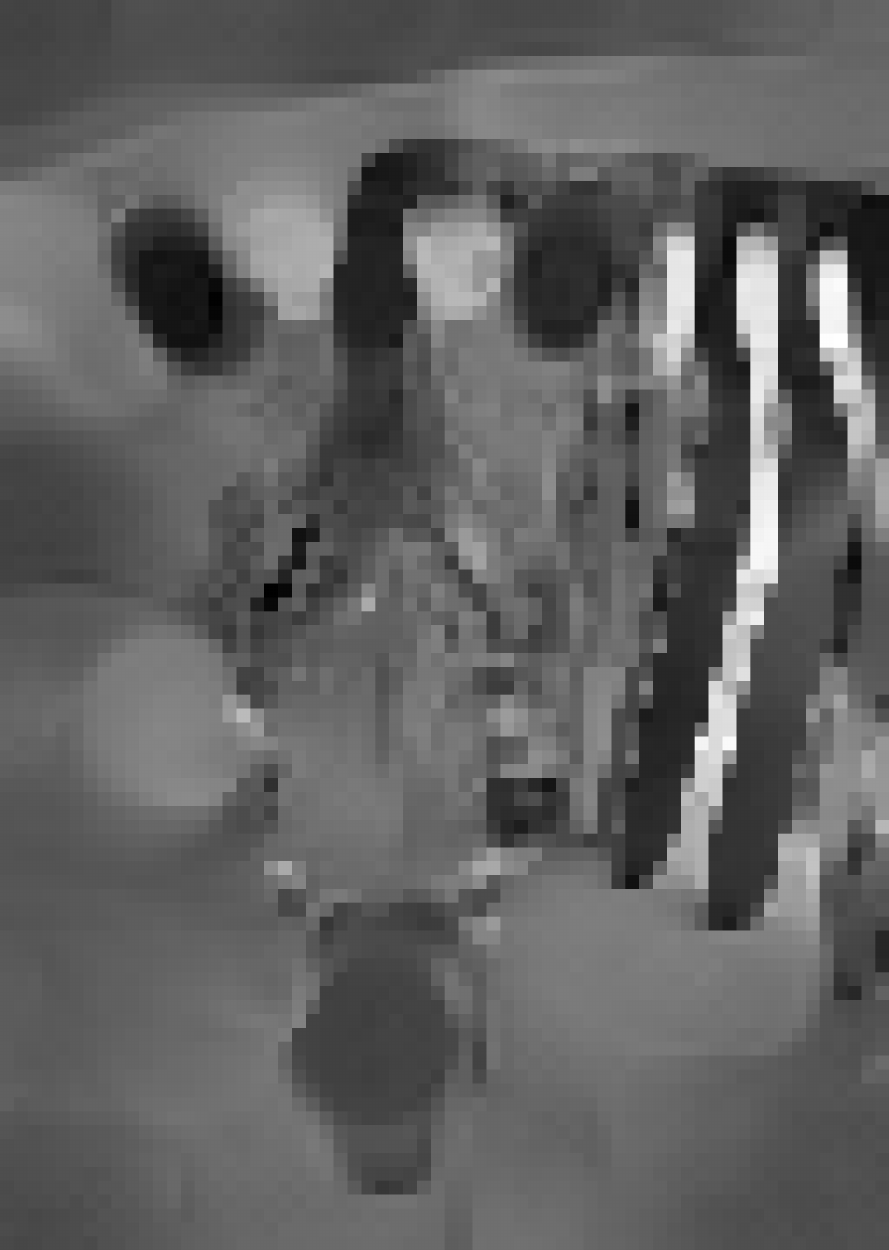}
\label{subfig:z3_DMD_101}
}
\subfloat[Purple items]
  {
\includegraphics[width=0.15\textwidth]{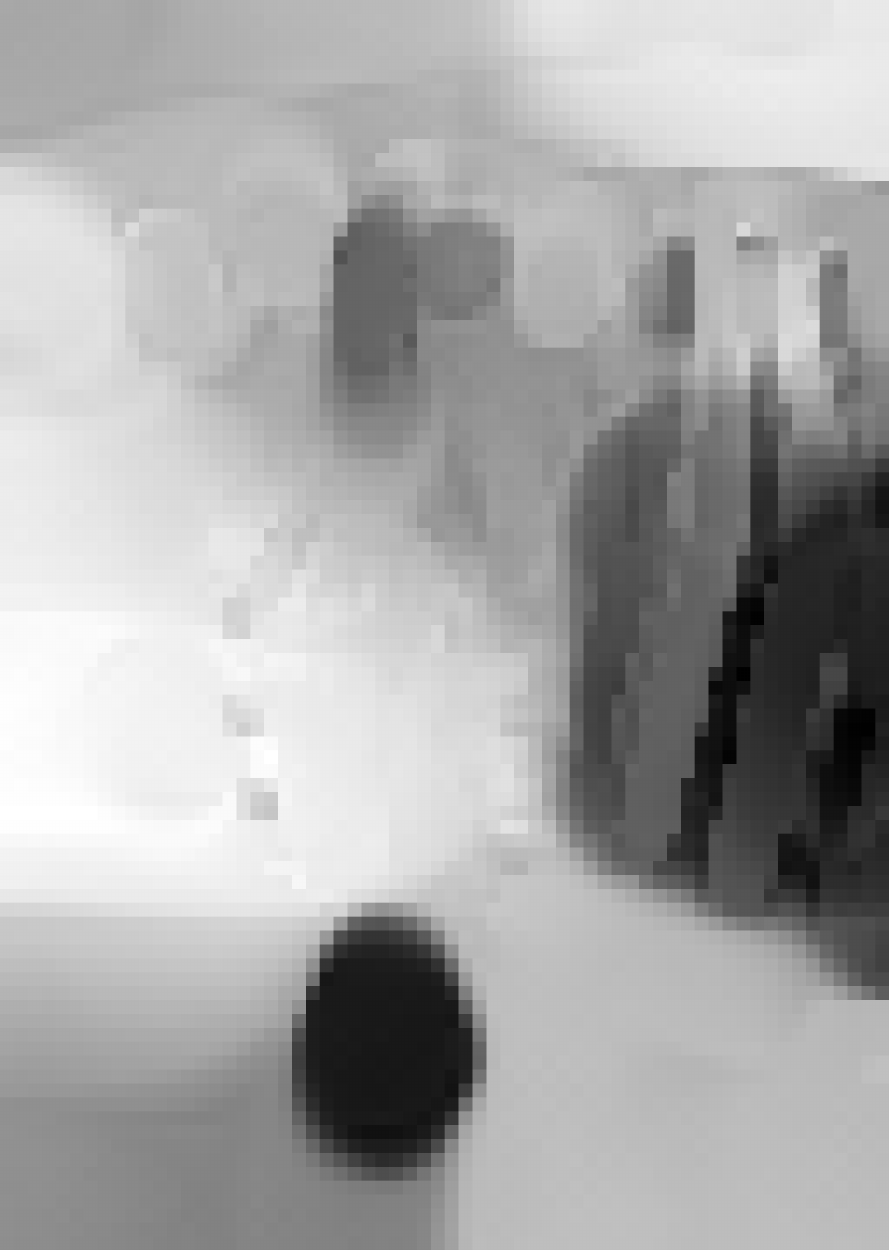}
\label{subfig:z4_DMD_101}
}\\
\subfloat[The posterior \ac{OrthoNS}]
  {
\includegraphics[width=0.3\textwidth]{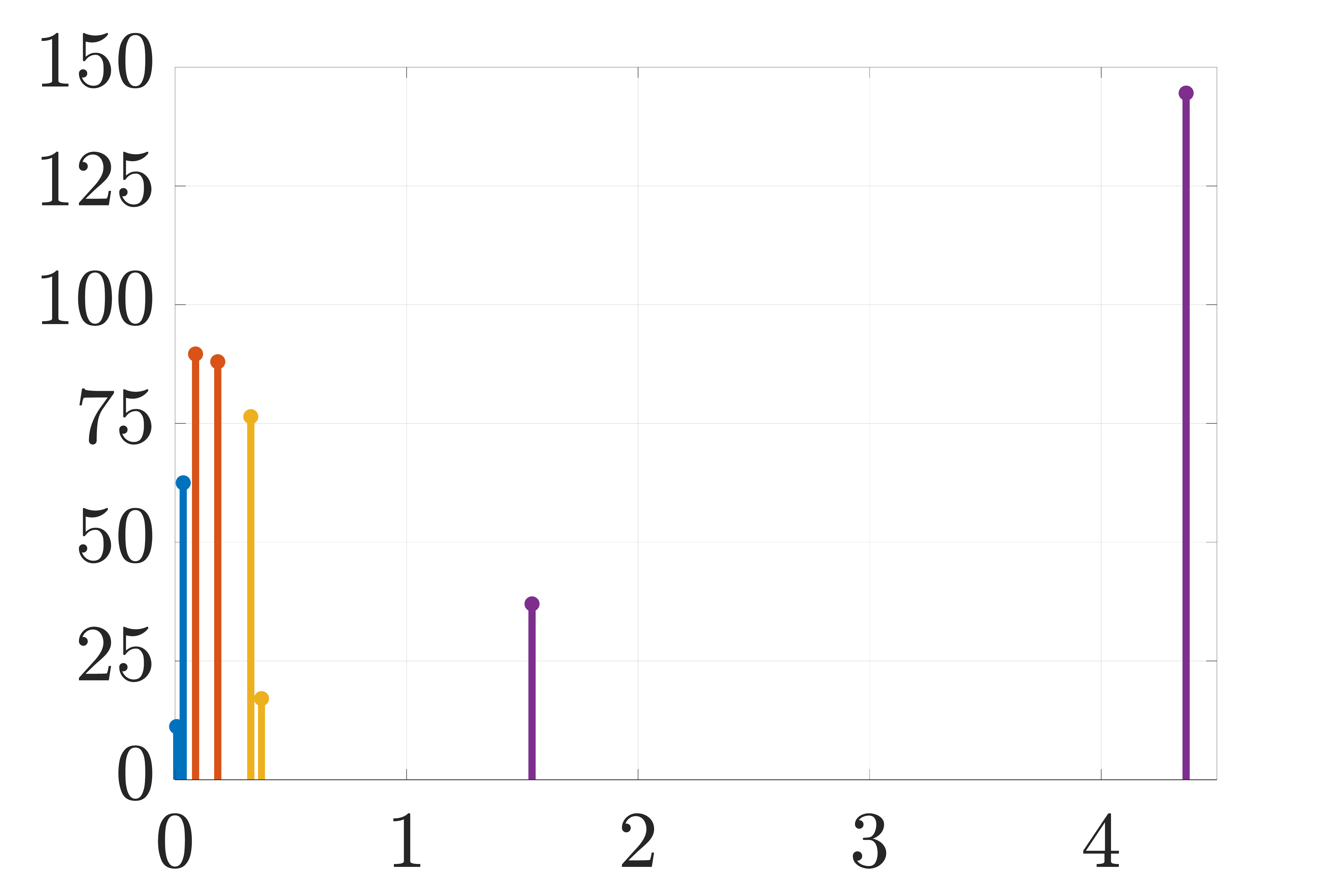}
\label{subfig:specDMDZebra1_01color_blind}
}
\subfloat[Blue items]
  {
\includegraphics[width=0.15\textwidth]{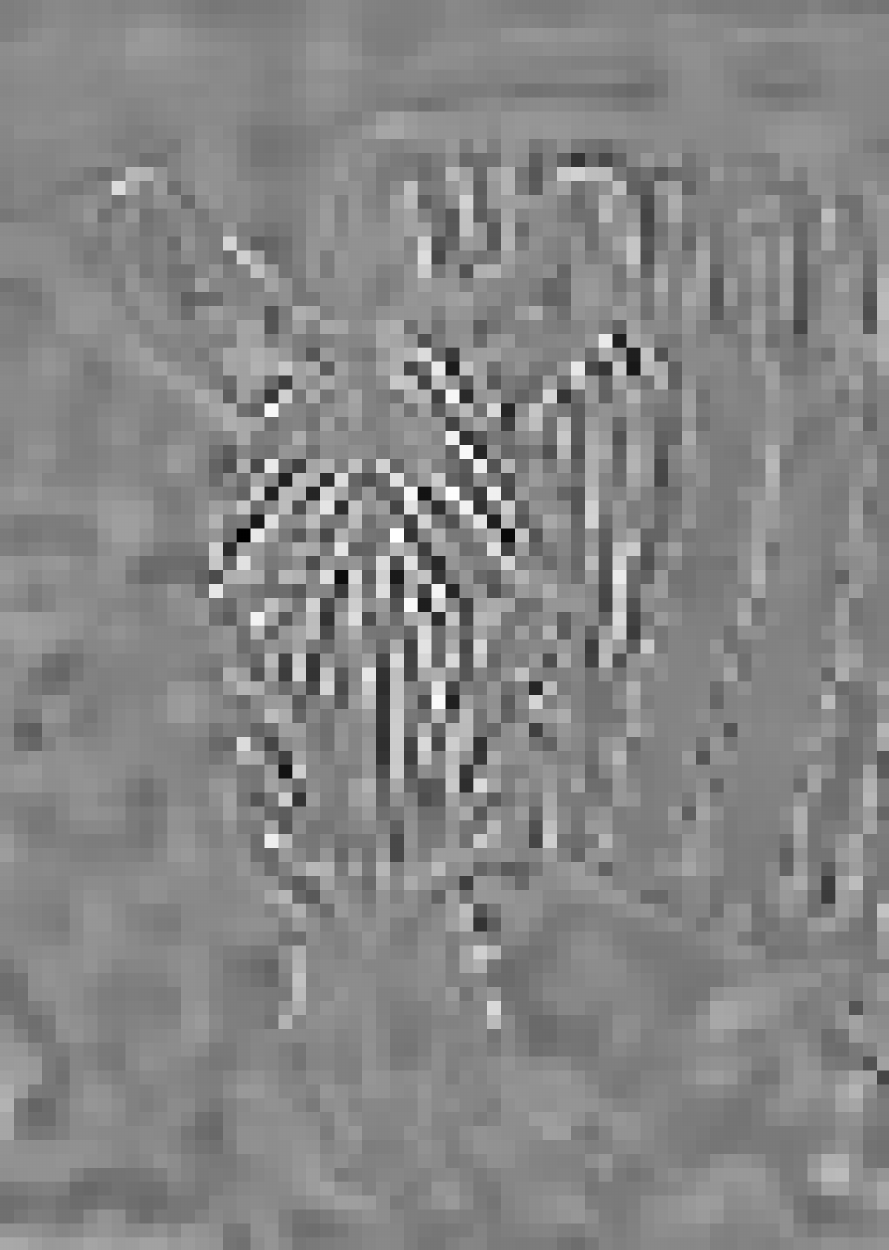}
\label{subfig:z1_DMD_101_blind}
}
\subfloat[Red items]
  {
\includegraphics[width=0.15\textwidth]{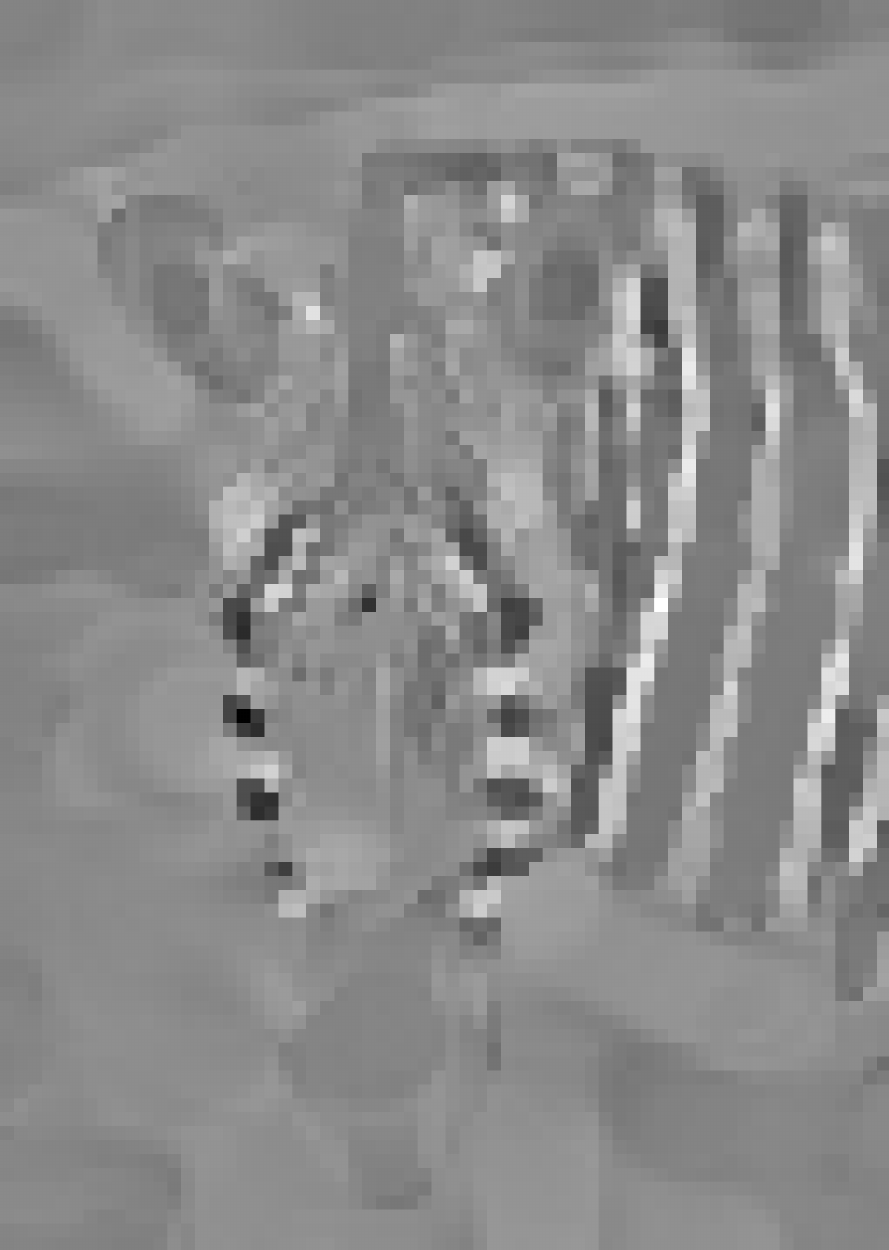}
\label{subfig:z2_DMD_101_blind}
}
\subfloat[Yellow items]
  {
\includegraphics[width=0.15\textwidth]{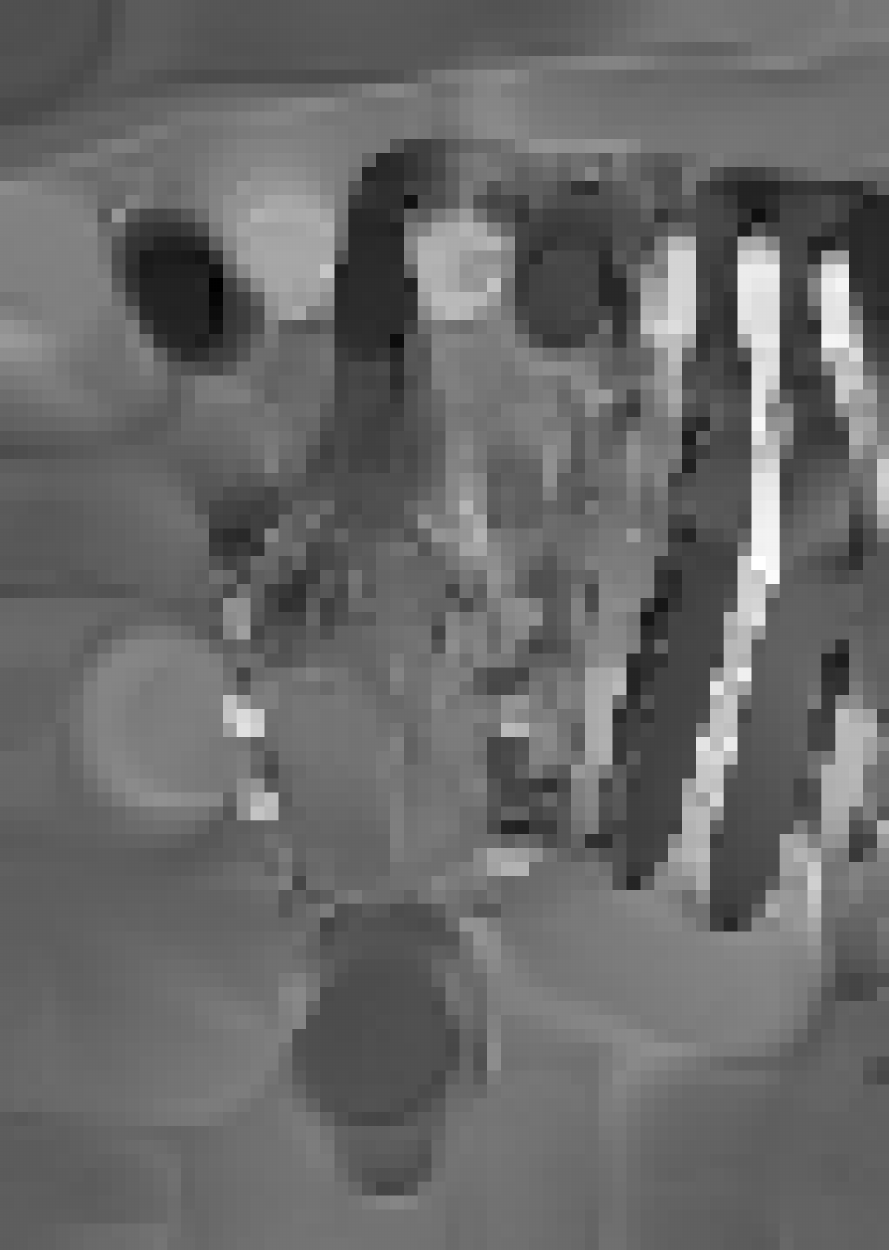}
\label{subfig:z3_DMD_101_blind}
}
\subfloat[Purple items]
  {
\includegraphics[width=0.15\textwidth]{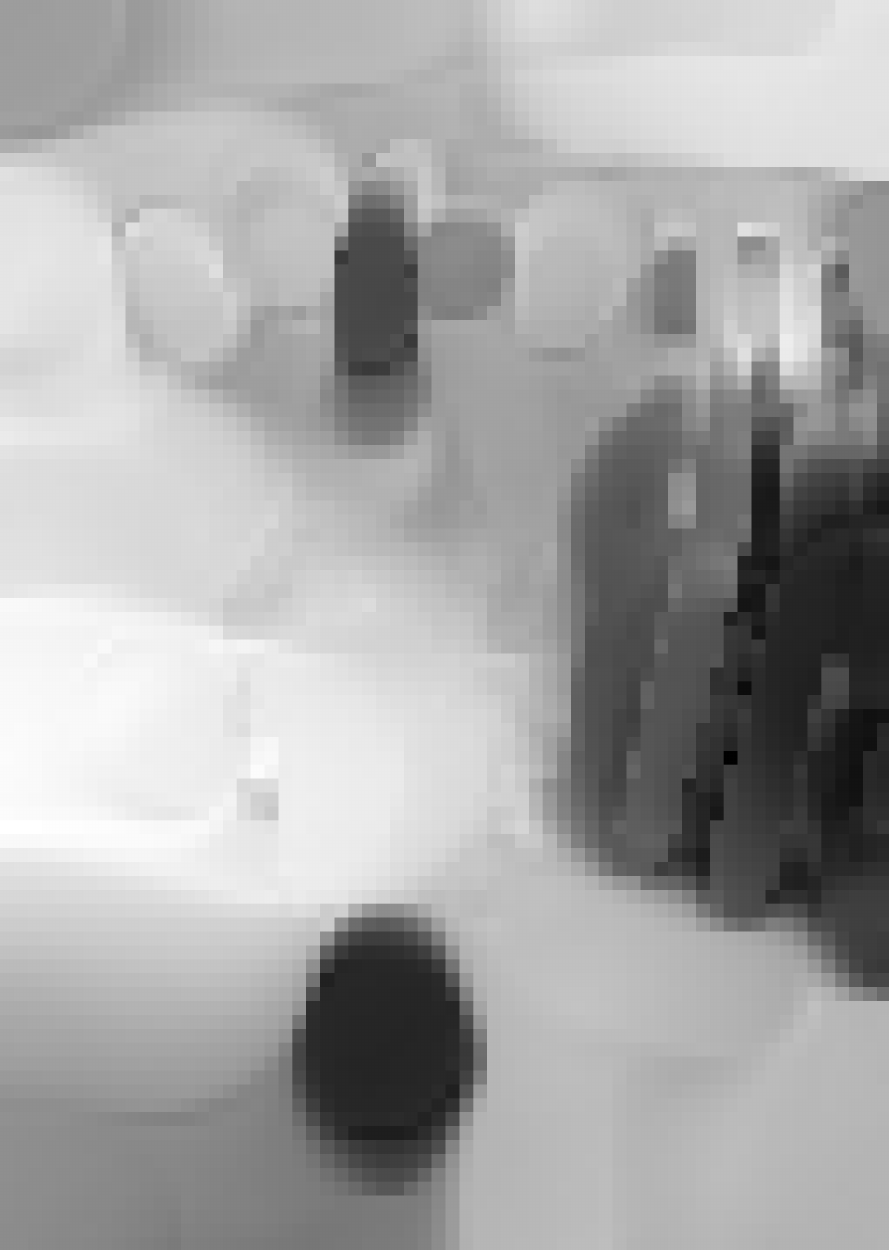}
\label{subfig:z4_DMD_101_blind}
}
\caption{{\bf{\acf{OrthoNS}}} The first row shows an image decomposition with $p=1.01$ from \cite{cohen2020Introducing}. In the second row, \ac{OrthoNS} is applied to the zebra image for $p=1.01$, whereas in the third row, the posterior \ac{OrthoNS} method is used where neither the operator nor the step size are known.}
\label{Fig:ZebraDecompositionDMD_101}
\end{figure}

\subsection{Run time Vs. Image size} \label{subsec:ResRT}
A prominent advantage of this method is the running time. The $p$-decomposition \cite{cohen2020Introducing} requires evaluating the \eqref{eq:pFlow} with a uniform small step size. Then, fractional derivative is calculated pointwise. This involves applying \emph{FFT} and \emph{IFFT} (along the sampling time axis) for every point in the image. In  Fig. \ref{Fig:RVS}, we show the computation time versus the size of the image.
\begin{figure}[phtb!]
\centering 
\includegraphics[trim=0 50 0 0, clip,width=0.5\textwidth]{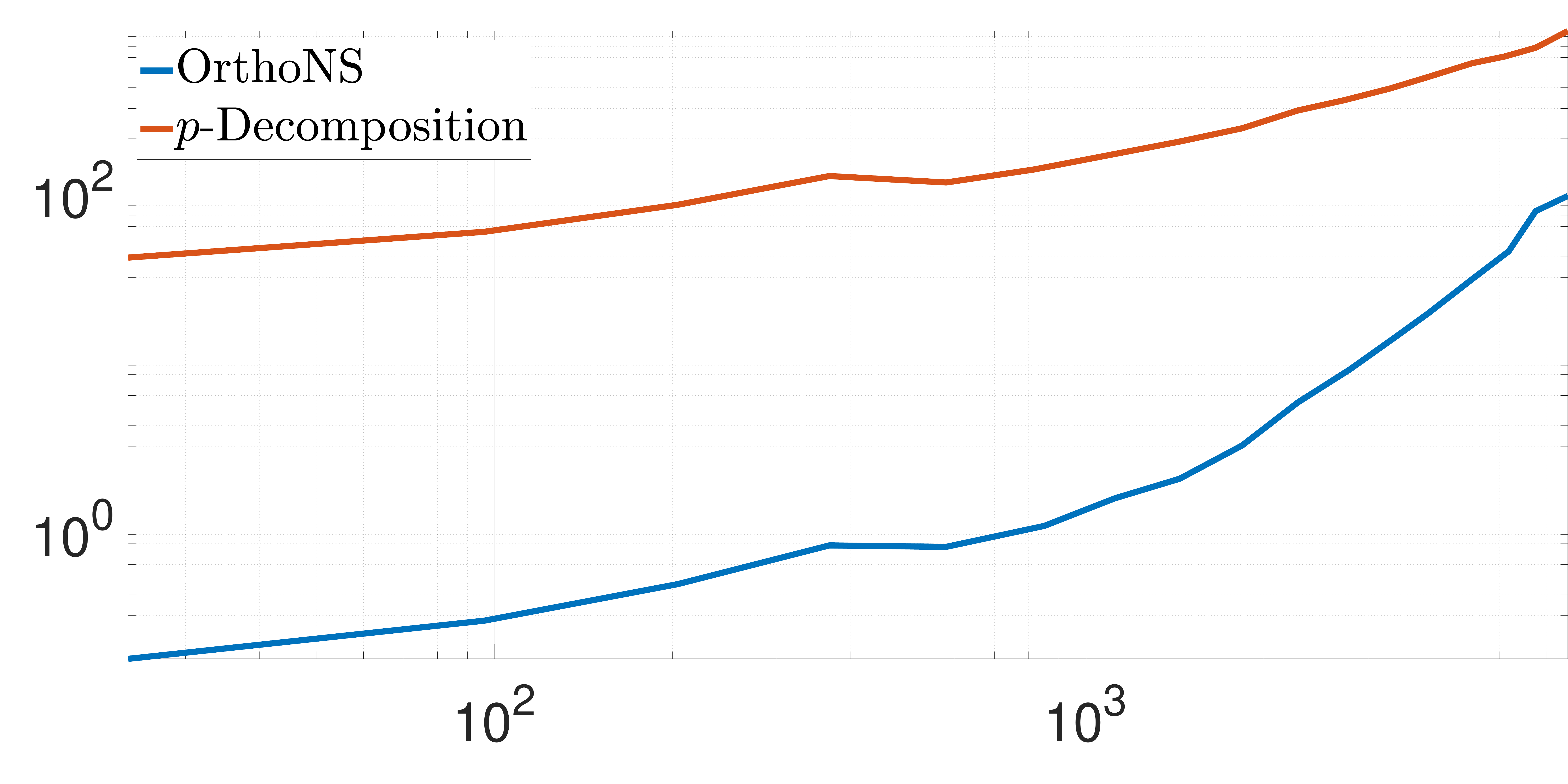}
\caption{Running time Vs. Image size}
\label{Fig:RVS}
\end{figure}
The $X$-axis indicates the size of the image (number of pixels) in log scale and the $Y$-axis indicates the running time taken to compute the decomposition. 
The running time of \ac{OrthoNS} is considerably lower than \cite{cohen2020Introducing} in 1-2  orders of magnitude.

\section{Conclusion and future work}\label{sec:conclusion}
In this work we investigated how to recover the main modes of homogeneous gradient flows through a linear dimensionality reduction algorithm. 
We examined \ac{DMD}, a leading method for this purpose in fluid-dynamics. We used explicit analytic solutions of such flows for cases where the initial condition is an eigenfunction of the nonlinear operator. The analytic solution of \ac{DMD} in these cases clearly shows its inability to express such flows faithfully. 
A significant observation is that \ac{DMD} can recover well homogeneous flows which are of degree one. 
We thus proposed a time re-scaling of the sampling points, such that it mimics the dynamics of 1-homogeneous flows sampled with uniform time steps. It was shown how this adaptation allows to fully recover the dynamics with analytic solutions.

Following these insights, two algorithms were proposed for time re-scaling, also when the original time samples and the operator of the flow are not known. 
Additionally, a different \ac{DMD} optimization was suggested in order to obtain a symmetric \ac{DMD} matrix (for non-oscillating flows).
We have shown that the modes correspond to approximations of nonlinear eigenfunctions (with respect to the operator of the flow). The OrhoNS mode decomposition was proposed. It yields a small set of the main modes of the flow and can be viewed as a linearization of the nonlinear spectral decomposition ($p$-spectra) introduced in \cite{cohen2020Introducing}. We believe this analysis and proposed representation can further advance the understanding of gradient flows and be used in various fields, wherever such flows are relevant.




{\bf{List of Symbols}}
\addcontentsline{toc}{chapter}{List of Symbols}
\begin{longtable}{lp{0.6\textwidth}}
  $P$& A nonlinear homogeneous operator\\
  $f$& An initial condition\\
  $\lambda$ &An eigenvalue of $P$\\
  $a(t)$& A decay profile\\
  $F(\psi,f)$& A fidelity term\\
  $R(\psi)$& A regularization term\\
  $\partial_psi R$& The variational derivative of $R$\\
  $T_{ext}$&The extinction time. The smallest time for which the system gets its steady state\\
  $\nabla$&The gradient of a function\\
  $J_p$&The Dirichlet energy\\
  $\Delta_p$& The $p$-Laplacian operator\\
  $dt,\,dt_k$&A fixed step size, a step size from $k-1$th sample to the $k$th\\
  $f$& Belongs to $\mathbb{R}^M$ (column vector), the initial condition of the dynamical system \\
  $\psi_k$ & The snapshot of the system after $k$th step in $\mathbb{R}^M$ (resulted from sampling or evolving a explicit scheme)\\
  $\Psi_0^{N-1},\Psi_1^N$& Data matrices $[\psi_0,\cdots,\psi_{N-1}],[\psi_1,\cdots,\psi_{N}]$\\
  $U,\Sigma, V$& \acf{SVD} of $\Psi_0^{N-1}$\\
  $U_r, V_r$& Sub-matrices of $U, V$ containing the first $r$ columns\\
  $\Sigma_r$&Sub-matrix of $\Sigma$ containing the most significant $r$ eigenvalues of the \ac{SVD} which are the diagonal of $\Sigma$\\
  $X,Y$&Dimensional reduced matrices of $\Psi_0^{N-1},\Psi_1^N$, respectively\\
  $F$ & The \ac{DMD} matrix ,approximating a linear mapping from $X$ to $Y$ (size $r\times r$)\\
  $w_i$& A column vector, the $i$th right eigenvectors of $F$\\
  $W$& $W=[{\bm{w_1}},\cdots,{\bm{w_r}}]$\\
  $\{\phi_i,\mu_i,\alpha_i\}_{i=1}^r$& Modes, eigenvalues, and coefficients resulted form \acf{DMD}\\
  $D$& $D=diag([\mu_1,\cdots,\mu_r])$\\
  ${\bm{\tilde{\psi}_k}}$&Data reconstruction by \ac{DMD} (discrete time setting)\\
  $A$ & $M\times M$ matrix, approximating a linear mapping from $\Psi_0^{N-1}$ to $\Psi_{1}^N$ \\
  ${\bm{\tilde{\psi(t)}}}$&Data reconstruction by \ac{DMD} (continuous time setting)\\
  $ERR_{DMD}$, $ERR_{Rec}^d$, $ERR_{Rec}^c$&The \ac{DMD}, the (time-discrete) and the (time-continuous) reconstruction errors\\
  $\{\tilde{\mu}\}_{i=1}^{r}$&Eigenvalues in the time continuous setting\\
  $\tilde{dt}_k,\,\tilde{t}_k$&Rescaled step size and time point\\
  $\lambda_\phi,\,\lambda_\mu$& A nonlinear eigenvalue restoration via the mode $\phi$, the eigenvalue $\mu$
  \end{longtable}
\appendix

\section{Finding a symmetric \ac{DMD} matrix}\label{sec:LRUC}
We are looking for linear mapping, $F$, between $X$ and $Y$ when the mapping is symmetric, i.e.
\begin{equation}\label{eq:AppConstDMD}
    \min_{F}\norm{Y-FX}^2_F, \quad s.t.\,\, F=F^T.
\end{equation}
In addition, according to the spectral theorem every symmetric real matrix can be diagonalized. Therefore, we can express the matrix as $F=Q^TDQ$ when $Q$ and $D$ are orthogonal and diagonal matrices. We can rewrite this expression as $F=Q^T\sqrt{D}^T\sqrt{D}Q$. Then, we can reformulate the optimization problem as
\begin{equation}\label{eq:AppConstDMD2}
    \min_{F}\norm{Y-FX}^2_F, \quad s.t.\,\, F=B^TB.
\end{equation}
Note, that $B$ is over the complex field and $^T$ denotes for the transpose 
operator. Embedding the constrain in the optimization expression,  we get
\begin{equation*}
    \min_{B}\norm{Y-B^TBX}^2_F.
\end{equation*}
Using $\norm{Y-B^TBX}^2_F=\Tr{Y-B^TBX}^T\{Y-B^TBX\}$ and the derivatives 
\begin{equation*}
\begin{split}
    \frac{\partial}{\partial B}\Tr{FBG}&=F^TG^T\\ \frac{\partial}{\partial B}\Tr{FB^TG}&=GF,
\end{split}
\end{equation*}
we get that the minimizer, $B$, admits
\begin{equation*}
    B^TBXX^T+XX^TB^TB = XY^T+YX^T.
\end{equation*}
Substituting $B^TB$ with $F$, we get that for the minimizer, $F$, of \eqref{eq:AppConstDMD} the following Sylvester equation holds
\begin{equation}\label{eq:Asylvester}
    FXX^T+XX^TF = XY^T+YX^T.
\end{equation}
The solution for $F$ exists and unique in this case. There are plenty of algorithm to solve this equation (see e.g.
\cite{jameson1968solution}\cite{tongxing1986solution}\cite{bartels1972solution}). In addition, for the specific form of the Sylvester equation \eqref{eq:Asylvester} a farther study was conducted in \cite{la1961stability}\cite{smith1966matrix}\cite{barnett1967analysis}. We use the Matlab implementation (command ``sylvester'') for solving the Sylvester equation, which is based on the algorithm of Hessenberg-Schur method. The implementation is based on the routines SB04QD. 

\section{\acf{SDMD}}\label{sec:SDMDResults}
\begin{figure}[phtb!]
\centering
\captionsetup[subfigure]{justification=centering}
\subfloat[$\lambda = -0.5$]
  {
\includegraphics[width=0.95\textwidth]{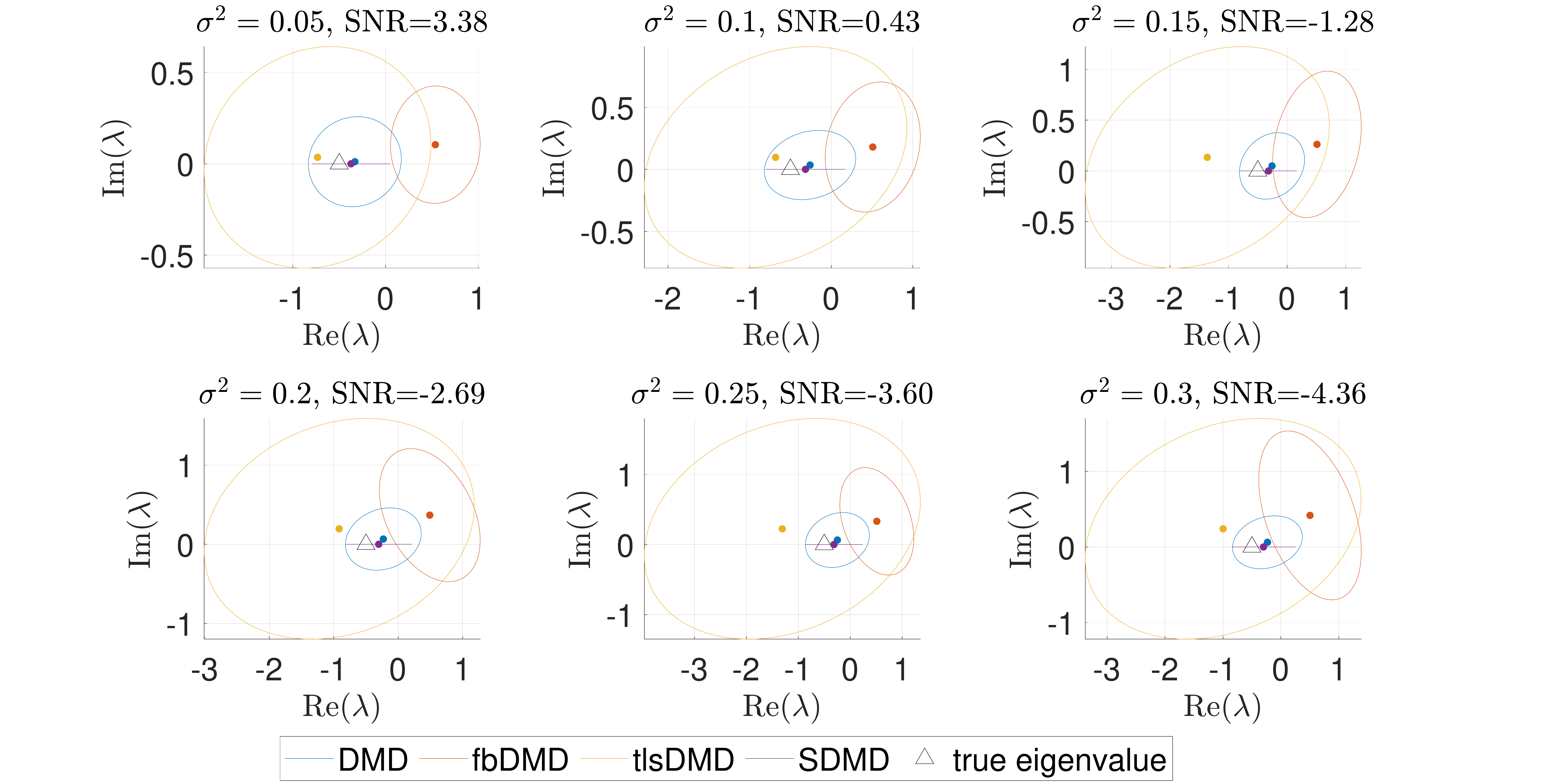}
\label{subfig:sys1_root1}
}\\
\subfloat[$\lambda = 0.7$]
  {
\includegraphics[width=0.95\textwidth]{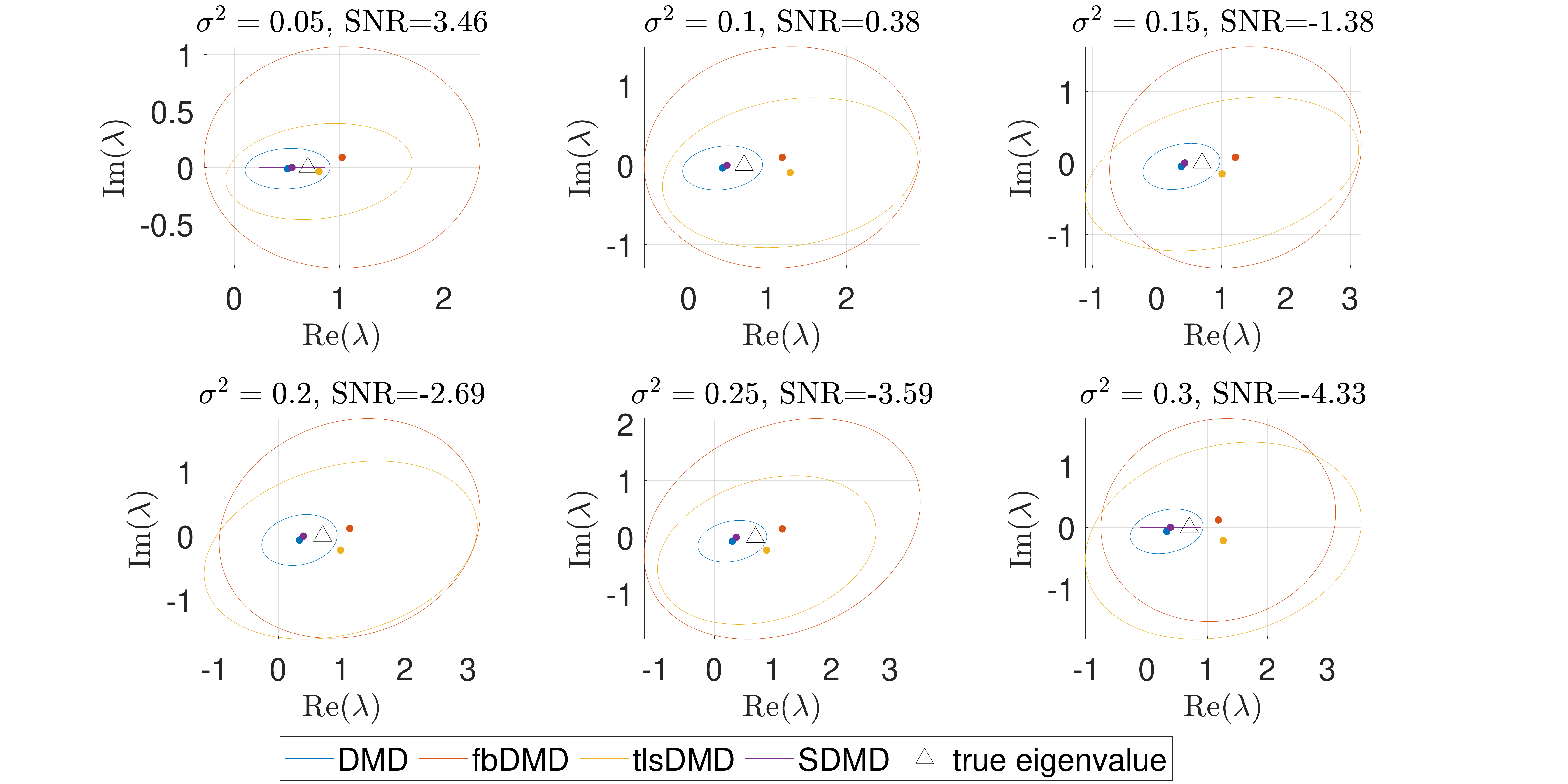}
\label{subfig:sys1_root2}
}
\caption{{\bf{Spectrum Reconstruction.}} We compare \ac{DMD} \cite{schmid2010dynamic}, tls\ac{DMD} \cite{hemati2017biasing}, fb\ac{DMD} \cite{dawson2016characterizing} and \ac{SDMD} based on their approximation for the eigenvalue of system \eqref{eq:linearSys} when various levels of noise are introduced, $-4 \leq SNR \leq 4$.}
\label{Fig:sys1}
\end{figure}

Here, we implement the \ac{SDMD} on a discrete stable linear system 
\begin{equation}\label{eq:linearSys}
    \psi_{k+1}=\begin{bmatrix}0.1&0.6\\0.6&0.1
    \end{bmatrix}\psi_k.
\end{equation}
The eigenvalues are $\lambda_{1,2}=-0.5,0.7$ and the initial condition is a normalized summation of the eigenvectors (namely $[1,0]^T$). We approximate the eigenvalues of this system based on $8$ snapshots in presence of white Gaussian noise. We repeat our experiment $N = 1000$
times and average of each of the methods. In Fig. \ref{Fig:sys1} we showcase the results and plot the ellipses which enclose the region of $95\%$ of the estimates that are closest to the true eigenvalue for each of the techniques (see \cite{dawson2016characterizing}). One can see that for this kind of systems, \ac{DMD} is superior on the tls\ac{DMD} and the fb\ac{DMD}. In particularly, a method that takes into account the system and its inverse is doom to fail for every stable system since the inverse system is not stable. Therefore, whilst the fb\ac{DMD} has good performances when the roots are on the unit cycle (BIBO stability) it fails when the roots are in the unit cycle. 

\section{Proof of Theorem \ref{theo:decayAndConvergence}}\label{sec:proofTheoConv}
\begin{proof}$\,$
\begin{enumerate}
    \item The functional $R(\psi)$ is convex, therefore
    \begin{equation*}
        R(\psi)-R(0)\le -\inp{P(\psi)}{\psi-0}.
    \end{equation*}
    The functional $R$ is zero at the point $0$ (it is assumed to be in its kernel). And thus
    \begin{equation*}
        R(\psi)\le -\inp{P(\psi)}{\psi}.
    \end{equation*}
    Applying the Brezis chain rule \cite{brezis1973ope}, we can write
    \begin{equation*}
        \begin{split}
            \frac{d}{dt}R(\psi)=\inp{-P(\psi)}{\psi_t}
            =\inp{-P( \psi)}{-\frac{\inp{P(\psi)}{\psi}}{\norm{P(\psi)}^2}P(\psi)}
            =\inp{P(\psi)}{\psi}\le -R(\psi).
        \end{split}
    \end{equation*}
    Using the Gr\"onwall's inequality, we can write $R(\psi(t))\le R(\psi(0))\cdot e^{-t}=R(f)\cdot e^{-t}$. Therefore, it converges.
    \item Let the initial condition, $f$, be an eigenfunction of $P$ with a corresponding eigenvalue $\lambda\ne 0$ ($f$ is not trivial). Then, the initial condition is an eigenfunction of the operator $G$ with the corresponding eigenvalue $-1$
    \begin{equation*}
        G(f)=-\frac{\inp{P(f)}{f}}{\norm{P(f)}^2}P(f)=-\frac{\lambda}{\lambda^2}\frac{\inp{f}{f}}{\norm{f}^2}\lambda f=-f.
    \end{equation*}
    In addition, the operator $G(\cdot)$ from Eq. \eqref{eq:NHF} is a one-homogeneous operator. Then, the solution is \cite{cohen2020Introducing} $\psi(t)=f\cdot e^{-t}$.
\end{enumerate}
$\quad$
\end{proof}

\bibliographystyle{unsrt}
\bibliography{smartPeople.bib}

\end{document}